\documentclass[12pt,reqno,oneside]{amsart}
\usepackage[latin1]{inputenc}
\usepackage[T1]{fontenc}
\usepackage{amsmath,amsfonts,amscd,amssymb,amsthm}
\usepackage{mathtools}
\usepackage{dsfont}
\usepackage{color}
\usepackage[mathscr]{euscript}
\usepackage{comment}
\usepackage{graphicx}
\newcommand{\cF}{\mathcal{F}}
\usepackage{placeins}
\usepackage{extsizes}
\usepackage{float}
\usepackage{multirow}
\usepackage{pdflscape}
\usepackage{caption}
\usepackage{pgfplots}
\usepackage{subcaption}
\usepackage{mathabx}
\usepackage{tikz}
\usepackage{hyperref}
\usepackage[dvipsnames]{xcolor}
\usetikzlibrary{decorations.pathreplacing}
\usetikzlibrary{backgrounds} 
\usetikzlibrary{calc}
\pgfdeclarelayer{bg}    
\pgfsetlayers{bg,main}
\usepackage{rotate}
\usepackage[noend]{algorithmic}
\usepackage{algorithm}
\usepackage{setspace}
\definecolor{color1bg}{HTML}{f73d28}
\definecolor{color2bg}{HTML}{FA8072}
\definecolor{bblue}{HTML}{00BFFF}
\definecolor{bblue2}{HTML}{00ffff}

\usetikzlibrary{arrows,calc}
\tikzset{
	>=stealth',
	help lines/.style={dashed, thick},
	axis/.style={<->},
	important line/.style={thick},
	connection/.style={thick, dotted},
}

\usetikzlibrary{shadows}
\usetikzlibrary{backgrounds}
\tikzset{
	diagonal fill/.style 2 args={fill=#2, path picture={
			\fill[#1, sharp corners] (path picture bounding box.south west) -|
			(path picture bounding box.north east) -- cycle;}},
	reversed diagonal fill/.style 2 args={fill=#2, path picture={
			\fill[#1, sharp corners] (path picture bounding box.north west) |- 
			(path picture bounding box.south east) -- cycle;}}
}
\usetikzlibrary{arrows.meta}

\makeatletter
\@addtoreset{equation}{section}
\makeatother
\newcounter{as}[section]

\newcommand{\om}{\omega}
  
\newcommand{\eps}{\epsilon}

\title{Metastability of random maps: a resolvent approach}

\author{Diego Marcondes \and Sandro Vaienti\\ $\;$ \\\today}
\address{Mathematical Sciences Institute and France-Australia Mathematical Sciences and Interactions ANU-CNRS International Research Lab, The Australian National University}
\email{\href{diego.marcondes@anu.edu.au}{diego.marcondes@anu.edu.au} }
\address{Universit\'e de Toulon,  Aix Marseille Universit\'e,  CNRS, CPT, 13009 Marseille, France}
\email{\href{vaienti@cpt.univ-mrs.fr}{vaienti@cpt.univ-mrs.fr} }
\allowdisplaybreaks
\newtheorem{theorem}{Theorem}[section]
\newtheorem{remark}[theorem]{Remark}

\newtheorem{definition}[theorem]{Definition}

\newtheorem{corollary}[theorem]{Corollary}
\newtheorem{lemma}[theorem]{Lemma}
\newtheorem{proposition}[theorem]{Proposition}

\newcommand{\mc}[1]{{\mathcal #1}}
\newcommand{\mf}[1]{{\mathfrak #1}}
\newcommand{\mb}[1]{{\mathbf #1}}
\newcommand{\bb}[1]{{\mathbb #1}}
\newcommand{\bs}[1]{{\boldsymbol #1}}
\newcommand{\ms}[1]{{\mathscr #1}}

\newcommand{\mfB}{\mathfrak{B}}
\newcommand{\mfC}{\mathfrak{C}}
\newcommand{\mfD}{\mathfrak{D}}

\newcommand{\mfF}{\mathfrak{F}}
\newcommand{\mfM}{\mathfrak{M}}
\newcommand{\mfR}{\mathfrak{R}}

\newcommand{\mcD}{\mathcal{D}}
\newcommand{\mcE}{\mathcal{E}}
\newcommand{\mcF}{\mathcal{F}}

\newcommand{\mcL}{\mathcal{L}}

\newcommand{\mcV}{\mathcal{V}}

\newcommand{\mbR}{\mathbb{R}}

\newcommand{\mbM}{\mathbb{M}}
\newcommand{\mbN}{\mathbb{N}}
\newcommand{\mbP}{\mathbb{P}}
\newcommand{\mbQ}{\mathbb{Q}}

\newcommand{\bli}{\begin{linenomath}}
\newcommand{\eli}{\end{linenomath}}

\usepackage{lineno}

\definecolor{bblue}{rgb}{.2,0.2,.8}

\begin{document}
	\maketitle

    \begin{abstract}
		We present a general framework to study the metastability of random perturbations of dynamical systems. It integrates techniques from the theory of Markov processes, in particular the resolvent approach to metastability, with the spectral analysis of transfer operators associated to the dynamics. The proposed framework is applied to study the metastability of one-dimensional dynamical systems generated by a map randomly perturbed by sub-Gaussian noise. \\       

        \noindent \textbf{Keywords}: Metastability, Random perturbations, Resolvent approach, Expanding maps, Chen-stein method
	\end{abstract}

    \tableofcontents
	
	\section{Introduction}
	
	Metastability is a characteristic of time-evolving physical systems that quickly attain a \textit{false-equilibrium}, that holds for a very long time before coming abruptly to an end, being followed by either the \textit{true-equilibrium} state, or a different \textit{false-equilibrium}. Metastable systems have been historically studied from the point of view of Markov processes and from that of perturbation of dynamical systems.
	
	Metastable Markov processes are such that, starting from a specific subset of the state space, the process will quickly equilibrate inside this subset, e.g., attain a \textit{false-stationary} distribution supported in this subset. But then, a sequence of rare events will take the process out of this equilibrium and into another one, that might be the real stable state of the process or another \textit{false-equilibrium}. 
	
	A dynamical system with stable invariant subsets, upon being deterministically or randomly perturbed, may become ergodic. However, when starting from a stable subset of the original system, the perturbed system may spend a very long time in the subset, as if it had attained equilibrium, before reaching another stable subset of the original system. The asymptotic behaviour of the system regarding the visits to the stable subsets and the limit of its unique invariant measure when the \textit{level of perturbation} tends to zero characterises the \textit{stability} of the system to perturbations. 
    
    Special classes of random perturbations of dynamical systems lead to Markov processes that can be analysed with their tools. In this paper, we present a general framework to study the metastability of random perturbations of dynamical systems that combines techniques from Markov processes, in particular the resolvent approach to metastability proposed by \cite{landim2021resolvent}, with the spectral analysis of operators associated to perturbations of dynamical systems, see the monographs \cite{Liv_demer_kam}, \cite{Hennion} and \cite{pene} for an advanced exposition of this topic. The proposed framework is applied to study the metastability of one-dimensional dynamical systems generated by a map randomly perturbed by sub-Gaussian noise. 
    
	\subsection{Metastability of Markov processes}
    
	The concept of metastability, motivated by a homonym physical phenomenon, was first introduced in the context of Markov process in the seminal work by Cassandro, Galves, Olivieri and Vares \cite{cassandro1984metastable}, where a rigorous method for deducing the metastable behaviour of Markov processes, called the \textit{pathwise approach to metastability}, was proposed. 
	
	Based on the theory of large deviations developed by \cite{ventsel1970small}, this approach considers that metastability occurs when a process ``has a unique stationary probability measure, but if the initial conditions are suitably chosen then the time to get to the asymptotic state becomes very large, and during this time the system behaves as if it were described by another stationary measure; finally, and abruptly, it goes to the true equilibrium'' \cite{cassandro1984metastable}. 
	
	Clearly, this approach only applies to processes that have a true equilibrium, but a \textit{metastable behaviour} can also be observed in processes that, starting from a subset of the state space, stays there for a very long time, as if it had stabilised, until, abruptly, it attains another subset of the state space where an analogous behaviour may happen; in particular, there might not be any stable subset, from which the process cannot leave.
	
	In this context, in the 2000s, a new approach for metastability based on potential theory was proposed by \cite{bovier2002metastability,bovier2004metastability,gayrard2005metastability}. In the \textit{potential theoretic approach to metastability}, sharp estimates for the transition times between specific subsets of the state space, called \textit{metastable wells}, are derived based on potential theory. In this approach, the metastability phenomenon is interpreted as a sequence of visits to different metastable wells and characterised via the precise analysis of the hitting times of these wells. A detailed account of the potential theoretic approach to metastability can be found in \cite{bovier2016metastability}.
	
	In this paper, we are concerned with the \textit{martingale approach to metastability} proposed by \cite{beltran2012tunneling,beltran2010tunneling,beltran2015martingale} that is based on the characterisation of Markov processes as unique solutions of martingale problems. Similar to the potential theoretic approach, metastability is described as the dynamics of visits to \textit{metastable wells}. 
	
	Informally, the Markov process is parameterised by some quantity, e.g., by the number of particles, by the level of perturbation or by some physical quantity, such as temperature, and the metastable behaviour is characterised in the limit when the respective parameter converges to $0$ or $\infty$. In metastable processes, there are $\kappa \geq 2$ disjoint subsets of the state space, called wells, such that, starting from a configuration or a measure supported in one of these wells, in the limit, the process will spend an exponentially distributed random time in this well before jumping to another one, spending a negligible amount of time outside the wells. In particular, in the limit, the dynamic of visits to the metastable wells, properly time scaled, will be represented by a Markov process.

    Through the martingale characterisation of Markov processes, sufficient conditions for metastability were derived based on potential theory by assuming the existence of an attractor inside each well \cite{beltran2012tunneling,beltran2010tunneling}, i.e., a state that is visited with probability tending to one before the process leaves the respective well. Sufficient conditions based on the mixing time and on the relaxation time of the process reflected at the boundary of the metastable wells were also derived \cite{beltran2015martingale}. The martingale approach has been successfully applied to many models, and we refer to \cite{landim2019metastable} for a review.
	
	Recently, a condition on the solution of a resolvent equation associated with the process generator was shown by \cite{landim2021resolvent} to be necessary and sufficient for the metastability of Markov processes when starting from a configuration in a metastable well. In particular, the resolvent approach to metastability has been successfully applied to derive the metastable behaviour of processes for which the previous techniques \cite{beltran2012tunneling,beltran2010tunneling,beltran2015martingale} could not be applied in the context of interacting particle systems \cite{choi2024emph,kim2024metastable,kim2025hierarchical,lee2022non,stephan2025coarse} and diffusion processes \cite{landim2025gamma,landim2023metastability,landim2024metastability,landim2025one}. 
    
    However, the resolvent approach, and to the best of our knowledge the martingale approach, have not yet been applied to Markov jump processes with uncountable state space such as randomly perturbed maps of the interval. Metastability of Markov processes via the martingale approach has been historically studied for interacting particle systems and diffusion processes, and in this paper we broaden its scope by combining the resolvent approach with spectral techniques to study the metastability of random perturbations of dynamical systems generated by a map. 

    \subsection{Metastability of random perturbations of a map}

    In this paper, we will be concerned with the metastability of dynamical systems generated by a randomly perturbed map, a framework known in the literature as {\em annealed} perturbation. This kind of metastability was first studied in \cite{bahsoun2013escape}, see also \cite{bahsoun_hu_vaienti, bahsoun_schmeling_vaienti} for complementary results,  and, in its simplest form, can be stated as follows. 
    
    Let $T$ be a piecewise expanding map of the interval that admits two ergodic invariant densities $\rho_1, \rho_2$ with support in disjoint sets $I_j, j=1,2$, and consider a random transformation built upon $T$. It is well known that random transformations induce a homogeneous Markov chain, which we assume possesses a unique smooth stationary probability measure with density $\rho_\eps$ where $\eps > 0$ represents the \textit{level of perturbation}. This implies that the perturbation brought about two small holes allowing leakage of mass between the two ergodic initial subsystems in $I_j, j=1,2$.
    
    In \cite{bahsoun2013escape} it was proved that when the size of the holes tend to zero, or equivalently the intensity $\eps$ of the perturbation tends to zero, the density $\rho_\eps$ converges in the $L^1$ norm to a convex combination of $\rho_1$ and $\rho_2$. This convergence is known in the literature as {\em stochastic stability}. It was shown that the ratio of the weights in the convex combination is equal to the ratio of the averages of the measures of the two holes. Moreover, those weights coincide with the ratio of the maximal eigenvalues of the transfer operators restricted to the original invariant subsets and perturbed by removing the respective holes, which are related to two {\em open} systems.
    
    The approach of \cite{bahsoun2013escape} was inspired on the paper by Gonzalez-Tokman, Hunt and Wright \cite{gonzalez2011approximating} where similar results were originally proved under deterministic perturbations for piecewise expanding maps of the interval. The assumption that the derivative was uniformly larger than one was then removed in \cite{bahsoun2011metastability} which dealt with maps admitting a neutral fixed point. The recent work \cite{gonzalez2025averaging} has studied metastability under the general framework of \cite{gonzalez2011approximating}, but considering \textit{quenched} random perturbations\footnote{Beside the aforementioned papers above, we also cite \cite{Froyland_Stancevic_2010, Froyland_Stancevic_2013, Gonzalez_tokman_quas_quarantine, Horan_21}, where metastability and related topics were treated in the context of quenched perturbed dynamical systems. Those papers have also in common the use of the multiplicative ergodic theorem to deal with cocycles of transfer operators.}. 
    
    A quenched random dynamical system is defined through compositions of maps $T_\omega$ from a collection $\{T_\omega\}_{\omega\in\Omega}$ induced by an ergodic invertible measure-preserving map $\sigma:\Omega\to\Omega$ on a probability space $(\Omega,\cF,\mu)$ that creates a map cocycle, or random dynamical system, namely $T_\omega^n := T_{\sigma^{n-1}\omega}\circ\cdots\circ T_{\sigma\omega}\circ T_\omega$. The randomness of the system arises from selecting $\omega$ from $\Omega$ with probability $\mu$. The annealed perturbation differs from this setting since, at each iteration, a new $\omega_{n}$ is independently draw from $\Omega$, instead of being determined deterministically from the previous one, and the random system is $T_{\bar{\omega}}^n := T_{\omega_{n}}\circ\cdots\circ T_{\omega_{1}}$. The annealed case  allows us to introduce the stationary measure to characterise the averaged distribution of the orbits; in the quenched case instead one considers a family of probability measures satisfying the equivariance condition $\mu_{\omega,0}\circ T_{\omega}^{-1}=\mu_{\sigma \omega,0}.$ These measures will be the underlying probabilities to establish limit theorems.
        
    The paper \cite{gonzalez2011approximating}, and all that followed, cited above, rely heavily on the perturbation results of Keller and Liverani \cite{KL, KL2}, which provide a description of metastability in terms of the spectrum of the transfer (Perron-Frobenius) operator. Since the map $T$ preserves two ergodic densities, its transfer operator has two eigenfunctions associated with the maximal eigenvalue $1$, which has geometric multiplicity $2$. By perturbing $T$, the system now admits a unique invariant (or stationary) measure, hence the maximal eigenvalue split into two eigenvalues of geometric multiplicity $1$, the first still equal to $1$, the second lesser, but very close to $1$. The rate of mixing of the measure with density $\rho_\eps$, being given by the spectral gap between the first and the second eigenvalues, goes to zero as $\eps$ tends to zero, implying that the invariant (or stationary) measure mixes slowly, with the orbits of the system lingering in the almost invariant \textit{metastable} sets. This behaviour is what we mean by {\em metastability} in the context of perturbation of dynamical systems, and can be naturally extended to systems with more than two ergodic components.
    
    The use of spectral techniques in \cite{gonzalez2011approximating} also inspired the paper by Dolgopiayt and Wright \cite{Dolgopyat_Wright}, which studies metastability in a sense close to that of the resolvent approach in Markov processes. They considered a deterministically perturbed map $T_{\epsilon}$ of a uniformly expanding map of the interval $T$ and studied the Markov process associated with $T_{\eps}$ which, starting from a point $x$, remains in this state for a mean one exponentially distributed random time before jumping to $T_{\eps}(x)$. Then, they considered a continuous-time stochastic process, with the number of states equal to the number of invariant components of the original map, representing the component the Markov process associated to $T_{\eps}$ is at each time. This is clearly not a Markov process, but they showed that, when $\eps$ tends to zero, this process, properly time rescaled, starting from an invariant measure of the original map converges to a Markov process. 
    
    In particular, this means that, in the limit starting from an invariant measure, the system spends an exponential time inside each invariant set before jumping to another. They prove that the rate of jumping from the component $i$ to $j$ is $\beta_{i,j}$ verifying $ \mu_i(H_{ij,\epsilon}) = \epsilon\beta_{i,j} + o(\epsilon)$ in which $H_{ij, \epsilon}$ is the hole in the interval $I_i$ whose points are sent into $I_j,$ and $\mu_i$ is the invariant measure on $I_i$. In this instance, the timescale of the process is $\eps$. This result is analogous to that which can be obtained via the martingale approach to metastability presented in \cite{beltran2015martingale}, although it was obtained with entirely different techniques, i.e., relying on the perturbation results of Keller and Liverani \cite{KL, KL2}. In \cite{Dolgopyat_Wright} they also showed that the diffusion coefficient for an observable of bounded variation can be approximated by the diffusion coefficient of the limiting Markov process. Analogous results have recently been obtained in \cite{gonzalez2025jumping} for quenched random perturbations.
    
    \subsection{A resolvent approach to the metastability of random maps}

    Although related to the papers discussed above, our approach to metastability is singular within the theory of randomly perturbed dynamical systems. Following the resolvent approach to metastability, we consider the Markov process associated with a random transformation of an expanding map and study the limit of the time scaled {\em order process}, that represents the invariant set, or {\em metastable well}, the associated Markov process is at each time.
    
    However, unlike \cite{Dolgopyat_Wright,gonzalez2025jumping}, we consider annealed instead of quenched and deterministic perturbations, and the limit of the order process is taken when the process starts from a point in a metastable well, instead from a distribution supported in it. This establishes the metastability over random orbits with \textit{any} initial value, rather than on average over orbits starting from an invariant set. Furthermore, by applying the Chen-Stein method \cite{barbour1988,chen1975,stein1972}, we show that stochastic stability follows from the resolvent approach as a corollary.

    The main tool in our analysis is the resolvent approach to metastability of \cite{landim2021resolvent}, which we extend to Markov jump processes in uncountable state spaces. Even though we rely on the perturbation results of Keller and Liverani \cite{KL, KL2}, unlike \cite{bahsoun2013escape} and the other papers discussed above, they are not used to \textit{describe} metastability, but are rather tools to prove a mixing condition that is sufficient for metastability under the resolvent approach. The spectral analysis is limited to the operator associated to a random map that equals the perturbed one for points not in the hole, and that can be defined in any desired way in the hole. This leads to less technical and simplified analysis compared to \cite{bahsoun2011metastability,bahsoun2013escape,Dolgopyat_Wright,gonzalez2011approximating,gonzalez2025averaging,gonzalez2025jumping}. For instance, the perturbed operator does not have discontinuities, a great source of technical difficulties in the aforementioned papers.

    In particular, the lack of discontinuities, allow to analyse \textit{continuous} perturbations, improving the result of \cite{bahsoun2013escape}. A crucial assumption in that paper (see Remark 3.10 therein) is that the map $\omega \mapsto T_{\omega}$ has a finite range\footnote{This was also the case for the quenched perturbation in \cite{gonzalez2025jumping}.}, or equivalently, that $\{T_{\omega}\}_{\omega \in \Omega}$ has a finite number of distinct elements. In this paper, we assume there are uncountably many maps in $\{T_{\omega}\}_{\omega \in \Omega}$, a case that can hardly be treated by spectral analysis alone.

    \subsection{Paper structure}
	
	In Section \ref{Sec_Meta}, we outline the general framework for metastability based on the resolvent approach in the context of random maps, and in Section \ref{Sec_Suf} we present sufficient conditions for the metastability of random mixing maps of the interval. In Section \ref{SecCS}, we prove that metastability implies the stochastic stability of the random map, namely, the strong convergence of the stationary measure of the randomly perturbed map to a convex combination of the invariant measures of the original map. We present rates for this strong convergence based on the Chen-Stein method \cite{barbour1988}. In Sections \ref{Sec_applications} and \ref{SecEx2}, we apply the general approach to expanding piecewise linear maps with two and three invariant components, respectively. In Section \ref{Sec_RM} we give our final remarks and proofs for the results are presented in Section \ref{Sec_proofs}.

    \subsection{Notation}

    For each measurable set $A \subset \mathbb{R}$, we denote by $\chi_{A}(x) = \mathds{1}\{x \in A\}$ its characteristic function and by $\partial A = \widebar{A}\setminus A$ its boundary. The Lebesgue measure of $A$ is denoted by $\text{Leb}(A)$. For $a,b \in \mathbb{R}$ we denote $a \vee b = \max\{a,b\}$ and $a \wedge b = \min\{a,b\}$. Given two sequences $\{a_{\eps}: \eps > 0\}$ and $\{b_{\eps}: \eps > 0\}$, we say that $a_{\eps} \ll b_{\eps}$ if $\lim_{\eps \to 0} a_{\eps}/b_{\eps} = 0$.
	
	\section{A resolvent approach to the metastability of randomly perturbed maps}
	\label{Sec_Meta}
	
	Fix a closed interval $I \subset \mbR$ and let $T: I \mapsto I$ be a map satisfying:	
	\begin{itemize}
		\item[\textbf{(A1)}] There exists a collection of disjoint subsets $I_{1},\dots,I_{\kappa}$ of $I$ for $\kappa < \infty$ such that:
		\begin{enumerate}
			\item[\textbf{(A1.1)}] Each set $I_{i}$ can be decomposed as
			\begin{equation*}
				I_{i} = \bigcup_{j = 1}^{\kappa_{i}} I_{ij}
			\end{equation*}
			in which $I_{ij}$ are open intervals where $T$ is continuous and one-to-one with $\kappa_{i} < \infty$.
			\item[\textbf{(A1.2)}] This collection of sets essentially covers $I$:
			\begin{equation*}
				\widebar{\bigcup_{i} I_{i}} = I 
			\end{equation*}
			\item[\textbf{(A1.3)}] For each set in the collection it holds $T(I_{i}) \subset \widebar{I_{i}}$ and $$T^{-1}(\partial I_{i}) \cap I_{i} = \{b^{(i)}_{1},\dots,b_{l_{i}}^{(i)}\}$$ with $1 \leq l_{i} < \infty$. We call $b^{(i)}_{j}$ \textit{infinitesimal holes}.
            \item[\textbf{(A1.4)}] For $i,j = 1,\dots,\kappa$, there exists a sequence $i_{0} = i, i_{1},\dots,i_{r} = j$ with $r \geq 1$ such that $T(I_{i_{l}}) \cap \partial I_{i_{l + 1}} \neq \emptyset$ for all $l = 1,\dots,r$.
		\end{enumerate}
	\end{itemize}
	To fix ideas, the reader can consider the map presented in Figure \ref{f1}, in which $I = [0,1]$, $I_{1} = (0,1/2)$ and $I_{2} = (1/2,1)$.
	
	Fix a probability space $(\Omega,\mcF,\mbP)$ and let $\{T_{\epsilon}: \epsilon > 0\}$ be a collection of $\mcF$-measurable maps $T_{\epsilon}: \Omega \times I \mapsto I$. For $x \in I$ fixed, denote by $\mbM_{\epsilon}^{x}$ the probability measure induced by $T_{\epsilon}(\cdot,x)$ in the measurable space $(I,\mfB_{I}),$ where $\mfB_{I}$ is the Borel $\sigma$-algebra of $I$. This probability measure satisfies
	\begin{equation*}
		\mbM_{\epsilon}^{x}(A) = \mbP(\omega \in \Omega: T_{\epsilon}(\omega,x) \in A)
	\end{equation*}
	for $A \in \mfB_{I}$. We denote simply by $T_{\epsilon}(x)$ the random variable $T_{\epsilon}(\cdot,x)$ for $x \in I$ and $\epsilon > 0$ and by $S \coloneqq \{1,\dots,\kappa\}$ the indexes of the invariant components of $T$. We make the following assumptions about the random maps $\{T_{\epsilon}: \epsilon > 0\}$:
	\begin{itemize}
		\item[\textbf{(A2)}] For all $x \in I$ and $\epsilon > 0$, the probability measure $\mbM_{\epsilon}^{x}$ is absolutely continuous wrt Lebesgue measure\footnote{We stress that this assumption is more general than that in \cite{bahsoun2013escape}, that the map $\omega \mapsto T_{\eps}(\omega,\cdot)$ has finite range.}. In particular, there exists a function $\rho_{\epsilon}: I \times I \mapsto \mbR_{+}$ such that, for $x \in I$ fixed, $\rho_{\epsilon}(x,\cdot)$ is a probability density function and
		\begin{align*}
			\mbM_{\epsilon}^{x}(A) = \int_{A} \rho_{\epsilon}(x,y) \ dy
		\end{align*}
		for all $A \in \mfB_{I}$.
		\item[\textbf{(A3)}] There exist constants $\gamma_{1},\gamma_{2} > 0$ such that, for all $\epsilon > 0$ and $t > 0$,
		\begin{equation}
			\label{sub_gaussian}
			\sup\limits_{x \in I} \mbP\left(\omega \in \Omega: |T_{\epsilon}(\omega,x) - T(x)| > t\right) \leq \gamma_{1} \exp -\frac{t^{2}}{\gamma_{2} \epsilon^{2}}
		\end{equation}
		and
		\begin{align}
			\label{barrier}
			\limsup\limits_{\epsilon \to 0} \max_{i \in S}  \sup\limits_{x \in I_{i}} \mbP\left(T_{\epsilon}(x) \in I_{i}^{c}\right) = 0.
		\end{align}
	\end{itemize}
	
	Condition \eqref{sub_gaussian} in \textbf{(A3)} means that $T_{\epsilon}$ is a random perturbation of $T$ with the absolutely continuous measure induced by $T_{\epsilon}(x)$ concentrating around $T(x)$ with a sub-Gaussian like tail decay\footnote{Formally, a random variable $X$ is sub-Gaussian if it concentrates around its mean $\mu$ as $\mbP\left(|X - \mu| > t\right) \leq 2 \exp -t^{2}/C^{2}$ for all $t > 0$ and some constant $C > 0$. Condition \eqref{sub_gaussian} is actually weaker than assuming that $T_{\epsilon}(x)$ is sub-Gaussian; for instance, we are not assuming that the expected value of $T_{\epsilon}(x)$ is $T(x)$ and we could have $\gamma_{1} > 2$.}. 
	
	Condition \eqref{barrier} in \textbf{(A3)} implies that, as $\epsilon \to 0$, the probability of $T_{\epsilon}$ mapping a point in the subset $I_{i}$ to its complement converges to zero. This assumption can be interpreted as the existence of an energy barrier that the random dynamic needs to overcome to escape from $I_{i}$, that causes the escape probability to be small starting from any point in $I_{i}$.
		
	Examples of random maps that satisfy \textbf{(A2)} and \textbf{(A3)} are those perturbed by additive noise:
	\begin{equation}
        \label{additive_noise}
		T_{\epsilon}(\omega,x) = T(x) + \sigma_{\epsilon}^{x}(\omega)
	\end{equation}
	for random variables $\sigma_{\epsilon}^{x}, x \in I,$ which are absolutely continuous with support in $[\inf I - T(x),\sup I - T(x)]$ and satisfy
	\begin{equation*}
		\mbP\left(|\sigma_{\epsilon}^{x}| > t\right) \leq \gamma_{1} \exp -\frac{t^{2}}{\gamma_{2} \epsilon^{2}}
	\end{equation*}
	for all $t > 0$. 
	
	In particular, \textbf{(A2)} and \textbf{(A3)} are satisfied when $\sigma_{\epsilon}^{x}$ is the truncation to $[\inf I - T(x),\sup I - T(x)]$ of a mean zero sub-Gaussian random variable $\sigma_{\epsilon}$ satisfying
	\begin{equation}
		\label{add_noise}
		\mbP\left(|\sigma_{\epsilon}| > t\right) \leq 2 \exp -\frac{t^{2}}{\gamma_{2} \epsilon^{2}},
	\end{equation}
	for $x \notin T^{-1}(\bigcup_{i}(\partial I_{i} + B_{\delta}))$ with $\delta > 0$, in which $\partial I_{i} + B_{\delta}$ is the union of the $\delta$-balls centred at the points in the boundary of $I_{i}$. For $x \in T^{-1}(\bigcup_{i}(\partial I_{i} + B_{\delta}))$, either $\mbP(\sigma_{\epsilon}^{x} > 0)$ or $\mbP(\sigma_{\epsilon}^{x} < 0)$ should converge to zero for \eqref{barrier} to hold, depending on $x \in I_{i}$ being closer to the right or left limit of the interval $I_{i,j}$ that contains it, respectively. 
	
	For example, $\sigma_{\epsilon}$ can be taken as the truncation of a mean zero Gaussian distribution with variance $\epsilon^{2}$ or a uniform distribution in $[-\epsilon,\epsilon]$, while for $x \in T^{-1}(\bigcup_{i}(\partial I_{i} + B_{\delta}))$, $\sigma_{\epsilon}^{x}$ can be a truncated Gaussian distribution with a suitable non-zero mean or a uniform distribution that is not symmetric around zero, for instance with support $[-\epsilon,\epsilon^{q}]$ for $q > 2$.
	
	Maps randomly perturbed by multiplicative noise of the form
    \begin{equation}
        \label{multiplicative}
		T_{\epsilon}(\omega,x) = (1 + \sigma_{\epsilon}^{x}(\omega)) \, T(x)
	\end{equation}
    with $\sigma_{\eps}^{x}$ absolutely continuous can also be considered. In this case, property \textbf{(A2)} holds for all $x \in I$ with $T(x) \neq 0$ and \textbf{(A3)} can be achieved by taking $\sigma_{\epsilon}^{x}$ concentrating around zero with probability high enough as $\eps \to 0$.
    
	\subsection{Markov chain associated with randomly perturbed maps}
	\label{Sec1}
	
	Each randomly perturbed map $T_{\epsilon}$ has an associated Markov chain. For $\epsilon > 0$, $x \in I$ and a sequence $\{\omega_{n}: n \geq 0\}$ of random elements, sampled independently from $\Omega$ with probability measure $\mbP$, consider the Markov chain $\{X_{n}^{\epsilon}: n \in \mbN\}$ with state space $I$ satisfying $X_{0}^{\epsilon} = x$ and
	\begin{equation*}
		X_{n+1}^{\epsilon} = T_{\epsilon}(\omega_{n},X_{n}^{\epsilon}).
	\end{equation*}	
	This is the Markov chain with transition density kernel $\rho_{\epsilon}(x,y)$. We assume that, for all $\epsilon > 0$, $X_{n}^{\epsilon}$ is aperiodic and has one absolutely continuous invariant\footnote{From now on we make no distinction between invariant and stationary measure in the context of Markov chains and processes. We adopt the term invariant to be consistent with dynamical system theory.} measure (ACIM):
	\begin{itemize}
		\item[\textbf{(A4)}] For all $\epsilon > 0$, the Markov chain $\{X_{n}^{\epsilon}: n \in \mbN\}$ is aperiodic and has a unique ACIM $\mu_{\epsilon}$ with probability density function $p_{\epsilon}$ wrt Lebesgue measure.
	\end{itemize}
	In view of \textbf{(A3)}, this assumption is not strong as there is quite flexibility on the selection of the random perturbation. For instance, let $\rho_{\epsilon}^{n}(x,y)$ be the $n$-step transition density function defined recursively as $\rho_{\epsilon}^{1}(x,y) = \rho_{\epsilon}(x,y)$ and
	\begin{equation*}
		\rho_{\epsilon}^{n}(x,y) = \int_{I} \rho_{\epsilon}^{n-1}(x,z)\rho_{\epsilon}(z,y) \ dz
	\end{equation*}
	for $x,y \in I$ and $n \geq 2$. The next proposition states that if $\rho_{\epsilon}^{n_{\epsilon}}(x,y)$, for a $n_{\eps} < \infty$, is bounded away from zero for $(x,y)$ in the support of $\rho_{\eps}(x,y)$, then $X_{n}^{\epsilon}$ has a unique ACIM. In particular, this condition holds trivially with $n_{\epsilon} = 1$ for additive truncated Gaussian noise (cf. \eqref{add_noise}) independently of the map $T$.

    \begin{proposition}
		\label{prob_cone_ACIM}
		Fix $\epsilon > 0$. If there exist $n_{\eps} < \infty$ and $A_{\eps} \subset I$, with $\text{Leb}(A_{\eps}) > 0$ and $\rho_{\eps}(x,y) = \rho_{\eps}(x,y) \chi_{A_{\eps}}(y)$ for all $x,y \in I$, such that
        \begin{align}
			\label{cond_unique_ACIM}
			\inf_{x \in I, y \in A_{\eps}} \rho_{\epsilon}^{n_{\epsilon}}(x,y) > 0,
		\end{align}
        then $X_{n}^{\epsilon}$ has a unique ACIM.
	\end{proposition}
	
	\subsection{Markov process associated with randomly perturbed maps}
	
	A Markov jump process is a continuous-time stochastic process with random jumping times following an exponential distribution and with the jumping probabilities of a Markov chain. Formally, for each $\epsilon > 0$, let $(\eta_{\epsilon}(t): t \geq 0)$ be a continuous-time stochastic process defined on $(\Omega,\mcF,\mbP)$ and taking values in $I$. Define $\tau_{0} = 0$ and
	\begin{equation*}
		\tau_{n} = \inf\{t > \tau_{n-1}:\eta_{\epsilon}(t) \neq \eta_{\epsilon}(\tau_{n-1})\}
	\end{equation*}
	the time until the $n$-th jump of this process for $n \geq 1$. The process $\eta_{\epsilon}(\cdot)$ is a Markov process with embedded chain $X_{n}^{\epsilon}$ if, for instance, (a) $\tau_{n} - \tau_{n-1}$ follows a mean one exponential distribution, that is, $\mbP(\tau_{n} - \tau_{n-1} > t) = e^{-t}$ for all $n \geq 1$ and $t > 0$, (b) $\tau_{n} - \tau_{n-1}$ is independent of $\tau_{m} - \tau_{m-1}$ if $n \neq m$, and (c) for $n \geq 0$
	\begin{equation}
		\label{embed}
		\eta_{\epsilon}(\tau_{n}) = X_{n}^{\epsilon}.
	\end{equation}
	The process $\eta_{\epsilon}(\cdot)$ is the continuous-time equivalent of $X_{n}^{\epsilon}$: starting from a point $x \in I$, the process will spend a mean one exponentially distributed time at this point before jumping to another point sampled with density $\rho_{\epsilon}(x,\cdot)$. We note that $\eta_{\epsilon}(\cdot)$ has the same ACIM as $X_{n}^{\epsilon}$.
	
	The Markov process $\eta_{\epsilon}(\cdot)$, starting from a point in $I$, induces a measure on $D(\mathbb{R}_{+},I)$, the space of right-continuous functions $\boldsymbol{x}: \mathbb{R}_{+} \to I$ with left-limits, endowed with the Skorohod topology and its associated Borel $\sigma$-algebra. Observe that elements of $D(\mathbb{R}_{+},I)$ are trajectories of $\eta_{\epsilon}(\cdot)$. Let $\boldsymbol{P}^{\epsilon}_{x}, x \in I$, be the probability measure induced by the process $\eta_{\epsilon}(\cdot)$ starting from $x \in I$ on $D(\mathbb{R}_{+},I)$, and denote expectation under this measure by $\boldsymbol{E}^{\epsilon}_{x}$.
	
	The semigroup $\{\ms P_{\epsilon}(t): t \geq 0\}$ of $\eta_{\epsilon}(\cdot)$ is a collection of operators $\ms P_{\epsilon}(t): L^{\infty}(I) \mapsto L^{\infty}(I)$ satisfying $\ms P_{\epsilon}(0)F = F$ and $\ms P_{\epsilon}(t + s) = \ms P_{\epsilon}(t)\ms P_{\epsilon}(s)$ for $t,s > 0$, that are given by\footnote{We assume that the semigroup, and the generator, act on functions in $L^{\infty}(I)$ because that is enough for the purposes of this paper. Nevertheless, their supports are actually greater than $L^{\infty}(I)$.}
	\begin{align}
		\label{semigroup}
		(\ms P_{\epsilon}(t)F)(x) = \boldsymbol{E}_{x}^{\epsilon}\left[F(\eta_{\epsilon}(t))\right], & & F \in L^{\infty}(I), x \in I.
	\end{align}
	For $x \in I$, we denote by $\delta_{x} \ms P_{\epsilon}(t)$ the measure induced by $\eta_{\epsilon}(t)$ on $\mfB_{I}$ when $\eta_{\epsilon}(0) = x$, that is,
	\begin{equation}
		\label{semigroup_measure}
		(\delta_{x} \ms P_{\epsilon}(t))(A) \coloneqq (\ms P_{\epsilon}(t)\chi_{A})(x) = \boldsymbol{P}_{x}^{\epsilon}\left[\eta_{\epsilon}(t) \in A\right]
	\end{equation}
	for $A \in \mfB_{I}$.	
	
	The infinitesimal generator of $\eta_{\epsilon}(\cdot)$ is the operator $L_{\epsilon}: L^{\infty}(I) \mapsto L^{\infty}(I)$ defined as
	\begin{align}
		\label{gen_def}
		(L_{\epsilon} F)(x) \coloneqq \lim\limits_{t \to 0^{+}} \frac{(\ms P_{\epsilon}(t)F)(x) - F(x)}{t}, & & x \in I, F \in L^{\infty}(I).
	\end{align}
	A straightforward computation (see \cite[Section~4.2]{ethier2009markov}) yields
	\begin{align*}
		(L_{\epsilon} F)(x) = \int_{I} \rho_{\epsilon}(x,y) \left[F(y) - F(x)\right] \ dy = \mbM_{\epsilon}^{x}[F] - F(x)
	\end{align*}
	in which $\mbM_{\epsilon}^{x}[\cdot]$ is expectation under the probability measure $\mbM_{\epsilon}^{x}$. A Markov jump process in this context is defined by, equivalently, its transition density kernel, its semigroup or its generator. We refer to \cite[Chapter~1]{ethier2009markov} for more details about the characterisation of Markov processes.
	
	Markov processes can also be characterised as solutions to a \textit{martingale problem}. By applying the definition of generator \eqref{gen_def}, one can prove that
	\begin{align}
		\label{martin}
		M_{\epsilon}(t) = F(\eta_{\epsilon}(t)) - F(\eta_{\epsilon}(0)) - \int_{0}^{t} (L_{\epsilon} F)(\eta_{\epsilon}(s)) \ ds, & & F \in L^{\infty}(I)
	\end{align}
	is a martingale\footnote{The process $M_{\epsilon}(t)$ is a martingale if $\boldsymbol{E}_{x}^{\eps}[M_{\epsilon}(t)|\mfF_{s}^{\epsilon}] = M_{\epsilon}(s)$ for any $s < t$ and $x \in I$. In particular, in this case, $\boldsymbol{E}_{x}^{\eps}[M_{\epsilon}(t)] = 0$ since $M(0) = 0$ with probability one.} with respect to the filtration $\{\mfF_{t}^{\epsilon}\}_{t \geq 0}$ that is the augmentation of $\mfF_{t}^{0,\epsilon} = \sigma(\eta_{\epsilon}(s): s \leq t)$, the natural filtration of $D(\mathbb{R}_{+},I)$ generated by $\{\eta_{\epsilon}(s): s \leq t\}$. However, the converse is also true: under broad conditions, for instance satisfied by the jump processes we consider in this paper, a stochastic process $Y(\cdot)$ is a Markov process with generator $L$ only if
	\begin{equation}
		\label{martin_problem}
		M(t) = F(Y(t)) - F(Y(0)) - \int_{0}^{t} (L F)(Y(s)) \ ds
	\end{equation} 
	is a martingale with respect to its respective filtration $\{\mfF_{t}\}_{t \geq 0}$ for functions $F$ in a suitable space. From now on, when we say that a process analogous to \eqref{martin_problem} is a martingale, it should be implicit that it is with respect to the respective filtration $\{\mfF_{t}\}_{t \geq 0}$. 
	
	Therefore, to prove that a certain stochastic process $Y(\cdot)$ is a Markov process, it is enough to show that \eqref{martin_problem} is a martingale. Conversely, a Markov process can be characterised as the unique solution of the martingale problem \eqref{martin_problem} for some generator $L$, that is, the unique stochastic process such that \eqref{martin_problem} is a martingale. The existence and uniqueness is in the sense of the finite-dimensional distributions, that is, fixed $t_{1},\dots,t_{k}$, if $Y(\cdot)$ and $Y^{\prime}(\cdot)$ are solutions of \eqref{martin_problem}, then the random vectors $(Y(t_{1}),\dots,Y(t_{k}))$ and $(Y^{\prime}(t_{1}),\dots,Y^{\prime}(t_{k}))$ have the same distribution. This characterisation of Markov processes was pioneered by Stroock and Varadhan \cite{stroock1997multidimensional}, and more details can be found in \cite[Chapter~4]{ethier2009markov}.
	
	In view of relation \eqref{embed}, we will analyse the asymptotic behaviour of $\eta_{\epsilon}(\cdot)$ as $\epsilon \to 0$ to characterise that of the Markov chain generated by the randomly perturbed map $T_{\epsilon}$. In particular, we will describe the metastable behaviour of $\eta_{\epsilon}(\cdot)$ that implies an analogous behaviour for $X_{n}^{\epsilon}$. Following the resolvent approach to metastability, it is described via the solution of a particular martingale problem.
	
	\subsection{Metastable wells}
	\label{Sec2}
	
	The first step to describe the metastable behaviour of a Markov process is to identify \textit{metastable wells} for the process dynamic. These are subsets of $I$ such that, starting from one of them, the process will spend a long time in it before jumping to another subset, without spending a significant amount of time outside these subsets. This characterisation of metastability will be formalised in Section \ref{SecMeta}, but for now, we focus on defining the metastable wells.
	
	Recall from \textbf{(A1)} that the map $T$ has $\kappa$ disjoint invariant sets, and that we are denoting $S = \{1,\dots,\kappa\}$. For each pair $i \neq j \in S$ and $\epsilon > 0$, let $B_{i,j}^{\epsilon} \subset \bar{I_{i}}$ be a closed set such that $\partial I_{i} \cap \widebar{I_{j}} \subset B_{i,j}^{\epsilon}$ and $T^{-1}(B_{i,j}^{\epsilon}) \cap I_{i}$ is the union of a finite number of intervals if such a set exists. Otherwise, let $B_{i,j}^{\eps} = \emptyset$. 
    
    We assume that $T^{-1}(B_{i,j}^{\epsilon}) \cap T^{-1}(B_{i,j^{\prime}}^{\epsilon}) = \emptyset$ if $j \neq j^{\prime}$, what can be achieved due to \textbf{(A1.3)}, since there are at most a finite number of points in $T^{-1}(\partial I_{i}) \cap I_{i}$. Define\footnote{The sets $\Delta^{\epsilon}_{i,j}$ could be generally defined as the union of neighbourhoods of points in $T^{-1}(\partial I_{i} \cap \bar{I_{j}}) \cap I_{i}$ satisfying $T(\Delta^{\epsilon}_{i,j}) \subset B_{i,j}^{\epsilon}$. We considered the definition \eqref{holes} for simplification, but the theory also holds otherwise.}
	\begin{align}
		\label{holes}
		\Delta^{\epsilon}_{i,j} \coloneqq T^{-1}\left(B_{i,j}^{\epsilon}\right) \cap I_{i}
	\end{align}
	as the points in $I_{i}$ such that $T(x)$ is in the neighbourhood $B_{i,j}^{\epsilon}$ of $I_{j}$. We call $\Delta_{i,j}^{\eps}$ the hole from $I_{i}$ to $I_{j}$. For each $i \in S$, due to \textbf{(A1.2)} and \textbf{(A1.3)}, $\Delta^{\epsilon}_{i,j} \neq \emptyset$ for at least one $j \in S$, but $\Delta^{\epsilon}_{i,j^{\prime}}$ may be empty for other values $j^{\prime} \in S$. We define the $i$-th hole as
	\begin{align*}
		\Delta^{\epsilon}_{i} \coloneqq \bigcup\limits_{\substack{j \in S \\ j \neq i}} \ \Delta^{\epsilon}_{i,j},
	\end{align*} 
	which is a non-empty union of a finite number of intervals. Finally, we set
	\begin{equation*}
		\Delta^{\epsilon} \coloneqq \bigcup_{i \in S} \Delta^{\epsilon}_{i}.
	\end{equation*}
    We assume that the sets $B_{i,j}^{\eps}$ are chosen in a way such that
    \begin{align}
        \label{lim_Leb_holes}
        \lim\limits_{\eps \to 0} \text{Leb}(\Delta^{\eps}) = 0.
    \end{align}
	
	Due to the sub-Gaussian assumption \textbf{(A3)}, as $\epsilon \to 0$, with increasing probability, starting from $I_{i}$ the process $\eta_{\epsilon}(\cdot)$ will reach $I_{j}$ in one jump only if it has started in $\Delta^{\epsilon}_{i,j}$ for a suitable choice of $B_{i,j}^{\epsilon}$. In other words, this is the set where, with high probability, the process escapes from $I_{i}$ to $I_{j}$. 
	
	For each $i \in S$, define the \textit{metastable well}
	\begin{equation*}
		\mcE^{\epsilon}_{i} \coloneqq I_{i}\setminus\Delta^{\epsilon}_{i} \, ,
	\end{equation*}
	which is also the union of a finite number of intervals, and denote
	\begin{align}
		\label{cwells}
		\mcE^{\epsilon} \coloneqq \bigcup_{i \in S} \mcE^{\epsilon}_{i} & & \check{\mcE}^{\epsilon}_{i} \coloneqq \bigcup_{\substack{j \in S \\ j \neq i}} \mcE^{\epsilon}_{j}
	\end{align}
	An illustration of the definition of holes and metastable wells is presented in Figure \ref{f1}. We assume that the unique ACIM $\mu_{\epsilon}$ of the process $\eta_{\epsilon}$ does not concentrate on $\Delta^{\epsilon}$ in the following sense:
	\begin{itemize}
		\item[\textbf{(A5)}] It holds
		\begin{align}
			\label{cm}
			\min_{i \in S} \lim\limits_{\epsilon \to 0} \mu_{\epsilon}(\mcE^{\epsilon}_{i}) > 0 & & \text{ and } & & \lim\limits_{\epsilon \to 0} \mu_{\epsilon}(\Delta^{\epsilon}) = 0.
		\end{align}
	\end{itemize}
	
	We note that \textbf{(A5)} should be taken into consideration when choosing the collection of sets $\{B_{i,j}^{\epsilon}: \epsilon > 0\}$ in the definition of $\Delta_{i,j}^{\epsilon}$ (cf. \eqref{holes}).

    \begin{figure}[ht]
		\centering
		\begin{tikzpicture}[scale=8]
			\tikzstyle{hs} = [circle,draw=black, rounded corners,minimum width=3em, vertex distance=2.5cm, line width=1pt]
			\tikzstyle{hs2} = [circle,draw=black,dashed, rounded corners,minimum width=3em, vertex distance=2.5cm, line width=1pt]
			
			\node at (0,-0.04) {$0$};
			\node at (1,-0.04) {$1$};
            \node at (0.5,-0.05) {$\frac{1}{2}$};
			\node at (0.25,-0.05) {$\frac{1}{4}$};
			\node at (0.75,-0.05) {$\frac{3}{4}$};
			
			\node at (-0.03,1) {$1$};
            \node at (-0.075,0.95) {$1-b$};
            \node at (-0.03,0.05) {$b$};
			\node at (-0.03,0.5) {$\frac{1}{2}$};
			\node at (-0.03,0.25) {$\frac{1}{4}$};
			\node at (-0.03,0.75) {$\frac{3}{4}$};
			
			\node at (0.125,0.03) {\color{red}$\Delta^{\epsilon}_{1,2}$};
			\node at (0.625,0.03) {\color{orange}$\Delta^{\epsilon}_{2,1}$};
			\node at (0.7,-0.05) {\color{violet}$\mcE^{\epsilon}_{2}$};
			\node at (0.2,-0.05) {\color{teal}$\mcE^{\epsilon}_{1}$};
			
			\begin{scope}[line width=0.5pt]
				\draw[->] (0,0) to (1.05,0);
				\draw[->] (0,0) to (0,1.05);
				
				\draw[-,blue] (0,0.05) to (0.125,0.5);
                \draw[-,blue] (0.125,0.5) to (0.25,0.05);
                \draw[-,blue] (0.25,0.05) to (0.375,0.5);
				\draw[-,blue] (0.375,0.5) to (0.5,0.05);
				
				\draw[-,blue] (0.5,0.95) to (0.6125,0.5);
                \draw[-,blue] (0.6125,0.5) to (0.75,0.95);
                \draw[-,blue] (0.75,0.95) to (0.875,0.5);
				\draw[-,blue] (0.875,0.5) to (1,0.95);
				
				\draw[-] (1,-0.01) to (1,0.01);
				\draw[-] (0.5,-0.01) to (0.5,0.01);
				\draw[-] (0.25,-0.01) to (0.25,0.01);
				\draw[-] (0.75,-0.01) to (0.75,0.01);
				
				\draw[-] (-0.01,1) to (0.01,1);
                \draw[-] (-0.01,0.95) to (0.01,0.95);
				\draw[-] (-0.01,0.5) to (0.01,0.5);
				\draw[-] (-0.01,0.25) to (0.01,0.25);
                \draw[-] (-0.01,0.05) to (0.01,0.05);
				\draw[-] (-0.01,0.75) to (0.01,0.75);
				
				\draw[-,teal,line width=2pt] (0,0) to (0.5,0);
				\draw[-,violet,line width=2pt] (0.5,0) to (1,0);
				
                \draw[-,red,line width=2pt] (0.125-0.025,0) to (0.125+0.025,0);
                \draw[-,red,line width=2pt] (0.375-0.025,0) to (0.375+0.025,0);
				\draw[-,orange,line width=2pt] (0.615-0.025,0) to (0.615+0.025,0);
                \draw[-,orange,line width=2pt] (0.875-0.025,0) to (0.875+0.025,0);
				
				\draw[-,dashed,opacity = 0.5] (0,0.5) to (1,0.5);
				\draw[-,dashed,opacity = 0.5] (0.5,0) to (0.5,1);
				\draw[-,dashed,opacity = 0.5] (0,0) to (1,1);
			\end{scope}
			\end{tikzpicture}
		\caption{Example of a map $T$ that satisfies \textbf{(A1)} with an illustration of the holes and metastable wells.} \label{f1}
	\end{figure}
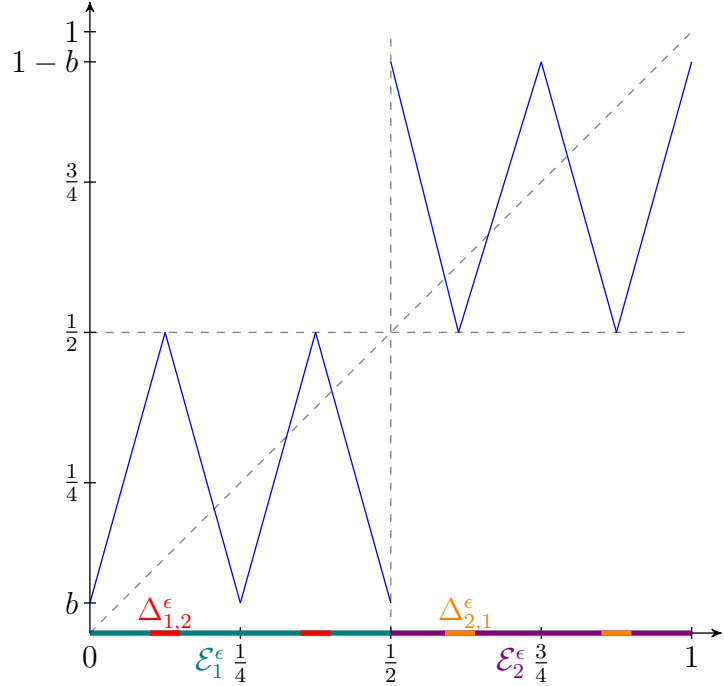
	
	\subsection{Speeded-up and order processes}
	
	The metastable behaviour of $\eta_{\epsilon}(\cdot)$ is described through the visits to the metastable wells of a speeded-up version of it, as follows.
	
	For a sequence $\{\beta_{\epsilon}: \epsilon > 0\}$ with $\liminf_{\epsilon \to 0} \beta_{\epsilon} = +\infty$, to be determined later, the process $\eta_{\epsilon}(\cdot)$ speeded-up by $\beta_{\epsilon}$ is defined as 
	\begin{equation*}
		\xi_{\epsilon}(t) \coloneqq \eta_{\epsilon}(\beta_{\epsilon}t),
	\end{equation*}
	that is the Markov process with generator
	\begin{align*}
		(\mathcal{L}_{\epsilon}F)(x) \coloneqq \beta_{\epsilon}(L_{\epsilon}F)(x) & & F \in L^{\infty}(I), x \in I.
	\end{align*}
	This follows by the martingale characterisation of Markov processes, since a change of variables in the integral in \eqref{martin} implies that
	\begin{equation}
		\label{martin_speed}
		M_{\epsilon}^{\prime}(t) = F(\eta_{\epsilon}(\beta_{\epsilon}t)) - F(\eta_{\epsilon}(0)) - \beta_{\epsilon} \int_{0}^{t} (L_{\epsilon} F)(\eta_{\epsilon}(\beta_{\epsilon}s)) \ ds
	\end{equation}
	is a martingale.
	
	The visits of the speeded-up process to the metastable wells are described by the order process. For $A \in \mfB_{I}$, denote by $T^{A}(t)$ the total time the process $\xi_{\epsilon}(\cdot)$ spends in $A$ in the time-interval $[0,t]$:
	\begin{linenomath}
		\begin{equation}
			\label{time_spent}
			T^{A}(t) = \int_{0}^{t} \chi_{A}(\xi_{\epsilon}(s)) \ ds.
		\end{equation}
	\end{linenomath}
	Denote by $S^{A}(t)$ the generalised inverse of $T^{A}(t)$:
	\begin{linenomath}
		\begin{equation}
			\label{gen_inverse}
			S^{A}(t) = \sup\{s \geq 0: T^{A}(s) \leq t\},
		\end{equation}
	\end{linenomath}
	that is, how long it takes for the process $\xi_{\epsilon}(\cdot)$ to spend an amount of time $t$ in $A$.
	
	The trace of $\xi_{\epsilon}(\cdot)$ on $A$, denoted by $(\xi_{\epsilon}^{A}(t): t \geq 0)$, is defined by
	\begin{linenomath}
		\begin{equation}
			\label{trace_process}
			\xi_{\epsilon}^{A}(t) = \xi_{\epsilon}(S^{A}(t)),
		\end{equation}
	\end{linenomath}
	that is a $A$-valued Markov process, obtained by stopping following the process $\xi_{\epsilon}(\cdot)$ when it visits $A^{c}$ and restarting when it reaches $A$ again.
	
	Recall the definition of $\mcE^{\epsilon}$ (cf. \eqref{cwells}) and that $S = \{1,\dots,\kappa\}$. Let $\Psi: \mcE^{\epsilon} \to S$ be the projection of $\mcE^{\epsilon}$ on $S$ given by
	\begin{linenomath}
		\begin{equation*}
			\Psi(x) = \sum_{i \in S} i \ \chi_{\mcE^{\epsilon}_{i}}(x)
		\end{equation*}
	\end{linenomath}
	which equals the index of the metastable well $\mcE^{\epsilon}_{i}$ that contains $x \in \mcE^{\epsilon}$. The order process $(Y_{\epsilon}(t): t \geq 0)$ is defined as
	\begin{linenomath}
		\begin{align*}
			Y_{\epsilon}(t) \coloneqq \Psi(\xi_{\epsilon}^{\mcE^{\epsilon}}(t)), 
		\end{align*}
	\end{linenomath}
	and represents the index of the set $\mcE^{\epsilon}_{i}$ that the trace process $\xi_{\epsilon}^{\mcE^{\epsilon}}(\cdot)$ is at each time. Observe that the stochastic process $Y_{\epsilon}(\cdot)$, which represents the evolution of the visits of the trace process to the metastable wells, is \textit{not} a Markov process. 
	
	To easy notation, we also denote $\boldsymbol{P}^{\epsilon}_{x}, x \in I$, for the probability measure induced on $D(\mathbb{R}_{+},I)$ by the process $\xi_{\epsilon}(\cdot)$ starting from $x \in I$, with expectation with respect to it also denoted by $\boldsymbol{E}_{x}^{\epsilon}$. Let $\mbQ_{x}^{\epsilon}, x \in I,$ be the probability measure induced by measure $\boldsymbol{P}_{x}^{\epsilon}$ and the order process $Y_{\epsilon}(\cdot)$ on $D(\mathbb{R}_{+},S)$, the space of right-continuous functions $\boldsymbol{x}: \mathbb{R}_{+} \to S$ with left-limits, endowed with the Skorohod topology and its associated Borel $\sigma$-algebra. Elements of $D(\mathbb{R}_{+},S)$ are trajectories of $Y_{\epsilon}(\cdot)$.
		
	From now on, we consider only the speeded-up process $\xi_{\epsilon}(\cdot)$ and the associated order process. The metastability of $\xi_{\epsilon}(\cdot)$ is formally defined in terms of the limit when $\epsilon \to 0$ of $\mbQ_{x}^{\epsilon}, x \in \mcE^{\epsilon},$ and of the expected time the process spends in the hole $\Delta^{\epsilon}$.
	
	\subsection{Metastability}
	\label{SecMeta}
	
	Let $\mathcal{L}$ be the generator of a $S$-valued irreducible Markov process given by
	\begin{linenomath}
		\begin{equation}
			\label{gen_limit}
			(\mathcal{L}f)(i) = \sum\limits_{j \in S} \theta(i,j) [f(j) - f(i)], i \in S,		
		\end{equation}
	\end{linenomath}
	for $f: S \to \mathbb{R}$ and jump rates $\theta: S^{2} \to \mathbb{R}_{+}$ such that $\theta(i,i) = 0$. This is the continuous-time equivalent of a Markov chain in $S$ with transition probabilities given by
	\begin{equation*}
		P_{\theta}(i,j) = \frac{\theta(i,j)}{\sum_{k \in S} \theta(i,k)}.
	\end{equation*}
	In particular, starting from $i \in S$, this process will spend a random time exponentially distributed with rate $\sum_{k \in S} \theta(i,k)$ in this index, and then jump to a $j \in S$ sampled with probability function $P_{\theta}(i,\cdot)$.
	
	The Markov process in $S$ with generator $\mcL$ will represent the asymptotic dynamic of the visits to the metastable wells $\mcE^{\epsilon}_{i}, i \in S$, when $\epsilon \to 0$. Denote by $\mbQ_{i}^{\mathcal{L}}, i \in S,$ the probability measure on $D(\mathbb{R}_{+},S)$ induced by the process with generator $\mathcal{L}$ starting from $i$.
	
	The definition of metastability relies on two conditions.  The first one states that the sequence of measures $\{\mbQ_{x^{\epsilon}}^{\epsilon}\}_{\eps > 0}$ converges to $\mbQ_{i}^{\mathcal{L}}$, when $x^{\epsilon} \in \mcE^{\epsilon}_{i}$ for all $\epsilon > 0$. This means that the dynamic of the visits of the trace process to the metastable wells is asymptotically described by a Markov process. In particular, the condition $\mathfrak{C}_{\mathcal{L}}$ below implies that, as $\epsilon \to 0$, the dynamic of $Y_{\epsilon}(\cdot)$ can be approximated by that of a Markov process in $S$ with rates $\theta(i,j)$.
	
	\vspace{0.25cm}
	
	\noindent \textbf{Condition} \textit{$\mathfrak{C}_{\mathcal{L}}$. For all $i \in S$ and sequence $\{x^{\epsilon}\}_{\epsilon > 0}$, such that $x^{\epsilon} \in \mcE^{\epsilon}_{i}$ for all $\epsilon > 0$, the sequence of laws $\{\mbQ_{x^{\epsilon}}^{\epsilon}\}_{\epsilon > 0}$ converges to $\mbQ_{i}^{\mathcal{L}}$ as $\epsilon \to 0$.}
	
	\vspace{0.25cm}
	
	The second condition states that, in the limit when $\epsilon \rightarrow 0$, the process $\xi_{\epsilon}(\cdot)$ spends a negligible amount of time in $\Delta^{\epsilon}$ on each finite time interval when starting from a point in a metastable well.
	
	\vspace{0.25cm}
	
	\noindent \textbf{Condition} \textit{$\mathfrak{D}$. For all $t > 0$,}
	\begin{linenomath}
		\begin{equation*}
			\lim\limits_{\epsilon \to 0} \max_{i \in S} \sup\limits_{x \in \mcE^{\epsilon}_{i}} \boldsymbol{E}_{x}^{\epsilon} \left[\int_{0}^{t} \chi_{\Delta^{\epsilon}}(\xi_{\epsilon}(s)) \ ds\right] = 0.
		\end{equation*}
	\end{linenomath}
	
	\vspace{0.25cm}
	
	A process is metastable if both conditions are satisfied. This definition of metastability is due to \cite{beltran2010tunneling}.
	
	\begin{definition}
		The process $\xi_{\epsilon}(\cdot)$ is $\mathcal{L}$-metastable if conditions $\mathfrak{C}_{\mathcal{L}}$ and $\mathfrak{D}$ hold.
	\end{definition}
	
	Informally, the process $\xi_{\epsilon}(\cdot)$ is metastable if, when $\epsilon \to 0$, starting from a point in $\mcE^{\epsilon}$, it spends a negligible time in $\Delta^{\epsilon}$ (condition $\mathfrak{D}$) and the dynamic determining its visits to the metastable wells can be approximated by a Markov process in $S$ (condition $\mathfrak{C}_{\mathcal{L}}$). In particular, $\mathfrak{C}_{\mathcal{L}}$ implies that, when $\epsilon \to 0$, the time it takes for the original process $\eta_{\eps}(\cdot)$ to jump to the metastable well $\mcE_{j}^{\eps}$ starting from a point in $\mcE_{i}^{\eps}$ can be approximated by an exponential distribution with rate $\beta_{\epsilon}\theta_{\epsilon}(i,j)$.

    \begin{remark}
        We note that assumption \textbf{(A1.4)} allows all metastable sets to be visited in a same timescale starting from any of them by adding a suitable sub-Gaussian noise. Without this assumption, there is no hope for $\mcL$ to be irreducible under assumption \textbf{(A3)}.
    \end{remark}
	
	\subsection{Resolvent approach to metastability}
	\label{sec_resolvent}
	
	Due to the characterisation of Markov processes as solutions of martingale problems, metastability can be equivalently defined in terms of the solution of a resolvent equation associated with the process generator $\mathcal{L}_{\epsilon}$. 
	
	Fix a function $g: S \to \mathbb{R}$, and let $G \coloneqq G_{\epsilon}: I \to \mathbb{R}$ be given by
	\begin{linenomath}
		\begin{equation*}
			G(x) = \sum_{i \in S} g(i)\chi_{\mcE^{\epsilon}_{i}}(x)
		\end{equation*}
	\end{linenomath}
	for $x \in I$, that is the function that equals $g(i)$ in $\mcE^{\epsilon}_{i}, i \in S,$ and zero in $\Delta^{\epsilon}$. For $\lambda > 0$, denote by $F_{\epsilon} \coloneqq F_{\epsilon}^{\lambda,g}$ the unique solution of the resolvent equation
	\begin{linenomath}
		\begin{equation}
			\label{resolvent_eq}
			(\lambda - \mathcal{L}_{\epsilon}) F_{\epsilon} = G.
		\end{equation}
	\end{linenomath}
	It is well known (see \cite[Proposition~1.2.1]{ethier2009markov}) that the solution of \eqref{resolvent_eq} is given by the resolvent
	\begin{linenomath}
		\begin{equation}
			\label{formula_sol_reseq}
			F_{\epsilon}(x) = \boldsymbol{E}_{x}^{\epsilon} \left[\int_{0}^{\infty} e^{-\lambda t} G(\xi_{\epsilon}(t)) \ dt\right], \ x \in I.
		\end{equation}
	\end{linenomath}
	In \cite{landim2021resolvent} it is proved that the following condition is equivalent to $\mathcal{L}$-metastability.
	
	\vspace{0.25cm}
	
	\noindent \textbf{Condition} \textit{$\mathfrak{R}_{\mathcal{L}}$. For all $\lambda > 0$ and $g: S \to \mathbb{R}$, the unique solution $F_{\epsilon}$ of the resolvent equation \eqref{resolvent_eq} is asymptotically constant in each set $\mcE^{\epsilon}_{i},i \in S,$ and
		\begin{linenomath}
			\begin{equation}
				\label{sol_constant}
				\lim\limits_{\epsilon \to 0} \sup\limits_{x \in \mcE^{\epsilon}_{i}} \left|F_{\epsilon}(x) - f(i)\right| = 0, \ i \in S,
			\end{equation}
		\end{linenomath} 
		where $f: S \to \mathbb{R}$ is the unique solution of the reduced resolvent equation
		\begin{linenomath}
			\begin{equation*}
				(\lambda - \mathcal{L})f = g.
			\end{equation*}
	\end{linenomath}}
	
	\vspace{0.25cm}
	
	The main result of \cite{landim2021resolvent} is the following.
	
	\begin{theorem}
		\label{theorem_resolvent}
		A Markov process $\xi_{\epsilon}(\cdot)$ is $\mathcal{L}$-metastable if, and only if, the condition $\mathfrak{R}_{\mathcal{L}}$ is fulfilled.
	\end{theorem}
	
	We give an overview of why the condition $\mfR_{\mcL}$ is sufficient for metastability, and refer to \cite{landim2021resolvent} for more details. To see that $\mfR_{\mcL}$ implies $\mfD$, fix $\lambda > 0$, set $g(i) = 1$ for all $i \in S$ and observe that $f(i) = 1/\lambda$ is the solution of $(\lambda - \mcL)f = g$. Since $G = \chi_{\mcE^{\epsilon}}$, by \eqref{formula_sol_reseq}, for all $x \in I$,
	\begin{equation*}
		F_{\epsilon}(x) - \frac{1}{\lambda} = \boldsymbol{E}_{x}^{\epsilon} \left[\int_{0}^{\infty} e^{-\lambda t} [G(\xi_{\epsilon}(t)) - 1] \ dt\right] = - \boldsymbol{E}_{x}^{\epsilon} \left[\int_{0}^{\infty} e^{-\lambda t} \chi_{\Delta^{\epsilon}}(\xi_{\epsilon}(t)) \ dt\right].
	\end{equation*}
	By \eqref{sol_constant}, $F_{\epsilon}(x)$ converges to $f(i) = 1/\lambda$ for $x \in \mcE^{\epsilon}_{i}$ as $\epsilon \to 0$, so
	\begin{equation*}
		\lim_{\epsilon\to 0} \sup\limits_{x \in \mcE^{\epsilon}} \boldsymbol{E}_{x}^{\epsilon} \left[\int_{0}^{\infty} e^{-\lambda t} \chi_{\Delta^{\epsilon}}(\xi_{\epsilon}(t)) \ dt\right] = 0.
	\end{equation*} 
	The condition $\mfD$ then follows since, for all $t > 0$ and $x \in I$,
	\begin{align*}
		\boldsymbol{E}_{x}^{\epsilon} \left[\int_{0}^{t} \chi_{\Delta^{\epsilon}}(\xi_{\epsilon}(s)) \ ds\right] &\leq e^{\lambda t}\boldsymbol{E}_{x}^{\epsilon} \left[\int_{0}^{t} e^{-\lambda s} \chi_{\Delta^{\epsilon}}(\xi_{\epsilon}(s)) \ ds\right] \\
		&\leq e^{\lambda t}\boldsymbol{E}_{x}^{\epsilon} \left[\int_{0}^{\infty} e^{-\lambda s} \chi_{\Delta^{\epsilon}}(\xi_{\epsilon}(s)) \ ds\right].
	\end{align*}
	
	There are some technical details to prove that $\mfR_{\mcL}$ implies $\mfC_{\mcL}$ (see Propositions 4.3 and 4.5 and Lemmata 4.2 and 4.4 in \cite{landim2021resolvent}), but the main ingredient is the martingale representation of Markov processes. 
	
	For instance, for $\lambda > 0$ fixed, it can be proved (see \cite[Lemma~4.3.2]{ethier2009markov}) that \eqref{martin_speed} is a martingale with $F_{\epsilon}$ in place of $F$  if, and only if,
	\begin{align*}
		M_{\epsilon}^{\lambda}(t) &= e^{-\lambda t} F_{\epsilon}(\xi_{\epsilon}(t)) - F_{\epsilon}(\xi_{\epsilon}(0)) + \int_{0}^{t} e^{-\lambda s} \left[(\lambda - \mcL_{\epsilon})F_{\epsilon}\right](\xi_{\epsilon}(s)) \ ds
	\end{align*}
	is a martingale. Since $F_{\epsilon}$ is the solution of \eqref{resolvent_eq} it holds
	\begin{equation*}
		M_{\epsilon}^{\lambda}(t) = e^{-\lambda t} F_{\epsilon}(\xi_{\epsilon}(t)) - F_{\epsilon}(\xi_{\epsilon}(0)) + \int_{0}^{t} e^{-\lambda s} G(\xi_{\epsilon}(s)) \chi_{\mcE^{\epsilon}}(\xi_{\epsilon}(s)) \ ds
	\end{equation*}
	in which the multiplication by $\chi_{\mcE^{\epsilon}}$ inside the integral is justifiable since $G(x) = 0$ for $x \notin \mcE^{\epsilon}$.
	
	Recall the definition of $S^{\mcE^{\epsilon}}(t)$ (cf. \eqref{gen_inverse}). Since $S^{\mcE^{\epsilon}}(t)$ is a stopping time with respect to $\mfF_{t}$ (see \cite[Lemma~7.2]{landim2019metastability}), the process $\hat{M}_{\epsilon}^{\lambda}(t) = M(S^{\mcE^{\epsilon}}(t))$ is a martingale. By the definition of the trace process (cf. \eqref{trace_process}) we conclude that:
	\begin{equation*}
		\hat{M}_{\epsilon}^{\lambda}(t) = e^{-\lambda S^{\mcE^{\epsilon}}(t)} F_{\epsilon}(\xi_{\epsilon}^{\mcE^{\epsilon}}(t)) - F_{\epsilon}(\xi_{\epsilon}^{\mcE^{\epsilon}}(0)) + \int_{0}^{S^{\mcE^{\epsilon}}(t)} \ e^{-\lambda s} \ G(\xi_{\epsilon}(s)) \ \chi_{\mcE^{\epsilon}}(\xi_{\epsilon}(s)) \ ds
	\end{equation*}
	is a martingale. Finally, with the change of variables $r = T^{\mcE^{\epsilon}}(s)$ (cf. \eqref{time_spent}) in the integral, that can be performed due to the multiplication by $\chi_{\mcE^{\epsilon}}$, we conclude that
	\begin{equation*}
		\hat{M}_{\epsilon}^{\lambda}(t) = e^{-\lambda S^{\mcE^{\epsilon}}(t)} F_{\epsilon}(\xi_{\epsilon}^{\mcE^{\epsilon}}(t)) - F_{\epsilon}(\xi_{\epsilon}^{\mcE^{\epsilon}}(0)) + \int_{0}^{t} \ e^{-\lambda S^{\mcE^{\epsilon}}(r)} \ G(\xi_{\epsilon}^{\mcE^{\epsilon}}(r)) \ dr.
	\end{equation*}
	Since $F_{\epsilon}$ satisfies \eqref{sol_constant}, by definition of $Y_{\epsilon}(\cdot), G$ and Lemma 4.4 in \cite{landim2021resolvent}, that enables replacing $S^{\mcE^{\epsilon}}(t)$ with $t$ in the limit of expectations of $\hat{M}_{\epsilon}^{\lambda}(t)$ when $\epsilon \to 0$, we conclude that
	\begin{equation*}
		\hat{M}_{\epsilon}^{\lambda}(t) = e^{-\lambda t} f(Y_{\epsilon}(t)) - f(Y_{\epsilon}(0)) + \int_{0}^{t} \ e^{-\lambda r} \ [(\lambda - \mcL)f](Y_{\epsilon}(r)) \ dr + R_{\epsilon}(t)
	\end{equation*}
	is a martingale for a residual function $R_{\epsilon}(t)$. Lemmata 4.1 and 4.4 in \cite{landim2021resolvent} imply that for all $t > 0$
	\begin{equation*}
		\lim\limits_{\epsilon \to 0} \sup_{x \in \mcE^{\epsilon}} \boldsymbol{E}_{x}^{\epsilon}\left[R_{\epsilon}(t)\right] = 0,
	\end{equation*}
	so in the limit when $\epsilon \to 0$, $\hat{M}_{\epsilon}^{\lambda}(t)$ is a martingale. By the uniqueness of solutions of martingale problems in finite sets, and the fact that the sequence of measures $\mbQ_{i}^{\epsilon}$ is tight (see \cite[Proposition~4.3]{landim2021resolvent}), we conclude that the measure induced by $Y_{\epsilon}(\cdot)$ converges to that induced by the Markov process with operator $\mcL$, that is $\mfC_{\mcL}$.
	
	That $\mfR_{\mcL}$ is necessary to metastability is proved in \cite[Section~4.2]{landim2021resolvent}.
	
	\section{Sufficient conditions for the metastability of random mixing maps}
	\label{Sec_Suf}
	
	In order to prove $\mathfrak{R}_{\mathcal{L}}$ one first shows that, for fixed $g$, the solution $F_{\epsilon}(x)$ of the resolvent equation \eqref{resolvent_eq} is asymptotically constant in each metastable well $\mcE^{\epsilon}_{i}$, and then figure out the generator $\mcL$ of the limit Markov process. We break down $\mathfrak{R}_{\mathcal{L}}$ into these two conditions.
	
	\vspace{0.25cm}
	
	\noindent \textbf{Condition} \textit{$\mathfrak{R}^{(1)}$. The solution $F_{\epsilon}$ of the resolvent equation \eqref{resolvent_eq} is asymptotically constant on each set $\mcE^{\epsilon}_{i}$:
		\begin{equation*}
			\lim\limits_{\epsilon \to 0} \sup\limits_{x,y \in \mcE^{\epsilon}_{i}} \left|F_{\epsilon}(x) - F_{\epsilon}(y)\right| = 0, \ i \in S.
	\end{equation*}}
	
	\vspace{0.25cm}
	
	Recall that $\mu_{\epsilon}$ is the unique ACIM of $\xi_{\epsilon}(\cdot)$ that has a density function $p_{\epsilon}$ (cf. \textbf{(A4)}).
	
	\vspace{0.25cm}
	
	\noindent \textbf{Condition} \textit{$\mathfrak{R}^{(2)}_{\mathcal{L}}$. Let $f_{\epsilon}: S \mapsto \mbR$ be the function given by
		\begin{equation}
			\label{f_eps}
			f_{\epsilon}(i) = \frac{1}{\mu_{\epsilon}(\mcE^{\epsilon}_{i})} \int_{\mcE^{\epsilon}_{i}} F_{\epsilon}(x) \, p_{\epsilon}(x) \ dx, \ i \in S.
		\end{equation}
		For all $i \in S$,
		\begin{equation*}
			\lim\limits_{\epsilon \to 0} f_{\epsilon}(i) = f(i)
		\end{equation*}
		where $f: S \to \mathbb{R}$ is the unique solution of the reduced resolvent equation
		\begin{equation*}
			(\lambda - \mathcal{L})f = g.
		\end{equation*}
	}
	
	\vspace{0.25cm}
	
	Clearly, conditions $\mathfrak{R}^{(1)}$ and $\mathfrak{R}^{(2)}_{\mathcal{L}}$ imply $\mathfrak{R}_{\mathcal{L}}$.
	
	\subsection{A sufficient condition for $\mathfrak{R}^{(1)}$}
	\label{SecR1}
	
	A sufficient condition for $\mathfrak{R}^{(1)}$ can be derived based on the mixing of a process that behaves as $\xi_{\epsilon}(\cdot)$ in a \textit{larger metastable well}.
	
	Denote by $\mcV_{i}^{\epsilon}$, $i \in S$, a set which contains $\mcE^{\epsilon}_{i}$, but is disjoint with the other wells: $\mcE^{\epsilon}_{i} \subset \mcV_{i}^{\epsilon}$ and $\mcV_{i}^{\epsilon} \cap \check{\mcE}^{\epsilon}_{i} = \emptyset$. For simplicity, one can think of this set as $\mcV_{i}^{\epsilon} = I_{i} = \Delta^{\epsilon}_{i} \cup \mcE^{\epsilon}_{i}$, but this is not strictly necessary.
	
	For each $A \in \mfB_{I}$ let
	\begin{equation}
		\label{htime}
		H(A) = \inf\{t \geq 0: \xi_{\epsilon}(t) \in A\}
	\end{equation}
	be the hitting time of $A$ by $\xi_{\epsilon}(\cdot)$. Let $\{\tilde{\xi}^{i}_\epsilon(t) : t \ge 0\}$ be a continuous-time Markov process on $\mcV_{i}^{\epsilon}$ such that there exists a coupling between $\tilde{\xi}^{i}_{\epsilon}(\cdot)$ and $\xi_\epsilon(\cdot)$ with
	\begin{equation}
		\label{coupling}
		\inf_{x \in \mcE^{\epsilon}_{i}} \boldsymbol{P}_{x}^{\epsilon}\left[\tilde{\xi}^{i}_{\epsilon}(t) =  \xi_\epsilon(t) \Big| t < H((\mcV_{i}^{\epsilon})^{c})\right] = 1
	\end{equation}
	for all $t > 0$. Formally, $\tilde{\xi}^{i}_{\epsilon}(\cdot)$ and $\xi_\epsilon(\cdot)$ are defined in a same probability space $(\Omega,\mcF,\mbP)$ and for each\footnote{Actually, this condition should hold for all $\omega \in \Omega$ with the possible exception of those in a $\mbP$-null subset. We omit this distinction throughout the paper without loss of generality.} $\omega \in \Omega$ it holds $\tilde{\xi}^{i}_{\epsilon}(\omega,t) = \xi_\epsilon(\omega,t)$ for all $t$ satisfying $t < H((\mcV_{i}^{\epsilon})^{c})(\omega)$. This implies that the processes $\tilde{\xi}^{i}_{\epsilon}(\cdot)$ and $\xi_\epsilon(\cdot)$ starting from a point in $\mcE^{\epsilon}_{i}$ behave exactly the same until $\xi_\epsilon(\cdot)$ reaches $(\mcV_{i}^{\epsilon})^{c}$. We assume that the process $\tilde{\xi}^{i}_{\epsilon}(\cdot)$ is induced by a random map $\tilde{T}_{\epsilon}^{i}: \Omega \times \mcV_{i}^{\epsilon} \mapsto \mcV_{i}^{\epsilon}$ that is a perturbation of the restriction of the original map $T$ to $\mcV_{i}^{\epsilon}$.
    
    Figure \ref{f3} presents an example of map $\tilde{T}_{\epsilon}^{i}$ for fixed $\omega$ when $T$ is the map in Figure \ref{f1} and $T_{\epsilon}$ is given by additive noise, where we can see that the maps $T_{\epsilon}(\omega,\cdot)$ and $\tilde{T}_{\epsilon}^{1}(\omega,\cdot)$ coincide for all $x$ such that $T_{\epsilon}(\omega,x) \in I_{1} = (0,1/2)$. If this is true with probability one over $\omega \in \Omega$, then \eqref{coupling} holds. We note in Figure \ref{f3} that $\tilde{T}_{\epsilon}^{1}$ could have been defined in any desired way for $x$ with $T_{\epsilon}(\omega,x) \notin I_{1}$, so we can define it in a way so the sufficient condition for $\mathfrak{R}^{(1)}$ to be deduced below is fulfilled.
	
	Denote by $d^i_{\rm TV}(\mu, \nu) \coloneqq d^{i,\epsilon}_{\rm TV}(\mu, \nu)$ the total variation distance between two probability measures $\mu$, $\nu$ on $\mcV_{i}^{\epsilon}$:
	\begin{equation}\label{dtv}
		d^i_{\rm TV} (\mu, \nu) \;=\; \frac{1}{2}\, \sup_J
		\Big|\, \int_{\mcV_{i}^{\epsilon}} J(x)\, d\mu(x) \,-\,
		\int_{\mcV_{i}^{\epsilon}} J(x)\, d\nu(x) \, \Big|\;,
	\end{equation}
	where the supremum is carried over all measurable functions	$J: \mcV_{i}^{\epsilon} \to \bb R$ bounded by $1$, i.e., $\lVert J \rVert_{\infty} < 1$.
	
	Assume that the process $\tilde{\xi}^{i}_{\epsilon}(\cdot)$ is ergodic. Denote by $\{\tilde{\ms{P}}^{i}_\epsilon (t) : t\ge 0\}$ its semigroup (cf. \eqref{semigroup}), by $\tilde{\mu}_{\epsilon}^{i}$ its invariant measure, and by $t^{\epsilon,i}_{\rm mix} (\varsigma)$, for $0 < \varsigma < 1$, its mixing time starting from $\mcE^{\epsilon}_{i}$:
	\begin{equation}
        \label{def_tmix}
		t^{\epsilon,i}_{\rm mix} (\varsigma) \;=\; \inf\left\{t> 0 : \sup_{x\in \mcE^{\epsilon}_{i}} d^i_{\rm TV} (\delta_x \tilde{\ms P}^{i}_\epsilon(t) \,,\,  \tilde{\mu}_{\epsilon}^{i}) \,\le\, \varsigma \,\right\}\;
	\end{equation}
	recalling that $\delta_x \tilde{\ms P}^{i}_\epsilon(t)$ is the probability measure induced on $\mfB_{\mcV_{i}^{\epsilon}}$ by $\tilde{\xi}^{i}_\epsilon(t)$ when $\tilde{\xi}^{i}_\epsilon(0) = x$ (cf. \eqref{semigroup_measure}).
	
	The following mixing property implies condition	$\mf R^{(1)}$.
	
	\vspace{0.25cm}
	
	\noindent \textit{\textbf{\bf Condition $\mf M$.} For all $i \in S$, the process $\xi_\epsilon(\cdot)$ starting from a set $\mc{E}_{i}^{\epsilon}$ cannot escape from the set $\mcV_{i}^{\epsilon}$ within a timescale $\mb h_\epsilon \ll 1$: 
		\begin{equation}
			\label{23}
			\lim_{\epsilon\to 0} \sup_{x \in \mc{E}_{i}} \mb{P}^\epsilon_{x}\,[\, H((\mcV_{i}^{\epsilon})^{c})
			\le \mb h_{\epsilon} \,]\;=\;0\;.
		\end{equation}
		Furthermore, there exists a $\varsigma_{0} > 0$ independent of $\epsilon$ such that for all $i \in S$ and $0 < \varsigma < \varsigma_{0}$ fixed,
		\begin{equation}
			\label{33}
			t^{\epsilon,i}_{\rm mix} (\varsigma) \;\le\; \mb h_\epsilon
		\end{equation}
		for all $\epsilon$ sufficiently small.}
	
	\vspace{0.25cm}
	
	Condition $\mf M$ states that the process $\tilde{\xi}^{i}_\epsilon(\cdot)$ attains the equilibrium, i.e., mixes, in a timescale that is lesser than that the process $\xi_{\epsilon}(\cdot)$ takes to leave $\mcV_{i}^{\epsilon}$. Since $\xi_{\epsilon}(t)$ and $\tilde{\xi}_{\epsilon}^{i}(t)$ are equal for $t < H((\mcV_{i}^{\epsilon})^{c})$, and $\mb h_{\epsilon} \ll H((\mcV_{i}^{\epsilon})^{c})$, the process $\xi_{\epsilon}(\cdot)$ attains a \textit{false-equilibrium} in $\mcE^{\epsilon}_{i}$ before jumping to another metastable well. If the original, not speeded-up processes were considered, then the condition $\mf M$ would mean that the restricted process has a mixing time of order lesser than $\beta_{\epsilon} \mb h_\epsilon \ll \beta_{\epsilon}$.
	
	The next proposition shows that $\mf M$ implies $\mf R^{(1)}$. This is a modification of Proposition 6.7 in \cite{landim2021resolvent} which considered the case of Markov jump processes with countable state spaces.
	
	\begin{proposition}
		\label{p03}
		If the mixing property $\mf M$ is satisfied, then the condition $\mf R^{(1)}$ holds.
	\end{proposition}
	
	Under assumption \textbf{(A3)}, considering $\mcV_{i}^{\epsilon} = I_{i}$, condition \eqref{23} holds if the supremum over $x \in I_{i}$ of the probability of jumping from $x$ to $I_{i}^{c}$ (cf. \eqref{barrier}) decreases \textit{fast enough} to zero.
	
	\begin{proposition}
		\label{prop_not_jump}
		For $i \in S$ and $\epsilon > 0$, denote
		\begin{align}
            \label{qeps}
			q_{\epsilon} \coloneqq \max\limits_{i \in S} \sup\limits_{x \in I_{i}} \mbP\left(T_{\epsilon}(x) \in I_{i}^{c}\right).
		\end{align}
		If
		\begin{align}
			\label{hyp_not_jump}
			\lim_{\epsilon\to 0} \ q_{\epsilon} \, \beta_{\epsilon} \, \mb h_{\epsilon} = 0 & & \text{ and } & & \lim_{\epsilon\to 0} \beta_{\epsilon} \, \mb h_{\epsilon} = \infty,
		\end{align}
		then \eqref{23} holds with $\mcV_{i}^{\epsilon} = I_{i}$.
	\end{proposition}
	
	\subsection{Mixing time of random transformations of dynamical systems}
	
	The sufficient condition $\mf M$ may be deduced by representing the process $\tilde{\xi}^{i}_{\epsilon}(\cdot)$ as a random perturbation of a dynamical system and analysing the spectral properties of the respective transfer operator. To this end, we make the following assumption about the smoothness of the maps:
    \begin{itemize}
        \item[\textbf{(A6)}] For all $i \in S$, the maps $T|_{I_{i}}$ and $\tilde{T}_{\epsilon}^{i}(\omega,\cdot)$, for all $\omega \in \Omega$ and $\epsilon > 0$, are $C^{2}$ at all but a finite number of points in $I_{i}$, but can be extended to a $C^{2}$ function in a neighbourhood of these points.
    \end{itemize}	
	In this section, we assume that $i \in S$ is fixed and $\mcV_{i}^{\epsilon} = I_{i}$. To easy notation, we denote simply $\tilde{T}_{\epsilon} \coloneq \tilde{T}_{\epsilon}^{i}$ and $\tilde{\xi}_{\epsilon} \coloneqq \tilde{\xi}_{\epsilon}^{i}$, and omit $i$ from the superscript of operators. 
    
	We now introduce the Banach space of functions, which we will use throughout the paper. Define the total variation of $F: I_{i} \mapsto \mbR$ as
	\begin{align*}
		|F|_{TV} = \sup\left\{\sum_{\ell = 1}^{n} |F(x_{\ell}) - F(x_{\ell-1})|: n \geq 1, x_{0} \leq x_{1} \leq \cdots \leq x_{n}, x_{\ell} \in I_{i}\right\}.
	\end{align*}
	The space $BV(I_{i})$ of functions $F: I_{i} \mapsto \mbR$ with bounded variation is the Banach space with norm
	\begin{align}
		\label{BV_norm}
		\lVert F \rVert:=\lVert F \rVert_{BV} = |F|_{TV} + \lVert F \rVert_{1}.
	\end{align}
	An element of $BV(I_{i})$ is technically an equivalence class of functions that are equal in all points of $I_{i}$, except on a Lebesgue null set\footnote{To be precise, in the definition of $|F|_{TV}$ one should take the infimum over all functions in the equivalence class of $F$.}. Furthermore, 
    \begin{equation}
        \label{BV_inf}
        \lVert F \rVert_{\infty} \leq \max\{1,(\text{Leb}(I_{i}))^{-1}\} \lVert F \rVert_{BV}\, ,
    \end{equation}
    so in particular $BV(I_{i}) \subset L^{\infty}(I_{i})$. Indeed, for any $x,y \in I_{i}$, $|F|(x) - |F|(y) \leq |F|_{TV}$ and \eqref{BV_inf} follows by integrating both sides in $y \in I_{i}$ and taking the essential supremum over $x \in I_{i}$.
	
	The transfer operator $P$ of map $T$ and $\tilde{P}_{\epsilon,\omega} \coloneqq \tilde{P}^{i}_{\epsilon,\omega}$ of map $\tilde{T}_{\epsilon}(\omega,\cdot)$ for $\omega \in \Omega$ fixed act on $F \in BV(I_{i})$ as
	\begin{align}
		\label{transfer_oper}
		(PF)(x) = \sum_{y \in T^{-1}(x)} \frac{F(y)}{|T^{\prime}(y)|} & & \text{ and } & & (\tilde{P}_{\epsilon,\omega}F)(x) = \sum_{y \in \tilde{T}_{\epsilon}^{-1}(\omega,x)} \frac{F(y)}{|\tilde{T}_{\epsilon}^{\prime}(\omega,y)|}
	\end{align}
	for $x \in I_{i}$. For all $\omega \in \Omega,$ $F\in BV(I_{i})$ and $G\in L^1(I_i),$ the transfer operator satisfies
	\begin{align}
		\label{cond_transferOmega}
		\int_{I_{i}} F(x) \, G(\tilde{T}_{\epsilon}(\omega,x)) \, dx = \int_{I_{i}} (\tilde{P}_{\epsilon,\omega}F)(x) \, G(x) \, dx.
	\end{align}
	In the following we may simply write $F\in BV,$ assuming from the context the set where the total variation is computed.
	
	We denote by $\tilde{\mcL}_{\epsilon}^{i}$ the generator of $\tilde{\xi}_{\epsilon}^{i}(\cdot)$ and assume that it acts on functions $F \in BV$ as
	\begin{align}
		\label{generator_restricted}
		(\tilde{\mcL}_{\epsilon}^{i}F)(x) = \beta_{\eps} \int_{I_{i}} \tilde{\rho}_{\epsilon}(x,y)[F(y) - F(x)] \, dy
	\end{align}
	for $x \in I_{i}$ in which $\tilde{\rho}_{\epsilon}(x,\cdot) \coloneqq \tilde{\rho}_{\epsilon}^{i}(x,\cdot)$ is a probability density function for all $x \in I_{i}$. We assume that $\tilde{\rho}_{\epsilon}(x,\cdot) \in BV$ for all $x \in I_{i}$ and that the ACIM $\tilde{\mu}_{\epsilon}^{i}$ of $\tilde{\xi}_{\epsilon}(\cdot)$ has a probability density function $\tilde{p}_{\epsilon} \coloneq \tilde{p}_{\epsilon}^{i} \in BV$.
	
	Let $\tilde{\tau}_{0} = 0$ and
	\begin{equation*}
		\tilde{\tau}_{n} = \inf\{t > \tilde{\tau}_{n-1}: \tilde{\xi}_{\epsilon}(t) \neq \tilde{\xi}_{\epsilon}(\tilde{\tau}_{n-1})\}
	\end{equation*}
	for $n \geq 1$ be the jumping times of $\tilde{\xi}_{\epsilon}(\cdot)$. Define by $\tilde{X}_{n}^{\epsilon} \coloneqq \tilde{\xi}_{\epsilon}(\tilde{\tau}_{n})$ for $n \geq 0$ the embedded Markov chain of $\tilde{\xi}_{\epsilon}(\cdot)$. This is the Markov chain in $I_{i}$ with transition density $\tilde{\rho}_{\epsilon}(x,y)$. Associate to $\tilde{\rho}_{\epsilon}$ the operator $\tilde{P}_{\epsilon} \coloneqq \tilde{P}_{\epsilon}^{i}$ that acts on $F \in BV$ as
	\begin{equation*}
		(\tilde{P}_{\epsilon}F)(x) = \int_{I_{i}} F(y) \tilde{\rho}_{\epsilon}(y,x) \, dy\, ,
	\end{equation*}
	that is the adjoint of the Markov operator of $\tilde{X}_{n}^{\epsilon}$. 
	
	From \eqref{cond_transferOmega}, it follows that, for all $F \in BV$, 
	\begin{equation}
		\label{randomTrans}
		(\tilde{P}_{\epsilon}F)(x) = \int_{\Omega} (\tilde{P}_{\epsilon,\omega}F)(x) \, d\mbP(\omega)
	\end{equation}
	for almost every $x \in I_{i}$ under the Lebesgue measure. This is true since, by multiplying both sides by a bounded positive measurable function $G$ and integrating, the left-hand side of \eqref{randomTrans} can then be written as
	\begin{align}
		\label{randomTrans2} \nonumber
		\int_{I_{i}} \int_{I_{i}} G(x) F(y) \tilde{\rho}_{\epsilon}(y,x) \, dy \, dx &= \int_{\Omega} \int_{I_{i}}  G(\tilde{T}_{\epsilon}(\omega,y)) F(y) \, dy \, d\mbP(\omega)\\
		&= \int_{\Omega} \int_{I_{i}} G(x) (\tilde{P}_{\epsilon,\omega}F)(x) \, dx \, d\mbP(\omega)
	\end{align}
	in which the last equality is due to \eqref{cond_transferOmega}. The almost everywhere equality \eqref{randomTrans} follows since \eqref{randomTrans2} holds for all $G$. In particular, it follows from \eqref{randomTrans} that
	\begin{align}
		\label{norm_P_Ptilde}
		\lVert \tilde{P}_{\epsilon} F \rVert \leq \int_{\Omega} \lVert \tilde{P}_{\epsilon,\omega}F \rVert \, d\mbP(\omega)
	\end{align}
	for $F \in BV$.
	
	Fix $x \in I_{i}$ and $n \geq 1$. By definition (cf. \eqref{semigroup_measure}), $\delta_x \tilde{\ms P}^{i}_\epsilon(\tilde{\tau}_{n})$ is the measure induced on $\mfB_{I_{i}}$ by $\tilde{X}_{n}^{\epsilon}$ when $\tilde{X}_{0}^{\epsilon} = x$. The probability density function of this measure, that we denote by $\tilde{p}_{\epsilon}^{n} \coloneqq \tilde{p}_{\epsilon,x}^{n}$, satisfies
	\begin{equation}
		\label{density_rt}
		\tilde{p}^{n}_{\epsilon}(y) = (\tilde{P}_{\epsilon}^{n-1}\tilde{\rho}_{\epsilon}(x,\cdot))(y)
	\end{equation}
	for $y \in I_{i}$. Clearly, $\tilde{p}^{1}_{\epsilon}(y) = \tilde{\rho}_{\epsilon}(x,y) = (\tilde{P}_{\epsilon}^{0}\tilde{\rho}_{\epsilon}(x,\cdot))(y)$ and for $n \geq 1$
	\begin{align*}
		\tilde{p}^{n}_{\epsilon}(y) = \int_{I_{i}} \tilde{p}^{n-1}_{\epsilon}(z) \tilde{\rho}_{\epsilon}(z,y) \, dz = (\tilde{P}_{\epsilon} \, \tilde{p}^{n-1}_{\epsilon})(y)
	\end{align*}
	and the equality \eqref{density_rt} follows by induction.
	
	Recalling the definition of $d_{TV}^{i}$ in \eqref{dtv}, we conclude that for $n \geq 1$
	\begin{align}
		\label{bound_dTV_BV} \nonumber
		d_{TV}^{i}(\delta_x \tilde{\ms P}^{i}_\epsilon(\tilde{\tau}_{n}),\tilde{\mu}_{\epsilon}^{i}) &= \frac{1}{2} \sup_{J} \left|\int_{I_{i}} J(y) \,(\tilde{P}_{\epsilon}^{n-1}\tilde{\rho}_{\epsilon}(x,\cdot))(y) \, dy - \int_{I_{i}} J(y) \, \tilde{p}_{\epsilon}(y) \, dy \right|\\ \nonumber
		&\leq \lVert \tilde{P}_{\epsilon}^{n-1}\tilde{\rho}_{\epsilon}(x,\cdot) - \tilde{p}_{\epsilon} \rVert_{1}\\
		&\leq \left\lVert \tilde{P}_{\epsilon}^{n-1}\tilde{\rho}_{\epsilon}(x,\cdot) - \tilde{p}_{\epsilon} \int_{I_{i}}\tilde{\rho}_{\epsilon}(x,z) \, dz  \right\rVert.
	\end{align}
	in which the last equality follows since the BV norm is greater than the $L_{1}$ norm by definition (cf. \eqref{BV_norm}) and the integral multiplying $\tilde{p}_{\epsilon}$ equals one.

	Define, for $0 < \varsigma < 1$, the mixing time of $\{\tilde{X}_{n}^{\epsilon}: n \geq 1\}$ starting from $\mcE_{i}^{\epsilon}$ as
	\begin{equation}
        \label{def_nmix}
		n^{\epsilon,i}_{\rm mix} (\varsigma) \;=\; \min\left\{n \geq 0 : \sup_{x \in \mcE_{i}^{\epsilon}} d^i_{\rm TV} (\delta_x \tilde{\ms P}^{i}_\epsilon(\tilde{\tau}_{n}) \,,\,  \tilde{\mu}_{\epsilon}^{i}) \,\le\, \varsigma \,\right\}.
	\end{equation}	
	The next lemma relates $t^{\epsilon,i}_{\rm mix} (\varsigma)$ with $n^{\epsilon,i}_{\rm mix} (\varsigma)$, so bounds for \eqref{bound_dTV_BV} imply bounds for $t^{\epsilon,i}_{\rm mix} (\varsigma)$.
	
	\begin{lemma}
		\label{lemma_Mix_nt}
		There exists a $\varsigma_{0} > 0$, independent of $\epsilon$, such that, for all $0 < \varsigma < \varsigma_{0}$, there exists a constant $C_{\varsigma} > 0$, also independent of $\epsilon$, such that
		\begin{equation*}
			t^{\epsilon,i}_{\rm mix} (\varsigma) \leq \frac{C_{\varsigma}}{\beta_{\epsilon}} \, n^{\epsilon,i}_{\rm mix} (\varsigma)
		\end{equation*}		
        for all $i \in S$ and $\eps$ small enough.
	\end{lemma}
	
	We have proved so far that to establish \eqref{33}, it is enough to properly bound \eqref{bound_dTV_BV} over $x \in \mcE_{i}^{\epsilon}$. We now show how bounds for \eqref{bound_dTV_BV} are a consequence of the spectral properties of the transfer operators $\tilde{P}_{\epsilon,\omega}$ when $T$ is strongly mixing:
	\begin{itemize}
		\item[] \textbf{(A7)} For all $i \in S$, the restriction $T|_{I_{i}}$ of $T$ to $I_{i}$ verifies the Lasota-Yorke (LY) inequality: there exists $0<\gamma<1$ and $C,D>0$ such that, for all $n \ge 1$ and $F \in BV(I_{i})$,
        \begin{equation}\label{LLYY}
        \lVert P^nF \rVert \leq C \gamma^{n} \lVert F \rVert + D \lVert F \rVert_{1}.
        \end{equation}
        Moreover $T|_{I_{i}}$  has a unique ACIM $\mu_{i}$ with density $p_{i} \in BV(I_{i})$ and it is strongly mixing, that is, there exists $0 < \lambda_{i} < 1$ and $\Lambda_{i} \geq 1$ such that for all $F \in BV(I_i)$ and $n \geq 1:$\footnote{By the compact embedding of $BV$ into $L^1$ and the inequality \eqref{LLYY}, the uniqueness and the mixing of the ACIM follow whenever $1$ is the only eigenvalue of $P$ in the unit circle.}
		\begin{equation}
			\label{sm}
			\left\lVert P^{n}F - p_{i} \int_{I_{i}} F(z) \, dz \right\rVert \leq \Lambda_{i} \lambda_{i}^{n}.
		\end{equation}
	\end{itemize}
	
	The next theorem allows to bound \eqref{bound_dTV_BV} by $\Lambda_{i} \lambda^{n-1}_{i}$ given in \eqref{sm} and is a consequence of the perturbation theorem of Keller and Liverani \cite{KL2} which asserts the spectral stability of $P$ under small perturbations.

	\begin{theorem}
		\label{main_theorem}
		If it holds for $i \in S$:
		\begin{itemize}
			\item[(a)] There exists $0 < \gamma < 1$ and $C, D > 0$ such that for all $\epsilon > 0$ and $n \geq 1$
			\begin{equation*}
				\lVert \tilde{P}_{\epsilon}^n F \rVert \leq C \gamma^{n} \lVert F \rVert + D \lVert F \rVert_{1}
			\end{equation*}
			for all $F \in BV(I_{i})$ and
			\item[(b)] 
			\begin{equation*}
				\lim\limits_{\epsilon \to 0} \sup\limits_{\lVert F \rVert \leq 1} \left|\int_{I_{i}} (\tilde{P}_{\epsilon} - P)F(x) \, dx \right| = 0,
			\end{equation*}
			then
			\begin{equation}
				\label{eq_Theorem_spectral}
				\left\lVert \tilde{P}_{\epsilon}^{n}F - \tilde{p}_{\epsilon} \int_{I_{i}}F(z) \, dz  \right\rVert \leq \Lambda_{i} \, \lambda^{n}_{i}
			\end{equation}
			for all $F \in BV(I_{i})$ in which $\lambda_{i}$ and $\Lambda_{i}$ are the constants in \eqref{sm}.
		\end{itemize}		
	\end{theorem}

	It follows from Theorem \ref{main_theorem} that the condition \eqref{33} of $\mfM$ holds. Indeed, for $0 < \varsigma < \varsigma_{0}$ fixed,
	\begin{align*}
		\Lambda_{i}\lambda^{n-1}_{i} \leq \varsigma \iff n \geq \frac{\log \varsigma - \log \Lambda_{i}}{\log \lambda_{i}} + 1
	\end{align*}
	and, when the above holds, it follows from \eqref{bound_dTV_BV} and \eqref{eq_Theorem_spectral} that
	\begin{align*}		
		\sup_{x \in \mcE_{i}^{\epsilon}} d^i_{\rm TV} (\delta_x \tilde{\ms P}^{i}_\epsilon(\tilde{\tau}_{n}) \,,\,  \tilde{\mu}_{\epsilon}^{i}) \leq \varsigma
	\end{align*}
	from which we conclude that
	\begin{equation}
		\label{bound_tmix}
		n^{\epsilon,i}_{\rm mix} (\varsigma) \leq \frac{\log \varsigma- \log \Lambda_{i}}{\log \lambda_{i}} + 1 \implies t^{\epsilon,i}_{\rm mix} (\varsigma) \leq \frac{C_{\varsigma}}{\beta_{\epsilon}} \left[\frac{\log \varsigma - \log \Lambda_{i}}{\log \lambda_{i}} + 1\right]
	\end{equation}
	where the last implication is due to Lemma \ref{lemma_Mix_nt}. Taking
	\begin{equation}
        \label{he}
		\bs{h}_{\epsilon} = \frac{a_{\epsilon}}{\beta_{\epsilon}}
	\end{equation}
	for any sequence $a_{\epsilon}$ satisfying $1 \ll a_{\epsilon} \ll q_{\epsilon}^{-1} \wedge \beta_{\epsilon}$, it follows from \eqref{bound_tmix} that \eqref{33} holds and from Proposition \ref{prop_not_jump} (cf. \eqref{hyp_not_jump}) that \eqref{23} holds, and hence $\mfM$ is in force.
	
	In summary, we have proved a general result about the metastability of randomly perturbed maps.
	
	\begin{theorem}
		\label{theorem_existence}
		If there exist random maps $\tilde{T}^{i}_{\epsilon}: \Omega \times I_{i} \mapsto I_{i}$ for $i \in S$ such that $T_{\epsilon}(\omega,x) = \tilde{T}^{i}_{\epsilon}(\omega,x)$ for all $x \in I_{i}$ with $T_{\epsilon}(\omega,x) \in I_{i}$, and the conditions (a) and (b) in Theorem \ref{main_theorem} hold, then $\mfR^{(1)}$ is in force. 
	\end{theorem}
	
	\begin{remark}
		\label{rem_LY}
		In view of \eqref{norm_P_Ptilde}, in order to establish (a), it is enough to show for instance that, for each $\epsilon > 0$ and $\omega \in \Omega$, there exist $0 \leq \gamma_{\epsilon,\omega},D_{\epsilon,\omega} < \infty$ such that
			\begin{align*}
				\lVert \tilde{P}_{\epsilon,\omega}F \rVert \leq \gamma_{\epsilon,\omega} \, \lVert F \rVert + D_{\epsilon,\omega} \, \lVert F \rVert_{1}
			\end{align*}
			for all $F \in BV$ with
			\begin{align}
				\label{exp_LY}
				\limsup\limits_{\epsilon \to 0} \int_{\Omega} \gamma_{\epsilon,\omega} \, d\mbP(\omega) < 1 & & \text{ and } & & \limsup\limits_{\epsilon \to 0} \int_{\Omega} D_{\epsilon,\omega} \, d\mbP(\omega) < \infty.
			\end{align}
            More generally, for $n > 1$,
            \begin{align*}
                \lVert \tilde{P}_{\eps}^{n}F \rVert &\leq \int_{\Omega^{n}} \lVert (\tilde{P}_{\epsilon,\omega_{n}} \circ \cdots \circ \tilde{P}_{\epsilon,\omega_{1}})F \rVert \, d\mbP^{\otimes}(\bar{\omega})\, ,
            \end{align*}
            in which $\bar{\omega} = (\omega_{1},\dots,\omega_{n})$ and $\mbP^{\otimes}$ is the $\mbP$-product measure in $\Omega^{n}$. Therefore, condition (a) holds if
            \begin{align*}
                \lVert (\tilde{P}_{\epsilon,\omega_{n}} \circ \cdots \circ \tilde{P}_{\epsilon,\omega_{1}})F \rVert \leq \gamma_{\eps,\bar{\omega}} \, \lVert F \rVert + D_{\eps,\bar{\omega}} \, \lVert F \rVert_{1}
            \end{align*}
            for all $F \in BV$ with
            \begin{align}
				\label{exp_LY2}
				&\limsup\limits_{\epsilon \to 0} \int_{\Omega^{n}} \gamma_{\eps,\bar{\omega}} \, d\mbP^{\otimes}(\bar{\omega}) < 1 & &  \text{ and } & &
                &\limsup\limits_{\epsilon \to 0} \int_{\Omega^{n}} D_{\eps,\bar{\omega}} \, d\mbP^{\otimes}(\bar{\omega}) < \infty.
			\end{align}            
	\end{remark}	
	
	\subsection{A sufficient condition for $\mathfrak{R}^{(2)}_{\mcL}$}
	
	If the condition $\mathfrak{R}^{(1)}$ holds, then $\mathfrak{R}^{(2)}_{\mcL}$ will be a consequence of \textbf{(A5)} (cf. \eqref{cm}) and the following condition. 
	
	For $A \in \mfB_{I}$, recall the definition of the hitting time $H(A)$ of $A$ in \eqref{htime} and let
	\begin{equation*}
		H^{+}(A) = \inf\{t \geq \tau_{1}: \xi_{\epsilon}(t) \in A\} \text{ in which } \tau_{1} = \inf\{t \geq 0: \xi_{\epsilon}(t) \neq \xi_{\epsilon}(0)\}
	\end{equation*}
	be the first visit to $A$ after at least one jump. Recall the definition of $\check{\mcE}_{i}^{\epsilon}$ (cf. \eqref{cwells}) and let 
	\begin{align}
		\label{asym_rate} \nonumber
		\theta_{\epsilon}(i,j) &= \frac{\beta_{\epsilon}}{\mu_{\epsilon}(\mcE_{i}^{\epsilon})} \int_{\mcE_{i}^{\epsilon}} p_{\epsilon}(x) \boldsymbol{P}_{x}^{\epsilon}[H(\mcE_{j}^{\epsilon}) < H^{+}(\check{\mcE}_{j}^{\epsilon})] \ dx \\
		&= \beta_{\epsilon} \ \! \boldsymbol{P}_{\mu_{\epsilon}}^{\epsilon}[H(\mcE_{j}^{\epsilon}) < H^{+}(\check{\mcE}_{j}^{\epsilon})|\xi_{\epsilon}(0) \in \mcE_{i}^{\epsilon}]
	\end{align}
	recalling that $p_{\epsilon}$ is the probability density function of $\mu_{\epsilon}$. The probability in the right-hand side of \eqref{asym_rate} is that of the process $\xi_{\epsilon}(\cdot)$, starting from the invariant measure conditioned on $\xi_{\epsilon}(0) \in \mcE^{\epsilon}_{i}$, attaining after at least one jump the set $\mcE_{j}^{\epsilon}$ before any other metastable well, inclusive $\mcE_{i}$. The following condition, together with $\mfR^{(1)}$ and \textbf{(A5)} is sufficient for metastability.
	
	\vspace{0.25cm}
	
	\noindent \textit{\textbf{\bf Condition \textbf{(H0)}.} For $i \neq j \in S$, the sequence $\theta_{\epsilon}(i,j)$ converges as $\epsilon \to 0$. Denote this limit by
		\begin{equation}
			\label{H0} \tag{\textbf{\textbf{H0}}}
			\theta(i,j) = \lim_{\epsilon\to 0} \theta_{\epsilon}(i,j).	
	\end{equation}}
	
	\vspace{0.25cm}
	
	The next result, that is an extension of Corollary 7.3 in \cite{landim2021resolvent} to Markov processes in uncountable state spaces, shows that if conditions $\mathfrak{R}^{(1)}$, \eqref{H0} and \textbf{(A5)} hold, then the condition $\mathfrak{R}_{\mcL}$ holds with the generator $\mcL$ in \eqref{gen_limit} given by the rates in \eqref{H0}.
	
	\begin{proposition}
		\label{cor_7.3}
		If conditions $\mathfrak{R}^{(1)}$, \textbf{(A5)} (cf. \eqref{cm}) and \eqref{H0} hold, then condition $\mathfrak{R}^{(2)}_{\mcL}$ holds, in which $\mcL$ is the generator with the rates given by the limit in \eqref{H0}. In particular, condition $\mathfrak{R}_{\mcL}$ holds.
	\end{proposition}

    When there exists a $\mathfrak{s} > 0$ such that, for all $\eps > 0$,
    \begin{align}
        \label{holes_far}
        \min_{\substack{i,j,j^{\prime} \in S \\ i \neq j \neq j^{\prime}}} d(\Delta_{j^{\prime},j}^{\eps},I_{i}) > \mathfrak{s},
    \end{align}
    then, in general, it should hold $\theta(i,j) > 0$ whenever $\Delta_{i,j}^{\epsilon} \neq \emptyset$ and $\theta(i,j) = 0$ otherwise, since jumps between invariant components $I_{i}$ and $I_{j}$ should happen only when a hole between them exists. In this case, \eqref{H0} holds if
    \begin{equation}
		\label{timeScale}
		0 < \lim_{\epsilon\to 0} \frac{\boldsymbol{P}_{\mu_{\epsilon}}^{\epsilon}[H(\mcE_{j}^{\epsilon}) < H^{+}(\check{\mcE}_{j}^{\epsilon})|\xi_{\epsilon}(0) \in \mcE_{i}^{\epsilon}]}{\boldsymbol{P}_{\mu_{\epsilon}}^{\epsilon}[H(\mcE_{j^{\prime}}^{\epsilon}) < H^{+}(\check{\mcE}_{j^{\prime}}^{\epsilon})|\xi_{\epsilon}(0) \in \mcE_{i^{\prime}}^{\epsilon}]} < \infty
	\end{equation}
	for all pairs $(i,j)$ and $(i^{\prime},j^{\prime})$ with $\Delta_{i,j}^{\epsilon}, \Delta_{i^{\prime},j^{\prime}}^{\epsilon} \neq \emptyset$, and 
	\begin{equation}
		\label{timeScale2}
		\lim_{\epsilon\to 0} \frac{\boldsymbol{P}_{\mu_{\epsilon}}^{\epsilon}[H(\mcE_{j}^{\epsilon}) < H^{+}(\check{\mcE}_{j}^{\epsilon})|\xi_{\epsilon}(0) \in \mcE_{i}^{\epsilon}]}{\boldsymbol{P}_{\mu_{\epsilon}}^{\epsilon}[H(\mcE_{j^{\prime}}^{\epsilon}) < H^{+}(\check{\mcE}_{j^{\prime}}^{\epsilon})|\xi_{\epsilon}(0) \in \mcE_{i^{\prime}}^{\epsilon}]} = 0
	\end{equation}
	whenever $\Delta_{i,j}^{\epsilon} = \emptyset$, but $\Delta_{i,j^{\prime}}^{\epsilon} \neq \emptyset$.

    On the other hand, if \eqref{holes_far} does not hold, then it might be possible to get from $\mcE_{i}^{\eps}$ to $\mcE_{j}^{\eps}$ even if $\Delta_{i,j}^{\epsilon} = \emptyset$ without passing through $\mcE_{j^{\prime}}^{\eps}$ by jumping directly from $\Delta_{i,j^{\prime}}^{\eps}$ to $\Delta_{j^{\prime},j}^{\eps}$ and then to $\mcE_{j}^{\eps}$. In this case, it could happen that $\theta(i,j) > 0$ when $\Delta_{i,j}^{\epsilon} = \emptyset$.
	
	Inequality \eqref{timeScale} implies that the timescale to jump from $\mcE_{i}$ to $\mcE_{j}$ and from $\mcE_{i^{\prime}}$ to $\mcE_{j^{\prime}}$ is the same whenever there is a hole between the respective components. In order for \eqref{timeScale} to hold, one should define the metastable wells by properly choosing the sets $B_{i,j}^{\epsilon}$ in the definition of $\Delta_{i,j}^{\epsilon}$ (cf. \eqref{holes}), as will be illustrated in the example of Section \ref{SecEx2}. Conditions \eqref{timeScale} and \eqref{timeScale2} are analogous to the \textit{limiting holes ratio} and \textit{limiting averaged holes ratio} conditions for metastability in the contexts of \cite{bahsoun2011metastability,bahsoun2013escape,gonzalez2011approximating,gonzalez2025jumping}.
	
	\section{Stochastic stability}
	\label{SecCS}
	
	In view of \textbf{(A1)} and \textbf{(A3)}, and the condition $\mfC_{\mcL}$ of metastability, it is expected that, when $\epsilon \to 0$, $\mu_{\epsilon}$ will converge to a convex combination of the ACIMs $\mu_{i}$ of the original map $T$ restricted to the invariant components $I_{i}$. This kind of convergence, known as stochastic stability, has been widely studied in the literature, see for instance \cite{AA,GK}  and references therein. 
	
	In this section, we show that, if the Markov process associated with the randomly perturbed map $T_{\epsilon}$ is $\mcL$-metastable, then the invariant measure $\mu_{\epsilon}$ converges strongly to the convex combination
	\begin{align*}
		\mu \coloneqq \sum_{i \in S} \pi(i) \mu_{i}
	\end{align*}
	in which $\pi$ is the invariant measure of the Markov process with generator $\mcL$ and $\mu_{i}$ is the invariant measure of the $i$-th component of the map $T$. Furthermore, we present an upper bound on the rate of this convergence. 
	
	We first note that, if $\mu_{\epsilon}$ converges to a convex combination $\sum_{i \in S} \alpha_{i} \mu_{i}$ then, by definition of the invariant measure, it must hold
	\begin{align*}
		\alpha_{i} = \lim\limits_{\epsilon \to 0} \lim\limits_{t \to \infty} \boldsymbol{P}_{x}^{\epsilon}[\xi_{\epsilon}(t) \in I_{i}]
	\end{align*}
	for any $x \in I$. Recall the definition of the trace process (cf. \eqref{trace_process}) and observe that $S^{\mcE^{\eps}}(T^{\mcE^{\eps}}(t)) \geq t$ for all $t > 0$. Since
	\begin{align*}
		\boldsymbol{P}_{x}^{\epsilon}[\xi_{\epsilon}(t) \in I_{i}] = \boldsymbol{P}_{x}^{\epsilon}[\xi_{\epsilon}(t) \in \mcE_{i}^{\epsilon}] + \boldsymbol{P}_{x}^{\epsilon}[\xi_{\epsilon}(t) \in \Delta_{i}^{\epsilon}]
	\end{align*}
	and $\{\xi_{\epsilon}^{\mcE^{\eps}}(T^{\mcE^{\epsilon}}(t)) \in \mcE_{i}^{\epsilon}\} = \{Y_{\eps}(T^{\mcE^{\epsilon}}(t)) = i\}$, we have that
	\begin{align*}
		\alpha_{i} &= \lim\limits_{\epsilon \to 0} \lim\limits_{t \to \infty} \boldsymbol{P}_{x}^{\epsilon}[Y_{\eps}(T^{\mcE^{\epsilon}}(t)) = i] + \lim\limits_{\epsilon \to 0} \mu_{\epsilon}(\Delta_{i}^{\epsilon}) = \pi(i)
	\end{align*}
	in which the second equality follows from the condition $\mathfrak{C}_{\mcL}$ of metastability and \textbf{(A5)}. Therefore, if the limit of $\mu_{\epsilon}$ is a convex combination $\sum_{i \in S} \alpha_{i} \mu_{i}$, then $\alpha_{i} = \pi(i)$. 
    
	We apply the \textit{Chen-Stein method}, proposed by \cite{stein1972} as a method for bounding the approximation error between the distribution of the sum of random variables and the Gaussian distribution, and extended by \cite{chen1975} to the Poisson distribution. In \cite{barbour1988} the Chen-Stein method was further extended to bound the approximation error to the invariant measure of a Markov process in what has been known as the \textit{generator method}.
	
	The generator method is based on the equivalence
	\begin{align*}
		\nu[\mcL_{\epsilon}F] = 0 \text{ for all } F \in \mcD(\mcL_{\epsilon}) \iff \nu = \mu_{\epsilon},
	\end{align*}
	in which $\nu$ is a measure in $\mfB_{I}$, $\nu[\cdot]$ means expectation under $\nu$ and $\mcD(\mcL_{\epsilon})$ is the domain of $\mcL_{\epsilon}$. Fix $A \in \mfB_{I}$ and let $F_{A}$ be the solution of the equation
	\begin{align}
		\label{gen_equation}
		\mcL_{\epsilon}F_{A}(x) = \chi_{A}(x) - \mu_{\epsilon}(A), \ \ \forall x \in I.
	\end{align}
	Taking expectations with respect to $\nu$ on both sides of \eqref{gen_equation}, the absolute values and then the supremum over $A \in \mfB_{I}$ we conclude that
	\begin{align}
		\label{gen_equation2}
		\sup_{A \in \mfB_{I}} |\nu[\mcL_{\epsilon}F_{A}]| = \sup_{A \in \mfB_{I}} \left|\nu(A) - \mu_{\epsilon}(A)\right| = d_{TV}(\nu,\mu_{\epsilon}),
	\end{align}
	so the total variation distance between $\nu$ and $\mu_{\epsilon}$ can be bounded by bounding the left-hand side of \eqref{gen_equation2}. Observe that if $\nu = \mu_{\epsilon}$ then both sides of \eqref{gen_equation2} equal zero, as expected.

    Recall the definition of $q_{\eps}$ in \eqref{qeps} and that $\tilde{\mu}_{\epsilon}^{i}$ is the invariant measure of the restricted process defined in Section \ref{SecR1}, and let $\tilde{\mu}_{\eps} \coloneqq \sum_{i \in S} \mu_{\eps}(I_{i}) \tilde{\mu}_{\epsilon}^{i}$. With an adaptation of the generator method, we can prove the following bound for $d_{TV}(\mu_{\epsilon},\mu)$ in terms of $ d_{TV}(\tilde{\mu}_{\eps},\mu)$ and $q_{\eps}$.

    \begin{proposition}
		\label{theorem_dTV}
		Assume the conditions of Proposition \ref{prop_not_jump} and Theorem \ref{main_theorem} are in force so, in particular, $\mathfrak{M}$ holds. If the process $\xi_{\epsilon}(\cdot)$ is $\mcL$-metastable, then 
		\begin{align*}
            d_{TV}(\mu_{\epsilon},\mu) \leq C \, q_{\eps} \, \log q_{\eps}^{-1} +  d_{TV}(\tilde{\mu}_{\eps},\mu)
		\end{align*}
        for a constant $C > 0$ that depends on quantities in \textbf{(A1)}-\textbf{(A7)}.
	\end{proposition}

    We now bound $d_{TV}(\tilde{\mu}_{\eps},\mu)$. Since
    \begin{align*}
        d_{TV}(\tilde{\mu}_{\eps},\mu) &= \sup_{A \in \mathfrak{B}_{I}} \left|\sum_{i \in S} \mu_{\eps}(I_{i}) \, \tilde{\mu}_{\epsilon}^{i}(A \cap I_{i}) - \pi(i) \, \mu_{i}(A \cap I_{i})\right|\\
        &\leq \sum_{i \in S} \pi(i) \sup_{B \in \mathfrak{B}_{I_{i}}} \left|\tilde{\mu}_{\epsilon}^{i}(B) - \mu_{i}(B)\right| + \tilde{\mu}_{\epsilon}^{i}(B) \left|\mu_{\eps}(I_{i}) - \pi(i)\right|\\
        &\leq \sum_{i \in S} \pi(i) \, d_{TV}(\tilde{\mu}_{\epsilon}^{i},\mu_{i}) + \sum_{i \in S} \left|\mu_{\eps}(I_{i}) - \pi(i)\right|
    \end{align*}
    we shall bound $ d_{TV}(\tilde{\mu}_{\epsilon}^{i},\mu_{i})$. 
    
    Fix $i \in S$. By a deduction analogous to \eqref{bound_dTV_BV}, for any $n \geq 1$ fixed, denoting by $p_{i}$ and $\tilde{p}_{\epsilon}^{i}$ the probability density functions of $\mu_{i}$ and $\tilde{\mu}_{\epsilon}^{i}$, respectively, it holds
	\begin{align*}
		d_{TV}(\tilde{\mu}_{\epsilon}^{i},\mu_{i}) &\leq \left\lVert \tilde{p}_{\epsilon}^{i} - p_{i} \right\rVert_{1} = \left\lVert \tilde{P}_{\epsilon}^{n}\tilde{p}_{\epsilon}^{i} - P^{n} p_{i} \right\rVert_{1}
	\end{align*}
    in which $\tilde{P}_{\epsilon} \coloneqq \tilde{P}_{\epsilon}^{i}$ is the transfer operator \eqref{transfer_oper} associated with $\tilde{\xi}_{\eps}^{i}(\cdot)$. By a standard technique, see for instance Theorem 3 in \cite{GG}, we have
    \begin{equation*}
        \left\lVert \tilde{P}_{\epsilon}^{n}\tilde{p}_{\epsilon}^{i} - P^{n} p_{i} \right\rVert_1\le \left\lVert \tilde{P}_{\epsilon}^{n}\tilde{p}_{\epsilon}^{i} - P^{n} \tilde{p}_{\epsilon}^{i}
     \right\rVert_1+\left\lVert P^n \tilde{p}_{\epsilon}^{i} - P^{n} p_{i} \right\rVert_1.
    \end{equation*}
	Since the difference of the densities has an integral equal to zero, we can apply the bound (\ref{sm}) in \textbf{(A7)} to the second term on the left-hand side to get
    \begin{equation*}
        \left\lVert P^n (\tilde{p}_{\epsilon}^{i} - p_{i}) \right\rVert_1\le \Lambda_i \lambda_i^n.
    \end{equation*}
	
    To deal with the first term, we write it as a telescopic sum
    \begin{equation*}
        (P^{n}-\tilde{P}_{\epsilon}^{n})\tilde{p}_{\epsilon}^{i}=\sum_{k=1}^n P^{n-k}(P-\tilde{P}_{\epsilon})\tilde{p}_{\epsilon}^{i}.
    \end{equation*}
    Then, we assume that there is a sequence $\{d_{\eps}\}_{\eps > 0}$ converging to zero such that, for all $i \in S$,
	\begin{equation}
		\label{oper_close}
		\|(P-\tilde{P}_{\epsilon})\tilde{p}_{\epsilon}^{i}\|_1 \leq d_{\eps} \, \|\tilde{p}_{\epsilon}^{i}\| ,
	\end{equation}
	and that the BV norm $\|\tilde{p}_{\epsilon}^{i}\|$ is bounded uniformly in $\epsilon$, that is, $\sup_\epsilon \lVert \tilde{p}_{\epsilon}^{i} \rVert \leq M$. Then, by the $L^1$-contraction of the transfer operator, it follows that
	\begin{equation}
        \label{ty}
		\|\tilde{p}_{\epsilon}^{i} - p_{i}\|_1\le \Lambda_i \, \lambda_i^n + n \, M \, d_{\eps}.
	\end{equation}
	By minimising the right-hand side of \eqref{ty} on $n$, we conclude there exists a constant $C > 0$ such that
    \begin{align*}
        d_{TV}(\tilde{\mu}_{\epsilon}^{i},\mu_{i}) \leq C \, d_{\eps} \log d_{\eps}^{-1}.
    \end{align*}
	We have proved the following theorem.
	
	\begin{theorem}
		\label{prop_dTV_final}
		Assume the conditions of Proposition \ref{prop_not_jump} and Theorem \ref{main_theorem} are in force. If the process $\xi_{\epsilon}(\cdot)$ is $\mcL$-metastable and, for all $i \in S$, \eqref{oper_close} holds and $\lVert \tilde{p}_{\epsilon}^{i} \rVert$ is uniformly bounded in $\epsilon$, then 
		\begin{align*}
			d_{TV}(\mu_{\epsilon},\mu) \leq C \, \left[q_{\eps} \, \log q_{\eps}^{-1} + d_{\eps} \log d_{\eps}^{-1}\right] + \sum_{i \in S} \left|\mu_{\eps}(I_{i}) - \pi(i)\right|
		\end{align*}
        for a constant $C > 0$ that depends on quantities in \textbf{(A1)}-\textbf{(A7)} and on the upper bound of $\lVert \tilde{p}_{\epsilon}^{i} \rVert$.
	\end{theorem}

    As a corollary, since $\left|\mu_{\eps}(I_{i}) - \pi(i)\right|$ converges to zero as $\eps \to 0$ by condition $\mathfrak{C}_{\mcL}$ and \textbf{(A5)}, we conclude that $\mcL$-metastability implies the stochastic stability in this case. In particular, the resolvent condition $\mathfrak{R}_{\mcL}$ implies the stochastic stability.

    \begin{corollary}
        \label{cor_stable}
        Assume the conditions of Proposition \ref{prop_not_jump} and Theorem \ref{main_theorem} are in force. If the process $\xi_{\epsilon}(\cdot)$ is $\mcL$-metastable, then it is stochastically stable, that is, $\mu_{\eps}$ converges strongly to a convex combination of $\mu_{i}$.
    \end{corollary}
    
    \begin{remark}
        In cases in which there is a \textit{symmetry} between the components $I_{i}$, as in the examples in Section \ref{Sec_applications}, it holds $\mu_{\eps}(I_{i}) = \pi(i)$ for all $i \in S$ and $\eps > 0$, so the bound in Theorem \ref{prop_dTV_final} depends solely on $q_{\eps}$ and $d_{\eps}$. In other cases, a more detailed analysis of $\mu_{\eps}(I_{i})$ may be necessary to obtain a meaningful bound.
    \end{remark}

	\section{Example: Expanding maps with two wells}
	\label{Sec_applications}
	
	The first example we consider is that of a Markov process generated by a perturbation of the expanding map with two wells presented in Figure \ref{f1} by an additive uniformly distributed noise. In this case, $I = [0,1]$ with $T$-invariant components $I_{1} = (0,1/2)$ and $I_{2} = (1/2,1)$, and $b > 0$ fixed is a free parameter. We assume that $b$ is small enough, so the absolute value of the slope of the linear components is greater than $2$.
    
    For $\epsilon < b$, we consider the perturbed map $T_{\epsilon}(x) = T(x) + \sigma_{\epsilon}^{x}$ with additive noise $\sigma_{\epsilon}^{x}$ following a uniform distribution in $[-\epsilon,\epsilon^{q}]$ for $x \in (0,1/2)$ and a uniform distribution in $[-\epsilon^{q},\epsilon]$ for $x \in (1/2,1)$ for $q > 2$ fixed. We couple all noises so that, for all $\omega \in \Omega$, $\sigma_{\epsilon}^{x}(\omega) = \sigma_{\epsilon}^{y}(\omega)$ if $x,y$ are in the same component $I_{i}$, and $\sigma_{\epsilon}^{x}(\omega) = - \sigma_{\epsilon}^{y}(\omega)$ if $x$ and $y$ are in different components. For all $\eps > 0$, we consider $B_{1,2}^{\eps}$ and $B_{2,1}^{\eps}$ in the definition of the holes (cf. \eqref{holes}) in a way that $\Delta_{2,1}^{\eps}$ is the translation of $\Delta_{1,2}^{\eps}$ by $1/2$.

    Observe that the noise $\sigma_{\epsilon}^{x}$ has a drift pointing inward to the respective invariant component $I_{i}$, representing an energy barrier that the process has to overcome to leave the well. We consider $b > 0$ as a simplification since it avoids truncation of the noise for $x$ in a neighbourhood of $0$, $1/2$ or $1$, but $b = 0$ could be considered at the cost of more technical details. Moreover, other sub-Gaussian noises could be considered again at the cost of technical details due to truncation.
	
	We will show that \textbf{(A1)}-\textbf{(A7)} hold in this scenario, and that $\mfR^{(1)}$ and $\mfR^{(2)}_{\mcL}$ are in force for the generator $\mcL$ of a symmetric Markov process in $S = \{1,2\}$, so it follows from Theorem \ref{theorem_resolvent} that the Markov process generated by the perturbed map is $\mcL$-metastable.
	
	\subsection{Assumptions \textbf{(A1)}-\textbf{(A7)}}
    \label{ex_Ass}
	
	Clearly, the map $T$ satisfies \textbf{(A1)} and the perturbed map $T_{\epsilon}$ satisfies \textbf{(A2)}. The first part of assumption \textbf{(A3)} is satisfied since $\mbP(|\sigma_{\epsilon}^{x}| > \epsilon) = 0$ for all $x \in I$ and the second condition follows since both $\sup_{x \in I_{1}} \mbP(\sigma_{\epsilon}^{x} > 0)$ and  $\sup_{x \in I_{2}} \mbP(\sigma_{\epsilon}^{x} < 0)$ converge to zero as $\epsilon \to 0$. 
	
	Assumption \textbf{(A4)} follows from Proposition \ref{prop_ACIM_expand} in Section \ref{SecAux} that implies the uniqueness of the ACIM of the Markov chain $X_{n}^{\eps}$ since $\inf_{x \in I_{ij}} |T^{\prime}(x)| \geq 2$ and $T(I_{ij}) = T(I_{ij^{\prime}})$, recalling from \textbf{(A1.1)} the decomposition $I_{i} = \bigcup_{j = 1}^{4} I_{ij}$ in which $I_{ij}$ are intervals in which $T$ is one-to one. Furthermore, the proof of Proposition \ref{prop_ACIM_expand} implies that \eqref{cond_unique_ACIM} holds in this case, so the Markov chain $X_{n}^{\eps}$ is aperiodic. 
    
    We turn to \textbf{(A5)}. Let
    \begin{equation*}
        P_{\eps}F(x) = \int_{I} F(y) \rho_{\eps}(y,x) \, dy
    \end{equation*}
    be the adjoint of the Markov operator associated with $\rho_{\eps}$, and observe that, analogously to \eqref{randomTrans}, it holds
    \begin{equation*}
        P_{\eps}F(x) = \int_{\Omega} (P_{\eps,\omega}F)(x) \, d\mbP(\omega)
    \end{equation*}
    in which $P_{\eps,\omega}$ is the transfer operator of the map $T_{\eps}(\omega,x) = T(x) + \sigma_{\epsilon}^{x}(\omega)$. We recall, see (\ref{LLYY}), that the operator $P_{\eps,\omega}$ satisfies the LY inequality if there exist $0 < \gamma < 1$ and $C,D > 0,$ such that for all $F \in BV$ and $n \ge 1$ it holds 
	\begin{equation}
		\label{LY}
		\lVert P_{\eps,\omega}^nF \rVert \le C \gamma^n \lVert F \rVert + D\lVert F \rVert_1.
	\end{equation}
    It is enough to prove such an inequality for a given $n_0 > 0$ with the term multiplying $\lVert F \rVert,$ say $a_{n_0},$ strictly lesser than $1,$ and the inequality \eqref{LY} follows by a standard computation. 

    For $\omega \in \Omega$, let $I_{\omega}^{(l)}, l = 1,\dots,8$, be the intervals in which the map $T_{\eps}(\omega,\cdot)$ is one-to-one. We note that $T_{\eps}(\omega,\cdot)$ has a same constant derivative with absolute value $s > 2$ in the interior of each interval $I^{(l)}_{\omega}$, and it can be smoothly extended to the boundaries of these intervals. Setting 
    \begin{equation*}
    	 \ell := \inf_l \text{Leb}(T_{\eps}(\omega,I^{(l)}_{\omega})) = \frac{1}{2} - b,
    \end{equation*}
    since the second derivative of $T_{\eps}(\omega,\cdot)$ is zero on $I^{(l)}_{\omega}$, we can apply Proposition 4.1 in \cite{Broise} to conclude that, for $F \in BV$,
    \begin{equation}
    	\label{BPF}
		\lVert P_{\eps,\omega} \rVert \leq \frac{2}{s} \lVert F \rVert + \frac{2}{\ell} \ \lVert F \rVert_1,
    \end{equation}
	yielding \eqref{LY} with the same parameters $C = 1$, $\gamma = 2/s$ and $D = 2/\ell$ for all $\omega \in \Omega$. By an inequality analogous to \eqref{norm_P_Ptilde}, it follows from \eqref{BPF} that
    \begin{equation}
		\label{LY2}
		\lVert P_{\eps}F \rVert \le \gamma \lVert F \rVert + D\lVert F \rVert_1.
	\end{equation}
    
    Taking $F = p_{\eps}$ in \eqref{LY2}, recalling that $p_{\eps}$ is the density function of $\mu_{\eps}$, and hence a fixed point of $P_{\eps}$, we have that
    \begin{equation}
    	\label{bbdd}
    	\|p_{\eps}\| \le \frac{D}{1 - \gamma}
    \end{equation}
    as $\lVert p_{\eps} \rVert_{1} = 1$. Since $\lVert p_{\eps}\rVert_{\infty} \le \lVert p_{\eps} \rVert$ by \eqref{BV_inf}, 
    \begin{equation*}
        \mu_{\eps}(\Delta^{\eps}) = \int_{\Delta^{\eps}} p_{\eps}(x) \, dx \leq \lVert p_{\eps}\rVert_{\infty} \, \text{Leb}(\Delta^{\eps}) \leq \frac{D}{1 - \gamma} \, \text{Leb}(\Delta^{\eps})\, ,
    \end{equation*}
    so $\lim_{\epsilon \to 0} \mu_{\epsilon}(\Delta^{\eps}) = 0$ since $\text{Leb}(\Delta^{\eps}) \to 0$ when $\eps \to 0$ (cf. \eqref{lim_Leb_holes}). By symmetry, we conclude that $\lim_{\epsilon \to 0} \mu_{\epsilon}(\mcE_{i}^{\epsilon}) = 1/2$, so \textbf{(A5)} holds.
	 
	Assumption \textbf{(A6)} clearly holds, since $T$ is piecewise linear. The proof of \textbf{(A7)} consists of showing that $T_{i} \coloneqq T|_{I_{i}}, i = 1,2,$ admits only one absolutely continuous invariant mixing measure. We will consider $T_{1}$ and an analogous proof holds for $T_{2}$ by symmetry. To easy notation, from now on, we denote simple $T \coloneqq T_1$ and by $P: BV(I_{1}) \mapsto BV(I_{1})$ the transfer operator defined in \eqref{transfer_oper} restricted to functions in $BV(I_{1})$.
	
	The standard proof in the case of expanding maps is to show that the LY inequality holds, and that $1$ is the maximal eigenvalue of $P$ and there is no other eigenvalue of modulus one (with an abuse of language, we call these spectral properties {\em quasi-compactness}), see for instance \cite{Hennion}. It is clear that $T$ satisfies the LY inequality \eqref{BPF} by the same arguments of $T_{\eps,\omega}$. In order to prove that $1$ is the only eigenvalue of $P$ on the unit circle, it will be enough to show that $T$ and its iterates admit only one ergodic ACIM.  
    
    Adapting to our case the proof of Proposition III.4 in \cite{Collet}, we proceed by contradiction, assuming that $T$ has two absolutely continuous invariant ergodic measures with densities $p_1, p_2$. As argued in the aforementioned Proposition III.4, we can find two disjoint open intervals $I_{1,1}, I_{1,2} \subset I_{1}$, and a positive number $c$, such that $p_1|_{I_{1,1}} > c$ and $p_2|_{I_{1,2}} > c$.
	
	Then, we observe that the set $\{T^{-n}(1/4)\}_{n\ge 1}$ of the preimages of the middle point is dense in $[b,1/2]$. Indeed, otherwise, we could find an interval whose orbit never intersects $1/4$, and therefore it is expanded indefinitely. In other words, at each iteration of $T$ applied to an interval of $(0,1/2)$, its size increases, so after some finite number of iterations, it should cover $(b,1/2)$, and hence contains $1/4$.
    
    Suppose $x_2 \in I_{1,2}$ is an element of $T^{-n'}(1/4)$. Then the sequence $T^{-m}(x_2)$, $m > n'$, will again be dense; take now  $x_1 \in T^{-m}(x_2)\cap I_{1,1}, m>n'.$ Then $T^m$ is continuous in $x_1$ and therefore the set $T^m(I_{1,1})\cap I_{1,2}$ will have positive measure, contradicting Birkhoff's ergodic theorem.  Considering higher iterates  $T^{(k)} \coloneqq T^k$ of $T$, the sequence $\{(T^{(k)})^{-n}(1/4)\}_{n\ge 1}$ will be again dense in $[b,1/2]$ and therefore we can repeat the argument above. This concludes the proof of \textbf{(A7)}.
    
    \subsection{Condition $\mfR^{(1)}$}
    \label{ex_R1}
	
	We will construct a restricted process in each invariant component and prove that it satisfies the conditions of Theorem \ref{main_theorem} so $\mfM$, and consequently $\mfR^{(1)}$, holds. Again, only the case of the first component $I_{1} = (0,1/2)$ will be considered, and the same results follow for the second component by symmetry. Recall that we are denoting $T \coloneqq T|_{I_{1}}$ and $P: BV(I_{1}) \mapsto BV(I_{1})$ the restriction of the transfer operator of $T$ to $I_{1}$.
		
	Observe that $T(x) + \sigma \in I_{1}$ for all $x \in I_{1}$ if $-\eps < \sigma < 0$ and for $x \in I_{1}\setminus([x_{1}(\sigma),x_{2}(\sigma)] \cup [x_{3}(\sigma),x_{4}(\sigma)])$ if $\eps^{q} > \sigma > 0$, in which all $x_{k}(\sigma)$ satisfy $T(x_{k}(\sigma)) + \sigma = 1/2$. Therefore, considering the coupling between the noises $\sigma_{\epsilon}^{x}$ and $\sigma_{\epsilon}^{y}$ for $x,y \in I_{1}$ and denoting by $\sigma_{\epsilon}(\omega)$ the common noise, we define $\tilde{T}_{\epsilon}(\omega,x)$ as
	\begin{align*}
		\begin{cases}
			T(x) + \sigma_{\epsilon}(\omega), &\text{ if } \sigma_{\epsilon}(\omega) < 0\, ,\\	
			T(x) + \sigma_{\epsilon}(\omega), &\text{ if } \sigma_{\epsilon}(\omega) > 0 \text{ and } x \notin [x_{1}(\omega),x_{2}(\omega)] \cup [x_{3}(\omega),x_{4}(\omega)]\, ,\\	
			a_{1}(\omega) \, x - k_{1}(\omega), & 	\text{ if } \sigma_{\epsilon}(\omega) > 0 \text{ and } x \in [x_{1}(\omega),x_{2}(\omega)]\\
            a_{2}(\omega) \, x - k_{2}(\omega), & 	\text{ if } \sigma_{\epsilon}(\omega) > 0 \text{ and } x \in [x_{3}(\omega),x_{4}(\omega)]
		\end{cases}
	\end{align*}
	in which $x_{i}(\omega)$ depends on $\omega$ through $\sigma_{\epsilon}(\omega)$, $a_{1}(\omega)$ and $k_{1}(\omega)$ are the slope and intercept of the line that passes through $(x_{1}(\omega),b)$ and $(x_{2}(\omega),1/2)$, and $a_{2}(\omega)$ and $k_{2}(\omega)$ that of the line that passes through $(x_{3}(\omega),b)$ and $(x_{4}(\omega),1/2)$. We note that
    \begin{equation*}
        x_{1}(\omega) = \frac{1}{8}\left(\frac{1 - 2b - 2\sigma(\omega)}{1 - 2b}\right)
    \end{equation*}
    and, since $x_{1}(\omega) + x_{2}(\omega) = 1/4$,
    \begin{equation}
        \label{difx}
        x_{2}(\omega) - x_{1}(\omega) = 2\left(\frac{1}{8} - x_{1}(\omega)\right) = \frac{\sigma(\omega)}{2 - 4b},
    \end{equation}
    and the same holds for $x_{4}(\omega) - x_{3}(\omega)$.

    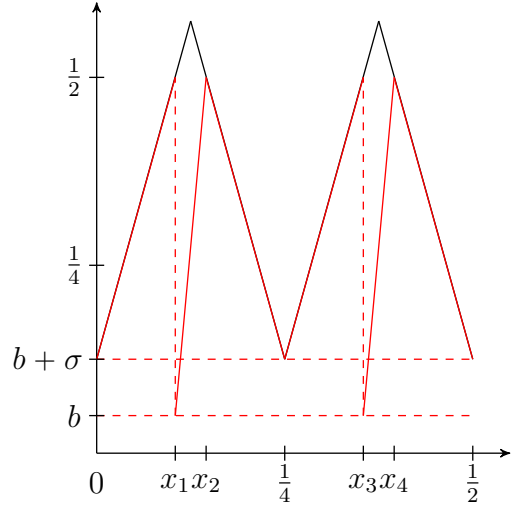
\begin{figure}[ht]
		\centering
		\begin{tikzpicture}[scale=10]
			\node at (0,-0.04) {$0$};
			\node at (0.5,-0.04) {$\frac{1}{2}$};
			\node at (0.25,-0.04) {$\frac{1}{4}$};
            \node at (0.1045,-0.04) {$x_{1}$};
            \node at (0.1455,-0.04) {$x_{2}$};
            \node at (0.3545,-0.04) {$x_{3}$};
            \node at (0.3955,-0.04) {$x_{4}$};
			
			\node at (-0.03,0.5) {$\frac{1}{2}$};
			\node at (-0.03,0.25) {$\frac{1}{4}$};
			\node at (-0.03,0.05) {$b$};
			\node at (-0.065,0.125) {$b + \sigma$};

			\begin{scope}[line width=0.5pt]
				\draw[->] (0,0) to (0.55,0);
				\draw[->] (0,0) to (0,0.6);
				
				
				\draw[-] (0,0.125) to (0.125,0.575);
                \draw[-] (0.125,0.575) to (0.25,0.125);
                \draw[-] (0.25,0.125) to (0.375,0.575);
				\draw[-] (0.375,0.575) to (0.5,0.125);

                \draw[-,red] (0,0.125) to (0.1045,0.5);
                \draw[-,red] (0.1045,0.05) to (0.1455,0.5);
                \draw[-,red,dashed] (0.1045,0.05) to (0.1045,0.5);
                \draw[-,red] (0.1455,0.5) to (0.25,0.125);
                \draw[-,red] (0.25,0.125) to (0.3545,0.5);
                \draw[-,red] (0.3545,0.05) to (0.3955,0.5);
                \draw[-,red,dashed] (0.3545,0.05) to (0.3545,0.5);                
				\draw[-,red] (0.3955,0.5) to (0.5,0.125);

                \draw[-,red,dashed] (0,0.05) to (0.5,0.05); 
                \draw[-,red,dashed] (0,0.125) to (0.5,0.125); 
				
				\draw[-] (0.5,-0.01) to (0.5,0.01);
				\draw[-] (0.25,-0.01) to (0.25,0.01);
				\draw[-] (0.1045,-0.01) to (0.1045,0.01);
				\draw[-] (0.1455,-0.01) to (0.1455,0.01);
                \draw[-] (0.3545,-0.01) to (0.3545,0.01);
				\draw[-] (0.3955,-0.01) to (0.3955,0.01);
				
				\draw[-] (-0.01,0.5) to (0.01,0.5);
				\draw[-] (-0.01,0.25) to (0.01,0.25);
				\draw[-] (-0.01,0.05) to (0.01,0.05);
				\draw[-] (-0.01,0.125) to (0.01,0.125);
			\end{scope}
		\end{tikzpicture}
		\caption{Map $\tilde{T}^{1}_{\epsilon}(\omega,\cdot)$ in red for fixed $\omega \in \Omega$ with $I_{1} = (0,1/2)$ when $T$ is the map in Figure \ref{f1} and $T_{\epsilon}(\omega,\cdot) = T(\cdot) + \sigma(\omega)$ is in black with $\sigma(\omega) > 0$. Observe that the maps $\tilde{T}^{1}_{\epsilon}(\omega,\cdot)$ and $T_{\epsilon}(\omega,\cdot)$ coincide for $x$ satisfying $T_{\epsilon}(\omega,x) \in (0,1/2)$.} \label{f3}
	\end{figure}
    
    This map is depicted in red in Figure \ref{f3} for $\sigma_{\epsilon}(\omega) > 0$. Clearly, the speeded-up Markov process $\tilde{\xi}_{\epsilon}(\cdot)$ generated by $\tilde{T}_{\epsilon}$ satisfies the coupling property \eqref{coupling} so it is a restriction of the process $\xi_{\epsilon}(\cdot)$ to $I_{1}$, and its generator can be written as \eqref{generator_restricted}. Furthermore, it has a unique ACIM (see Remark \ref{rem_restricted}).
	
	Recall that $\tilde{P}_{\epsilon,\omega}$ is the transfer operator of $\tilde{T}_{\epsilon}(\omega,\cdot)$ and that $\tilde{P}_{\epsilon}$ is the adjoint of the Markov operator associated with the Markov chain $\tilde{X}^{\epsilon}_{n}$. We now prove that the operator $\tilde{P}_{\epsilon}$ satisfies the conditions of Theorem \ref{main_theorem}, so $\mfM$ holds.
	
	In view of Remark \ref{rem_LY}, in order to prove condition (a) in Theorem \ref{main_theorem}, it is enough to prove a suitable LY inequality for all $\omega \in \Omega$ and $\epsilon > 0$. By the same arguments that led to \eqref{BPF}, we conclude that, for all $\omega \in \Omega$ and $\eps > 0$,
    \begin{equation}
    	\label{BPF2}
		\lVert \tilde{P}_{\epsilon,\omega} \rVert \leq \frac{2}{s_{\epsilon,\omega}} \lVert F \rVert + \frac{2}{\ell_{\epsilon,\omega}} \ \lVert F \rVert_1
    \end{equation}
    for all $F \in BV$, with $s_{\epsilon,\omega} > 2$ and $\ell_{\epsilon,\omega} \geq 1/2 - b - \eps^{q}$. Therefore, \eqref{exp_LY} holds and (a) is in force.
    
    We turn to condition (b). We will estimate $\int_{I_{1}} (P - \tilde{P}_{\eps})F(x) \, dx$ by comparing directly $P$ with $\tilde{P}_{\eps, \om}$ and we refer the reader to Figure \ref{f3} for a visualisation of the arguments below.
		
	For each $\omega \in \Omega$, we denote by $\iota_1, \iota_2, \iota_3, \iota_4$ the inverse branches of $T$ and by\footnote{We omit the dependence of $\upsilon_1,\upsilon_2,\upsilon_3,\upsilon_4,\upsilon_5,\upsilon_6$ on $\omega$ to ease notation.} $\upsilon_1, \upsilon_2, \upsilon_3, \upsilon_4, \upsilon_5, \upsilon_6$ the inverse branches of $\tilde{T}_{\eps}(\om,\cdot)$, being $\upsilon_2$ and $\upsilon_5$ the extra branches when they exist, that is, when $\sigma(\om) > 0$. These are functions satisfying $T^{-1}(x) = \{\iota_{1}(x),\iota_2(x),\iota_3(x),\iota_4(x)\}$ for $x$ in the image of $T$ and $\tilde{T}^{-1}_{\epsilon}(\om,x) = \{\upsilon_{1}(x),\upsilon_2(x),\upsilon_3(x),\upsilon_4(x),\upsilon_5(x),\upsilon_6(x)\}$ for $x$ in the image of $\tilde{T}_{\epsilon}(\om,\cdot)$ with the convention that $\upsilon_2(x) = \emptyset$ and $\upsilon_5(x) = \emptyset$ when $\sigma(\om) \leq 0$. We denote by $s > 2$ the common absolute value of the slope of $\iota_{1}, \iota_{2},\iota_{3},\iota_{4}, \upsilon_{1}, \upsilon_{3}, \upsilon_{4}$ and $\upsilon_{6}$, and by $s_{2} > 2$ the common slope of $\upsilon_2$ and $\upsilon_{5}$. We first fix $\omega \in \Omega$ with $\eps^{q} > \sigma_{\eps}(\omega) > 0$ and consider two cases.
	
	\underline{Case 1:} $b \le x \le b + \sigma_{\epsilon}(\omega)$. In this case, for $F \in BV$, $\tilde{P}_{\eps,\om}F(x) = [F(\upsilon_{2}(x)) + F(\upsilon_{5}(x))]/s_{2}$ and therefore 
	\begin{align}
		\label{p1} \nonumber
		\int_{b}^{b + \sigma_{\eps}(\omega)}\left|PF(x)-\frac{F(\upsilon_{2}(x)) + F(\upsilon_{5}(x))}{s_{2}}\right|&dx \le  2 \left(\lVert PF \rVert + 2 \, s_{2}^{-1} \lVert F \rVert\right) \, \sigma_{\eps}(\omega) \\
		&\le  2 \, (C \gamma + D + 2 \, s_{2}^{-1}) \, \|F\| \, \epsilon^{q}
	\end{align}
    in which the first inequality is due to \eqref{BV_inf} and the second due to the LY inequality in \textbf{(A7)}.
	
	\hfill$\square$
	
	\underline{Case 2:} $b + \sigma_{\epsilon}(\omega) \le x \le 1/2$. In this case, for $F \in BV$, $|(P-P_{\eps,\om})F(x)|$ is bounded from above by
	\begin{align}
		\label{dif_transfer}
		\sum_{j = 1,3,4,6} \frac{1}{s}[|F(\iota_j(x)) - F(\upsilon_j(x))|] + \sum_{j = 2,5} \frac{1}{s_{2}} |F(\upsilon_{j}(x))|.
	\end{align}
	Integrating in $x$ on both sides of \eqref{dif_transfer}, we note that the contribution of the second term is equal to
	\begin{align}
		\label{p2} \nonumber
		\sum_{j = 2,5} &\int_{b + \sigma_{\epsilon}(\omega)}^{1/2} \frac{|F(\upsilon_j(x))|}{s_{2}} \, dx =\int_{\upsilon_{2}(b + \sigma_{\epsilon}(\omega))}^{x_{2}(\omega)} |F(y)| \, dy + \int_{\upsilon_{5}(b + \sigma_{\epsilon}(\omega))}^{x_{4}(\omega)} |F(y)| \, dy\\
		&\le 2 \, \|F\| \left[(x_{2}(\omega) - x_{1}(\om)) + (x_{4}(\omega) - x_{3}(\om)) \right] \leq  \frac{2 \, \eps^{q}}{1 - 2b} \, \|F\|
	\end{align}
	in which the first inequality is due to \eqref{BV_inf} and the second due to \eqref{difx}.
    
    To treat the first term in \eqref{dif_transfer}, we first observe that the point $\upsilon_j(x)$ is in a $o(\eps^{q})$-neighbourhood of the point $\iota_j(x)$ for $j = 1,3,4,6$, then we use the formula\footnote{This formula states that, for $t > 0$, $\int_{0}^{1} \sup_{|x - y| \leq t} |F(x) - F(y)| \, dx \leq 2 \, t \, |F|_{TV}$.} (3.11) in \cite{Conze} to conclude that, for a constant $C > 0$ that does not depend on $\eps$ that may change from line to line,
    \begin{align}
    \label{pp3}
        \nonumber
        \frac{1}{s} \int_{b + \sigma_{\epsilon}(\omega)}^{1/2} |F(\iota_j(x)) - F(\upsilon_j(x))|dx &\le \frac{1}{s} \int_{0}^{1/2} \sup\limits_{|x - y| \leq C \eps^{q}} |F(x) - F(y)| \, dx \\
        &\le C \eps^{q} \, \|F\|.
    \end{align}
    Therefore, by combining \eqref{p1}, \eqref{dif_transfer}, \eqref{p2} and \eqref{pp3}, we conclude that, if $\sigma(\omega) > 0$ then 
     \begin{equation*}
        \int_{I_{1}} (P - \tilde{P}_{\eps})F(x) \, dx \leq C \, \eps^{q} \, \|F\|.
     \end{equation*}
    With the same arguments of Case 2, we can prove that, if $\sigma(\omega) < 0$, then
    \begin{equation*}
        \int_{I_{1}} (P - \tilde{P}_{\eps})F(x) \, dx \leq C \, \eps \, \|F\|,
     \end{equation*}
     and condition (b) follows. This ends the proof that $\mfM$, and therefore $\mfR^{(1)}$, holds for the perturbed map $T_{\epsilon}$.

    \begin{remark}
        Similar arguments could be considered to prove part (b) if a general sub-Gaussian noise satisfying \eqref{sub_gaussian} was considered instead. For $\lVert F \rVert \leq 1$, denoting $\Omega_{\epsilon} = \{\omega \in \Omega: |\sigma_{\eps}(\omega)| \leq \epsilon \sqrt{\gamma_{2} \log(1/\epsilon)}\}$, $\int_{I_{1}} (P - \tilde{P}_{\eps})F(x) \, dx$ is equal to
    	\begin{align*}
    	   \iint_{I_i \times \Omega_{\epsilon}} (P - \tilde{P}_{\eps, \om})F(x) \, dx \, d\mathbb{P}(\om) + \iint_{I_i\times \Omega_{\epsilon}^{c}} (P - \tilde{P}_{\eps, \om})F(x) \, dx \, d\mathbb{P}(\om).	
    	\end{align*}
    	The absolute value of the second term is bounded by 
    	\begin{equation*}
    		\label{p0}
    		2 \, \lVert F \rVert_{1} \, \mathbb{P}(\Omega_{\epsilon}^{c}) \le 2 \, \gamma_{1} \, \epsilon,
    	\end{equation*}
    	due to \eqref{sub_gaussian} since $\lVert F \rVert_1 \le \Vert F \rVert \leq 1$ and $\lVert PF \rVert_{1} = \lVert \tilde{P}_{\eps, \om}F \rVert_{1} = \lVert F \rVert_{1}$. The first term can be bounded by comparing $P$ with $\tilde{P}_{\eps, \om}$ as above for $\omega \in \Omega_{\eps}$.
    \end{remark}
	
	\subsection{Condition $\mfR_{\mcL}^{(2)}$}
    \label{ex_R2}
	
	That \eqref{H0} holds in this case is direct, since \eqref{timeScale} clearly holds due to the symmetry of the components of the original map $T$ and of the noise $\sigma_{\epsilon}$. Furthermore, $\theta(1,2) = \theta(2,1)$ so the limiting Markov process is symmetric. It follows from Proposition \ref{cor_7.3} that $\mfR_{\mcL}^{(2)}$ is in force. Finally, from Theorem \ref{theorem_resolvent} follows that the Markov process generated by the perturbed map $T_{\epsilon}$ is $\mcL$-metastable.
    
    In particular, the stochastic stability follows from Theorem \ref{prop_dTV_final}. Let
    \begin{equation*}
        \mu = \frac{1}{2} \left(\mu_{1} + \mu_{2}\right),
    \end{equation*}
    recalling that $\mu_{i}$ is the ACIM of $T$ with support contained in $I_{i}$. Observe that, since $q > 2$,
    \begin{equation*}
        q_{\eps} = \frac{\eps^{q}}{\eps + \eps^{q}} \ll \eps.
    \end{equation*}
    We note that \eqref{oper_close} follows with $d_{\eps} = C \eps$ by the proof of condition (b) in Section \ref{ex_R1}, and that $\lVert \tilde{p}_{\eps} \rVert$ is uniformly bounded on $\eps$ follows by the same arguments as \eqref{bbdd} since $\tilde{P}_{\eps}$ satisfies the LY inequality with uniformly bounded parameters (cf. \eqref{BPF2}). By symmetry $1/2 = \mu_{\eps}(I_{i}) = \pi(i)$, so we
    apply Theorem \ref{prop_dTV_final} to obtain
    \begin{equation*}
        d_{TV}(\mu_{\epsilon},\mu) \leq C \, \eps \, \log(\eps^{-1}),
    \end{equation*}
    a bound on the convergence of $\mu_{\eps}$ to the convex combination $\mu$. 

    \subsection{More examples of maps with two wells}

    Other examples could be considered with modifications to the arguments of Sections \ref{ex_Ass}, \ref{ex_R1} and \ref{ex_R2}. On the one hand, the map in Figure \ref{f2} \textbf{(A)} could be treated with minor modifications of the proofs above, since all of its branches have a derivative greater than $2$ and the same image. On the other hand, to apply the theory for the map in Figure \ref{f2} \textbf{(B)}, a more careful analysis is necessary to establish the LY inequalities. While proving assumptions \textbf{(A1)}-\textbf{(A4)} and \textbf{(A6)}, part (b) of Theorem \ref{main_theorem} and \eqref{H0} can be done as in the previous sections, the arguments used to prove the LY inequalities do not hold since the slope of the branches of this map has an absolute value equal to $2$. 

    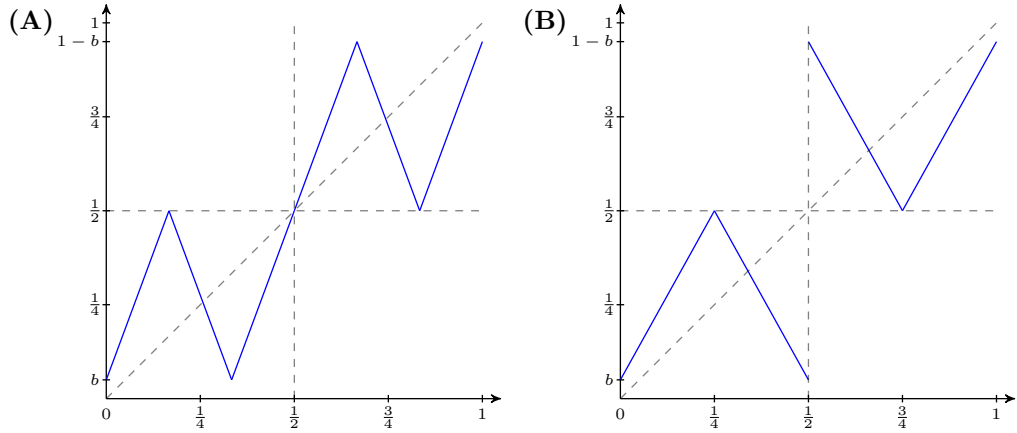
\begin{figure}[ht]
		\centering
        \begin{tikzpicture}[scale=5]
			\tikzstyle{hs} = [circle,draw=black, rounded corners,minimum width=3em, vertex distance=2.5cm, line width=1pt]
			\tikzstyle{hs2} = [circle,draw=black,dashed, rounded corners,minimum width=3em, vertex distance=2.5cm, line width=1pt]
			
			\node at (0,-0.04) {\tiny $0$};
			\node at (1,-0.04) {\tiny $1$};
			\node at (0.5,-0.05) {\tiny $\frac{1}{2}$};
			\node at (0.25,-0.05) {\tiny $\frac{1}{4}$};
			\node at (0.75,-0.05) {\tiny $\frac{3}{4}$};
			
			\node at (-0.03,1) {\tiny $1$};
			\node at (-0.03,0.5) {\tiny $\frac{1}{2}$};
			\node at (-0.03,0.25) {\tiny $\frac{1}{4}$};
			\node at (-0.03,0.75) {\tiny $\frac{3}{4}$};
            \node at (-0.075,0.95) {\tiny $1-b$};
            \node at (-0.03,0.05) {\tiny $b$};
            \node at (-0.2,1) {\footnotesize \textbf{(A)}};
			
			\begin{scope}[line width=0.5pt]
				\draw[->] (0,0) to (1.05,0);
				\draw[->] (0,0) to (0,1.05);
				
				\draw[-,blue] (0,0.05) to (1/6,0.5);
				\draw[-,blue] (1/6,0.5) to (1/3,0.05);
                \draw[-,blue] (1/3,0.05) to (0.5,0.5);
				
				\draw[-,blue] (0.5,0.5) to (2/3,0.95);
				\draw[-,blue] (2/3,0.95) to (5/6,0.5);
                \draw[-,blue] (5/6,0.5) to (1,0.95);
				
				\draw[-] (1,-0.01) to (1,0.01);
				\draw[-] (0.5,-0.01) to (0.5,0.01);
				\draw[-] (0.25,-0.01) to (0.25,0.01);
				\draw[-] (0.75,-0.01) to (0.75,0.01);
				
				\draw[-] (-0.01,1) to (0.01,1);
				\draw[-] (-0.01,0.5) to (0.01,0.5);
				\draw[-] (-0.01,0.25) to (0.01,0.25);
				\draw[-] (-0.01,0.75) to (0.01,0.75);
                \draw[-] (-0.01,0.05) to (0.01,0.05);
                \draw[-] (-0.01,0.95) to (0.01,0.95);
				
				\draw[-,dashed,opacity = 0.5] (0,0.5) to (1,0.5);
				\draw[-,dashed,opacity = 0.5] (0.5,0) to (0.5,1);
				\draw[-,dashed,opacity = 0.5] (0,0) to (1,1);
			\end{scope}
			\end{tikzpicture}
		\begin{tikzpicture}[scale=5]
			\tikzstyle{hs} = [circle,draw=black, rounded corners,minimum width=3em, vertex distance=2.5cm, line width=1pt]
			\tikzstyle{hs2} = [circle,draw=black,dashed, rounded corners,minimum width=3em, vertex distance=2.5cm, line width=1pt]
			
			\node at (0,-0.04) {\tiny $0$};
			\node at (1,-0.04) {\tiny $1$};
			\node at (0.5,-0.05) {\tiny $\frac{1}{2}$};
			\node at (0.25,-0.05) {\tiny $\frac{1}{4}$};
			\node at (0.75,-0.05) {\tiny $\frac{3}{4}$};
			
			\node at (-0.03,1) {\tiny $1$};
			\node at (-0.03,0.5) {\tiny $\frac{1}{2}$};
			\node at (-0.03,0.25) {\tiny $\frac{1}{4}$};
			\node at (-0.03,0.75) {\tiny $\frac{3}{4}$};
            \node at (-0.075,0.95) {\tiny $1-b$};
            \node at (-0.03,0.05) {\tiny $b$};
            \node at (-0.2,1) {\footnotesize \textbf{(B)}};
			
			\begin{scope}[line width=0.5pt]
				\draw[->] (0,0) to (1.05,0);
				\draw[->] (0,0) to (0,1.05);
				
				\draw[-,blue] (0,0.05) to (0.25,0.5);
				\draw[-,blue] (0.25,0.5) to (0.5,0.05);
				
				\draw[-,blue] (0.5,0.95) to (0.75,0.5);
				\draw[-,blue] (0.75,0.5) to (1,0.95);
				
				\draw[-] (1,-0.01) to (1,0.01);
				\draw[-] (0.5,-0.01) to (0.5,0.01);
				\draw[-] (0.25,-0.01) to (0.25,0.01);
				\draw[-] (0.75,-0.01) to (0.75,0.01);
				
				\draw[-] (-0.01,1) to (0.01,1);
				\draw[-] (-0.01,0.5) to (0.01,0.5);
				\draw[-] (-0.01,0.25) to (0.01,0.25);
				\draw[-] (-0.01,0.75) to (0.01,0.75);
                \draw[-] (-0.01,0.05) to (0.01,0.05);
                \draw[-] (-0.01,0.95) to (0.01,0.95);
				
				\draw[-,dashed,opacity = 0.5] (0,0.5) to (1,0.5);
				\draw[-,dashed,opacity = 0.5] (0.5,0) to (0.5,1);
				\draw[-,dashed,opacity = 0.5] (0,0) to (1,1);
			\end{scope}
			\end{tikzpicture}
		\caption{Examples of maps $T$ for which the developed theory could be applied.} \label{f2}
	\end{figure}
    
    In this case, the LY inequality might be proved for higher iterates of the respective maps. Since the second iterate $T^{2}$ is analogous to the map in Figure \ref{f1}, with the difference that the image of the two outer branches is $(c,1/2)$ with $c = T^{2}(0) > b$,  an LY inequality analogous to \eqref{BPF2} holds for $\lVert P^2 \rVert$ and \textbf{(A7)} follows, with the uniqueness of the ACIM established by the same arguments as in Section \ref{ex_Ass}. The LY inequalities for $P_{\eps}$ and $\tilde{P}_{\eps}$ in the proof of \textbf{(A5)} and condition (a) of Theorem \ref{main_theorem}, respectively, can be established with the method described in Remark \ref{rem_LY}. In this case, the restricted map $\tilde{T}_{\eps}$ can be defined analogously to that in Figure \ref{f3}, by adding one linear branch around $x = 1/4$ with image $(b,1/2)$ when the noise is positive.

    For all $\omega_{1},\omega_{2} \in \omega$, the branches of the map $T_{\eps}(\omega_{2},T_{\eps}(\omega_{1},\cdot))$ have slope with absolute value greater than $2$ and images with Lebesgue measure bounded from below uniformly on $\omega_{1},\omega_{2}$, so the LY inequality for $P_{\eps}^{2}$ follows by combining an inequality analogous to \eqref{BPF} with \eqref{exp_LY2}. A more careful analysis is necessary to prove LY for $\tilde{P}_{\eps}$ since, while the slopes of the branches of $\tilde{T}_{\eps}(\omega_{n},\cdot) \circ \cdots \circ \tilde{T}_{\eps}(\omega_{1},\cdot)$ have absolute value greater than $2$ for all $\omega_{1},\dots,\omega_{n} \in \Omega$ and $n \geq 2$, it is necessary to prove that the image of the branches are \textit{not too small} so the second condition in \eqref{exp_LY2} holds. We believe it is possible to prove \eqref{exp_LY2} in this case, for values as low as $n = 2$, but we do not analyse this further for the sake of brevity, since this is not central to this paper as the method has been already illustrated with other maps.

    Throughout this section, and in the next, we considered an expanding map that is piecewise linear for simplicity of exposition. If the branches are not linear, but the map is still expanding and $C^2$, then we could still recover in a similar way the Lasota-Yorke inequality for the unperturbed and perturbed systems and the comparison between the unperturbed and perturbed operators follows in the same way by using standard distortion estimates. In the case $b = 0$, we could consider a multiplicative noise of the form \eqref{multiplicative} and the results should follow with similar arguments considering a suitable distribution for the noise. For the additive noise, the results would also follow with modifications if the perturbed map considered $b = b_{\eps} \to 0$ as $\eps \to 0$ with a suitable rate, while the original map has $b = 0$.
	
	\section{Example: Expanding map with three wells}
    \label{SecEx2}
    
	In this section, we consider a random perturbation of the expanding map with three wells presented in Figure \ref{f4}, with invariant components $I_{1} = (0,1/3), I_{2} = (1/3,2/3)$ and $I_{3} = (2/3,1)$. As in the previous example, the perturbed map is $T_{\eps}(\omega,x) = T(x) + \sigma_{\eps}^{x}(\omega)$ in which $\sigma_{\eps}^{x}$ is uniformly distributed in $[-\epsilon,\eps^{q}]$ for $x \in I_{1}$ and $x \in (\bar{x},2/3)$, and uniformly distributed in $[-\eps^{q},\eps]$ for $x \in I_{3}$ and $x \in (1/3,\bar{x})$ for $q > 2$, in which $\bar{x}$ is the second fixed point of $T$ in $I_{2}$. For each $\omega \in \Omega$ we couple the noise as in the previous section for $x \in I_{1} \cup I_{3}$, and make $T_{\epsilon}(\omega,x) = T(x) + \sigma_{\eps}(\omega)$ for $x \in (1/3,\bar{x})$ and $T_{\epsilon}(\omega,x) = T(x) - \sigma_{\eps}(\omega)$ for $x \in (\bar{x},2/3)$ in which $\sigma_{\eps}$ has a uniform distribution in $[-\eps^{q},\eps]$.

    \begin{figure}[ht]
		\centering
		\begin{tikzpicture}[scale=8.5]
			\tikzstyle{hs} = [circle,draw=black, rounded corners,minimum width=3em, vertex distance=2.5cm, line width=1pt]
			\tikzstyle{hs2} = [circle,draw=black,dashed, rounded corners,minimum width=3em, vertex distance=2.5cm, line width=1pt]
			
			\node at (0,-0.04) {$0$};
			\node at (1,-0.04) {$1$};
			\node at (1/3,-0.05) {$\frac{1}{3}$};
			\node at (2/3,-0.05) {$\frac{2}{3}$};
            \node at (0.472,-0.05) {$\bar{x}$};
			
			\node at (-0.03,1) {$1$};
			\node at (-0.03,1/3) {$\frac{1}{3}$};
			\node at (-0.03,2/3) {$\frac{2}{3}$};

            \node at (1/6,-0.05) {\color{teal} $\mc{E}_{1}^{\eps}$};
            \node at (1/12,0.03) {\color{red} $\Delta_{1,2}^{\eps}$};

            \node at (1/3+1/6+0.025,-0.05) {\color{cyan} $\mc{E}_{2}^{\eps}$};
            \node at (5/12,0.03) {\color{purple} $\Delta_{2,1}^{\eps}$}; 
            \node at (0.55,0.03) {\color{olive} $\Delta_{2,3}^{\eps}$}; 
            
            \node at (2/3+1/6,-0.05) {\color{violet} $\mc{E}_{3}^{\eps}$};
            \node at (9/12,0.03) {\color{orange} $\Delta_{3,2}^{\eps}$}; 
			
			\begin{scope}[line width=0.5pt]
				\draw[->] (0,0) to (1.05,0);
				\draw[->] (0,0) to (0,1.05);
				
				\draw[-,blue] (0,0.05) to (1/12,1/3);
                \draw[-,blue] (1/12,1/3) to (2/12,0.05);
                \draw[-,blue] (2/12,0.05) to (3/12,1/3);
				\draw[-,blue] (3/12,1/3) to (1/3,0.05);
				
				\draw[-,blue] (8/12,0.95) to (9/12,2/3);
				\draw[-,blue] (9/12,2/3) to (10/12,0.95);
                \draw[-,blue] (10/12,0.95) to (11/12,2/3);
                \draw[-,blue] (11/12,2/3) to (1,0.95);
				
				\draw[-,blue] (1/3,0.635) to (5/12,1/3);
				\draw[-,blue] (5/12,1/3) to (0.55,2/3);
				\draw[-,blue] (0.55,2/3) to (2/3,0.375);
				
				\draw[-] (1,-0.01) to (1,0.01);
				\draw[-] (1/3,-0.01) to (1/3,0.01);
				\draw[-] (2/3,-0.01) to (2/3,0.01);
                \draw[-] (0.472,-0.01) to (0.472,0.01);
				
				\draw[-] (-0.01,1) to (0.01,1);
				\draw[-] (-0.01,1/3) to (0.01,1/3);
				\draw[-] (-0.01,2/3) to (0.01,2/3);
				
				\draw[-,dashed,opacity = 0.5] (0,1/3) to (1,1/3);
				\draw[-,dashed,opacity = 0.5] (0,2/3) to (1,2/3);
				\draw[-,dashed,opacity = 0.5] (1/3,0) to (1/3,1);
				\draw[-,dashed,opacity = 0.5] (2/3,0) to (2/3,1);
				\draw[-,dashed,opacity = 0.5] (0,0) to (1,1);
                \draw[-,dashed,opacity = 0.5] (0.472,0) to (0.472,1);

                \draw[-,teal,line width=2pt] (0,0) to (1/3,0);
                \draw[-,cyan,line width=2pt] (1/3,0) to (2/3,0);
                \draw[-,violet,line width=2pt] (2/3,0) to (1,0);

                \draw[-,red,line width=2pt] (1/12-1/36,0) to (1/12+1/36,0);
                \draw[-,red,line width=2pt] (3/12-1/36,0) to (3/12+1/36,0);

                \draw[-,orange,line width=2pt] (9/12-1/36,0) to (9/12+1/36,0);
                \draw[-,orange,line width=2pt] (11/12-1/36,0) to (11/12+1/36,0);

                \draw[-,purple,line width=2pt] (5/12-1/36,0) to (5/12+1/36,0);
                \draw[-,olive,line width=2pt] (0.55-1/36,0) to (0.55+1/36,0);
			\end{scope}
		\end{tikzpicture}
		\caption{Example of a map $T$ that satisfies \textbf{(A1)} with an illustration of the holes and metastable wells.} \label{f4}
	\end{figure}
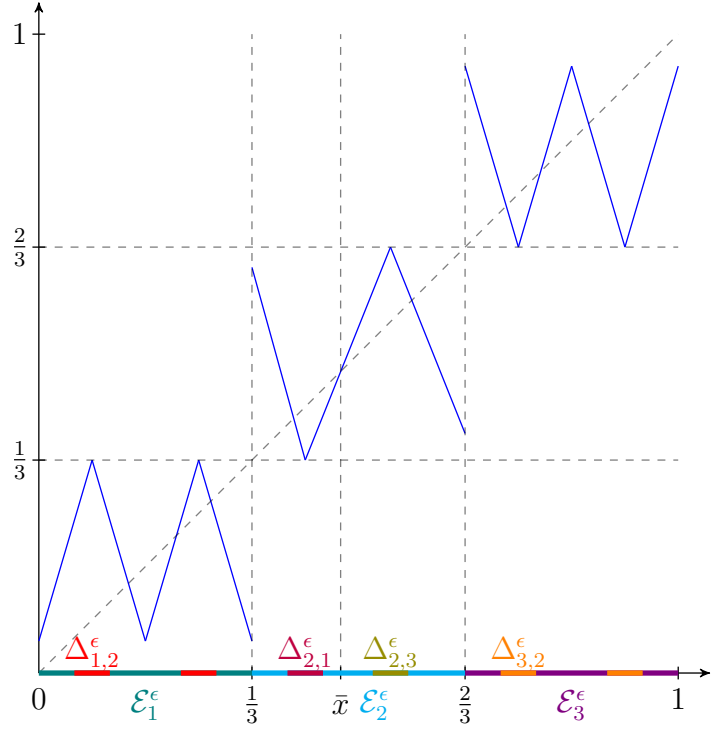

    Assumptions \textbf{(A1)}-\textbf{(A7)} clearly hold for this map by the sames arguments of the previous section\footnote{See Remark \ref{rem_mid_branch} for a discussion of why \textbf{(A4)} holds in this case.}. In particular, the LY inequalities for $P_{\eps}$, and for $P$ restricted to $BV(I_{i})$, hold since the slope of all branches have an absolute value greater than $2$ and their image are uniformly bounded from below. 

    In order to prove $\mathfrak{R}^{(1)}$, we can again apply Theorem \ref{main_theorem}. Conditions (a) and (b) are satisfied by the restricted map defined as that in Figure \ref{f3} in the components $I_{1}$ and $I_{3}$, with the same proof as the previous section. For the component $I_{2}$, similarly we consider $\tilde{T}^{i}_{\eps}$ for $i = 2$ as the map given by adding an extra branch when $T(x) + \sigma(\omega) \notin I_{2}$ as depicted in Figure \ref{f5}. Conditions (a) and (b) then follow by analogous arguments as the previous example, so the details are omitted, and $\mf{R}^{(1)}$ is in force.

    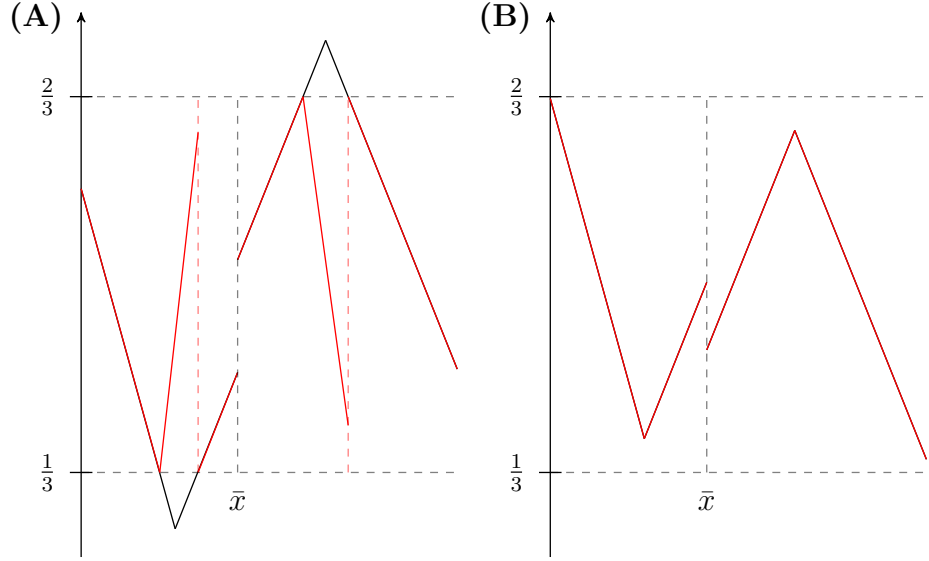
\begin{figure}[ht]
		\centering
		\begin{tikzpicture}[scale=15]
			\tikzstyle{hs} = [circle,draw=black, rounded corners,minimum width=3em, vertex distance=2.5cm, line width=1pt]
			\tikzstyle{hs2} = [circle,draw=black,dashed, rounded corners,minimum width=3em, vertex distance=2.5cm, line width=1pt]
			
			\node at (0.472,1/3-0.025) {$\bar{x}$};
            \node at (-0.03+1/3,1/3) {$\frac{1}{3}$};
			\node at (-0.03+1/3,2/3) {$\frac{2}{3}$};
            \node at (-0.04+1/3,2/3+0.07) {\textbf{(A)}};
			
			\begin{scope}[line width=0.5pt]
                \draw[->] (1/3,1/3-0.075) to (1/3,2/3+0.075);
                \draw[-] (-0.01+1/3,1/3) to (0.01+1/3,1/3);
				\draw[-] (-0.01+1/3,2/3) to (0.01+1/3,2/3);
                

                \draw[-] (1/3,0.635-0.05) to (5/12,1/3-0.05);
                \draw[-] (5/12,1/3-0.05) to (0.472,0.472-0.05);
                \draw[-] (0.472,0.472+0.05) to (0.55,2/3+0.05);
                \draw[-] (0.55,2/3+0.05) to (2/3,0.375+0.05);

                \draw[-,red] (1/3,0.635-0.05) to (0.403,1/3);
                \draw[-,red] (0.403,1/3) to (0.437,0.635);
                \draw[-,red] (0.437,1/3) to (0.472,0.472-0.05);
                \draw[-,red] (0.472,0.472+0.05) to (0.53,2/3);
                \draw[-,red] (0.53,2/3) to (0.57,0.375);
                \draw[-,red] (0.57,2/3) to (2/3,0.375+0.05);
				
				\draw[-,dashed,opacity = 0.5] (1/3,1/3) to (2/3,1/3);
				\draw[-,dashed,opacity = 0.5] (1/3,2/3) to (2/3,2/3);
				\draw[-,dashed,opacity = 0.5] (0.472,1/3) to (0.472,2/3);

                \draw[-,dashed,opacity = 0.5,red] (0.437,1/3) to (0.437,2/3);
                \draw[-,dashed,opacity = 0.5,red] (0.57,1/3) to (0.57,2/3);
			\end{scope}
		\end{tikzpicture}
        \begin{tikzpicture}[scale=15]
			\tikzstyle{hs} = [circle,draw=black, rounded corners,minimum width=3em, vertex distance=2.5cm, line width=1pt]
			\tikzstyle{hs2} = [circle,draw=black,dashed, rounded corners,minimum width=3em, vertex distance=2.5cm, line width=1pt]
			
			\node at (0.472,1/3-0.025) {$\bar{x}$};
            \node at (-0.03+1/3,1/3) {$\frac{1}{3}$};
			\node at (-0.03+1/3,2/3) {$\frac{2}{3}$};
            \node at (-0.04+1/3,2/3+0.07) {\textbf{(B)}};
			
			\begin{scope}[line width=0.5pt]
                \draw[->] (1/3,1/3-0.075) to (1/3,2/3+0.075);
                \draw[-] (-0.01+1/3,1/3) to (0.01+1/3,1/3);
				\draw[-] (-0.01+1/3,2/3) to (0.01+1/3,2/3);
                

                \draw[-] (1/3,0.635+0.03) to (5/12,1/3+0.03);
                \draw[-] (5/12,1/3+0.03) to (0.472,0.472+0.03);
                \draw[-] (0.472,0.472-0.03) to (0.55,2/3-0.03);
                \draw[-] (0.55,2/3-0.03) to (2/3,0.375-0.03);

                \draw[-,red] (1/3,0.635+0.03) to (5/12,1/3+0.03);
                \draw[-,red] (5/12,1/3+0.03) to (0.472,0.472+0.03);
                \draw[-,red] (0.472,0.472-0.03) to (0.55,2/3-0.03);
                \draw[-,red] (0.55,2/3-0.03) to (2/3,0.375-0.03);

                \draw[-,dashed,opacity = 0.5] (1/3,1/3) to (2/3,1/3);
				\draw[-,dashed,opacity = 0.5] (1/3,2/3) to (2/3,2/3);
				\draw[-,dashed,opacity = 0.5] (0.472,1/3) to (0.472,2/3);
			\end{scope}
		\end{tikzpicture}
		\caption{Map $\tilde{T}^{i}_{\epsilon}(\omega,\cdot)$ for $i = 2$ in red for fixed $\omega \in \Omega$ with $I_{2} = (1/3,2/3)$ when $T$ is the map in Figure \ref{f4}. The perturbed map satisfies $T_{\epsilon}(\omega,x) = T(x) + \sigma(\omega)$ for $x \in (1/3,\bar{x})$ and $T_{\epsilon}(\omega,x) = T(x) - \sigma(\omega)$ for $x \in (\bar{x},2/3)$, and it is represented in black for $\sigma(\omega) < 0$ in \textbf{(A)} and for $\sigma(\omega) > 0$ \textbf{(B)}. Observe that the maps $\tilde{T}^{i}_{\epsilon}(\omega,\cdot)$ and $T_{\epsilon}(\omega,\cdot)$ coincide for $x$ satisfying $T_{\epsilon}(\omega,x) \in (1/3,2/3)$.} \label{f5}
	\end{figure}

    \subsection{Condition $\mf{R}_{\mcL}^{(2)}$}

    Due to the lack of symmetry, condition \eqref{H0} is not direct in this case, and we will prove it by showing that \eqref{timeScale} and \eqref{timeScale2} hold. These conditions follow by estimating the probability of the paths of $X_{n}^{\eps}$ in the event $\{H(\mc{E}_{j}^{\eps}) < H^{+}(\check{\mc{E}}_{j}^{\eps}),X_{0}^{\eps} \in \mc{E}_{i}^{\eps}\}$. These are the paths such that $X_{0}^{\eps} \in \mc{E}_{i}^{\eps}$ and $X_{1}^{\eps} \in \mcE_{j}^{\eps}$ or, for a $n \geq 2$,
    \begin{equation}
        \label{valid_paths}
        X_{0}^{\eps} \in \mc{E}_{i}^{\eps}, X_{1}^{\eps} \in \Delta^{\eps},\dots,X_{n-1}^{\eps} \in \Delta^{\eps}, X_{n}^{\eps} \in \mcE_{j}^{\eps}.
    \end{equation}
    Throughout this section, we assume that $\eps > 0$ is small enough for the arguments to hold.

    We start with the case $i = 1$ and $j = 3$ for which $\Delta_{1,3}^{\eps} = \emptyset$.  First, it is clear that, for all $x \in I_{1}$,
    \begin{equation}
        \label{long_jumps}
        \mb{P}_{x}^{\eps}\left[X_{1}^{\eps} \in \mc{E}_{3}^{\eps} \cup \Delta_{2,3}^{\eps} \cup \Delta_{2,1}^{\eps} \cup \Delta_{3,2}^{\eps}\right] = 0.
    \end{equation}
    However, for a path to satisfy \eqref{valid_paths} it is necessary that, at some time $m \geq 1$, a jump from $I_{1} = \mcE_{1}^{\eps} \cup \Delta_{1,2}^{\eps}$ to $\mc{E}_{3}^{\eps} \cup \Delta_{2,3}^{\eps} \cup \Delta_{2,1}^{\eps} \cup \Delta_{3,2}^{\eps}$ occur. Therefore, by the strong Markov property and \eqref{long_jumps}, it holds
    \begin{align}
        \label{dec1} \nonumber
        &\mb{P}_{\mu^{\eps}}^{\eps}\left[H(\mc{E}_{3}^{\eps}) < H^{+}(\check{\mc{E}}_{3}^{\eps}),X_{0}^{\eps} \in \mc{E}_{1}^{\eps}\right] \\ \nonumber
        &= \sum_{m = 0}^{\infty} \mb{P}_{\mu^{\eps}}^{\eps}\left[H(\mc{E}_{3}^{\eps}) < H^{+}(\check{\mc{E}}_{3}^{\eps}),X_{0}^{\eps} \in \mc{E}_{1}^{\eps},X_{m}^{\eps} \in I_{1},X_{m+1}^{\eps} \in \mc{E}_{3}^{\eps} \cup \Delta_{2,3}^{\eps} \cup \Delta_{2,1}^{\eps} \cup \Delta_{3,2}^{\eps}\right]\\ \nonumber
        &\leq \sum_{m = 0}^{\infty} \mb{P}_{\mu^{\eps}}^{\eps}\left[X_{m}^{\eps} \in I_{1},X_{m+1}^{\eps} \in \mc{E}_{3}^{\eps} \cup \Delta_{2,3}^{\eps} \cup \Delta_{2,1}^{\eps} \cup \Delta_{3,2}^{\eps}\right]\\ \nonumber
        &= \sum_{m = 0}^{\infty} \mb{E}_{\mu^{\eps}}^{\eps}\left[\mb{P}_{X_{m}^{\eps}}\left[X_{m+1}^{\eps} \in \mc{E}_{3}^{\eps} \cup \Delta_{2,3}^{\eps} \cup \Delta_{2,1}^{\eps} \cup \Delta_{3,2}^{\eps}\right],X_{m}^{\eps} \in I_{1}\right]\\
        &\leq \sum_{m = 0}^{\infty} \sup\limits_{x \in I_{1}} \mb{P}_{x}\left[X_{1}^{\eps} \in \mc{E}_{3}^{\eps} \cup \Delta_{2,3}^{\eps} \cup \Delta_{2,1}^{\eps} \cup \Delta_{3,2}^{\eps}\right] = 0,
    \end{align}
    so, for instance, condition \eqref{timeScale2} holds for $i = 1$ and $j = 3$ if the denominator is greater than zero for $\eps > 0$. An analogous deduction yields the same result for $i = 3$ and $j = 1$.

    We turn to the cases where $\Delta_{i,j}^{\eps} \neq \emptyset$. Observing that 
    \begin{align}
        \label{hole_unstable}
        \sup\limits_{x \in \Delta^{\eps}} \mb{P}_{x}^{\eps}\left[X_{1}^{\eps} \in \Delta^{\eps}\right] = 0 & & \text{ and } & & \sup\limits_{x \in I_{i}} \mb{P}_{x}^{\eps}\left[X_{1}^{\eps} \in \Delta^{\eps}\setminus\Delta_{i}^{\eps}\right] = 0,
    \end{align}
    we conclude, again by the strong Markov property with a computation analogous to \eqref{dec1}, that the only possible paths \eqref{valid_paths} are those satisfying $X_{1}^{\eps} \in \mcE_{j}^{\eps}$ or $X_{1}^{\eps} \in \Delta_{i,j}^{\eps}$ and $X_{2}^{\eps} \in \mcE_{j}^{\eps}$. On the one hand, $\mb{P}_{\mu^{\eps}}^{\eps}\left[X_{0}^{\eps} \in \mc{E}_{i}^{\eps},X_{1}^{\eps} \in \mcE_{j}^{\eps}\right]$ is equal to
    \begin{align}
        \label{path1}
         \frac{1}{\eps + \eps^{q}} \int_{\mc{E}_{i}^{\eps}} p_{\eps}(x_{0}) \, \left(\eps^{q} - d(T(x_{0}),I_{j})\right) \, \mathds{1}\{d(T(x_{0}),I_{j}) < \eps^{q}\} \, dx_{0}.
    \end{align}
    On the other hand, taking $i = 1$ and $j = 2$, $\mb{P}_{\mu^{\eps}}^{\eps}\left[X_{0}^{\eps} \in \mc{E}_{i}^{\eps},X_{1}^{\eps} \in \Delta_{i,j}^{\eps},X_{2}^{\eps} \in \mcE_{j}^{\eps}\right]$ is equal to
    \begin{align}
        \label{path2} \nonumber
         \left(\frac{1}{\eps + \eps^{q}}\right)^{2} \int_{\mc{E}_{i}^{\eps}} \int_{\Delta_{i,j}^{\eps}}&  p_{\eps}(x_{0}) \, \mathds{1}\{x_{1} \in [T(x_{0}) - \eps,T(x_{0}) + \eps^{q}] \cap \Delta_{i,j}^{\eps}\} \times \\ &\left(\eps^{q} - d(T(x_{1}),I_{2})\right) \, \mathds{1}\{d(T(x_{1}),I_{2}) < \eps^{q}\} \, dx_{1} \, dx_{0}\, ,
    \end{align}
    and an analogous expression follows for other values of $i,j$, possibly considering that the noise is uniformly distributed in $[-\eps^{q},\eps]$ instead. Since $\text{Leb}(\Delta_{i,j}^{\eps}) > 0$, we conclude that \eqref{path2} is greater than zero for all $\eps > 0$. In particular, combining this fact with \eqref{dec1} yields \eqref{timeScale2}.

    In view of \eqref{path1} and \eqref{path2}, if $\Delta_{i,j}^{\eps} \neq \emptyset$, we conclude that the probability $\mb{P}_{\mu_{\eps}}\left[H(\mc{E}_{j}^{\eps}) < H^{+}(\check{\mc{E}}_{j}^{\eps})|X_{0}^{\eps} \in \mc{E}_{i}^{\eps}\right]$, which equals
    \begin{align*}
         \frac{1}{\mu_{\eps}(\mc{E}_{i}^{\eps})} \left[\mb{P}_{\mu^{\eps}}^{\eps}\left[X_{0}^{\eps} \in \mc{E}_{i}^{\eps},X_{1}^{\eps} \in \mcE_{j}^{\eps}\right] + \mb{P}_{\mu^{\eps}}^{\eps}\left[X_{0}^{\eps} \in \mc{E}_{1}^{\eps},X_{1}^{\eps} \in \Delta_{1,2}^{\eps},X_{2}^{\eps} \in \mcE_{2}^{\eps}\right]\right],
    \end{align*}
    varies continuously with the size of $\Delta_{i,j}^{\eps}$, which in turn varies continuously with the size of $B_{i,j}^{\eps}$ in its definition (cf. \eqref{holes}). Therefore, it is possible to choose the neighbourhoods $B_{i,j}^{\eps}$ in \eqref{holes} in a way such that \eqref{timeScale} holds. Hence, with this choice, \eqref{H0} is in force, so $\mathfrak{R}^{(2)}_{\mcL}$ holds, and the process is $\mcL$-metastable in which $\mcL$ is the generator of a Markov process in $S = \{1,2,3\}$ with rates $\theta(1,3) = \theta(3,1) = 0$, and the remaining ones given by the limit \eqref{H0} for the right choice of $\beta_{\eps}$. In particular, $\theta(1,2) = \theta(3,2)$ due to symmetry.
    
    \begin{remark}
        In order to compute $\beta_{\eps}$ or the ratio $\theta(2,1)/\theta(2,3)$ it is in principle necessary to estimate the density $p_{\eps}$, what is not straightforward and is outside the scope of this paper. If the map $T$ restricted to $(1/3,2/3)$ in Figure \ref{f4} was symmetric around $x = 1/2$, then clearly $\theta(2,1) = \theta(2,3)$, but it was considered as asymmetric to make the point that the value of $\theta(i,j)$ is not always obtained directly. Furthermore, due to this asymmetry, it is not direct if $\mu_{\eps}(I_{i}) = \pi(i)$ so, even though the stochastic stability follows from Corollary \ref{cor_stable}, the bound in Theorem \ref{prop_dTV_final} is not meaningful unless $\mu_{\eps}(I_{i})$ could be estimated.
    \end{remark}

    \begin{remark}
        Although condition \eqref{hole_unstable} simplifies the study of the paths in \eqref{valid_paths}, it does not always hold. For example, it fails for the map in Figure \ref{f2} \textbf{(A)}, where a more careful analysis would be needed if there were no symmetry. Considering uniformly distributed noises also simplifies the presentation, but the results should hold with other sub-Gaussian noises, at the cost of more technical details.
    \end{remark}
    
	\section{Final remarks}
	\label{Sec_RM}

    The approach to metastability for one-dimensional random dynamical systems presented in this paper differs substantially from the existing literature. While previous works rely primarily on the spectral analysis of transfer operators, in our approach this spectral theory serves instead to verify the assumptions and prove the sufficient condition $\mfM$ for metastability under the resolvent approach. In particular, the spectral techniques are applied to analyse the mixing properties of restricted maps such as those illustrated in Figures \ref{f3} and \ref{f5}, instead of discontinuous perturbed operators. This considerably simplifies the applications: compare the generality of the assumptions and the technical difficulties, or the lack thereof, of the results in Sections \ref{Sec_applications} and \ref{SecEx2} with \cite{bahsoun2013escape}, the closest paper to ours.  
    
    From the perspective of Markov process theory, this paper extends the resolvent approach to jump processes on uncountable state spaces. In particular, we develop a framework for analysing the metastability of a class of non-reversible processes, a task that is known to be more challenging than in the reversible setting.
    
	The resolvent approach to metastability for random dynamical systems could be generalised in several directions. The first one would be to consider perturbations of piecewise uniformly expanding maps in higher dimensions. Not only the resolvent approach generalises seamlessly to higher dimensions, but these systems have associated Banach spaces that are particularly adapted to the spectral analysis required by the theory. Possible issues could come only from the geometry of the holes and its control under iterations of different maps. 
    
    Another target are hyperbolic diffeomorphisms, eventually with singularities, like billiards. Even in this case, one disposes of very useful (anisotropic) Banach spaces to which apply the spectral stability results. The latter are of great help when the systems are uniformly expanding or hyperbolic. Otherwise, making use of induction or tower schemes, hyperbolicity could be first established on subsystems and then the desired statistical properties could be eventually recovered everywhere else. It was using this technique that metastability was proved for non-uniformly expanding maps of the interval in \cite{bahsoun2011metastability}. Those induction and tower schemes are very flexible, and they could be adapted to random versions of the previous systems as well.

    The theory could also be extended for systems with an attracting fixed point, that would generate a Markov process within the framework of \cite{beltran2010tunneling} and for which suitable sufficient conditions can be deduced based on the resolvent approach, as was done for the critical zero-range process in \cite{landim2021resolvent}. There is also the possibility of extending the approach to deterministic and quenched perturbations aiming to generalise the results of \cite{Dolgopyat_Wright} and \cite{gonzalez2025jumping} for almost all orbits starting in a metastable well instead of on average. The asymptotics of the diffusion coefficient of an observable of bounded variation could also be studied via the resolvent approach.
	
	\FloatBarrier 
	\section{Proof of results}
	\label{Sec_proofs}
	
	\subsection{Proof of Proposition \ref{prob_cone_ACIM}}

    That $X_{n}^{\eps}$ has at least one invariant measure, follows from \cite[Proposition~2]{rosenthal1995minorization}. This proposition states that, if there exists a probability distribution $\nu(\cdot)$, a constant $\vartheta > 0$ and $n_{\eps} < \infty$ such that, for all $x \in I$ and $A \in \mathfrak{B}_{I}$,
    \begin{align}
        \label{minorization}
        \int_{A} \rho_{\eps}^{\eps}(x,y) \, dy \geq \vartheta \, \nu(A),
    \end{align}
    then $X_{n}^{\eps}$ has at least one invariant measure. Inequality \eqref{minorization} is called a \textit{minorisation condition} (see \cite{rosenthal1995minorization} and the references therein for more details). If \eqref{cond_unique_ACIM} holds, then \eqref{minorization} holds with $\nu$ equal to the uniform measure in $A_{\eps}$.

    To show that the invariant measure is actually unique, we first note that \eqref{cond_unique_ACIM} implies
    \begin{align}
		\label{cond_unique_ACIM2}
         \inf_{x \in I, y \in A_{\eps}} \sum_{n = 1}^{\infty} 2^{-n} \, \rho_{\epsilon}^{n}(x,y) > 0.
	\end{align}
    The uniqueness then follows from \cite[Theorem~2]{asmussen2011new}, which implies that: if $X_{n}^{\eps}$ has an invariant measure, \eqref{cond_unique_ACIM2} holds and $\rho_{\eps}(x,y) = \rho_{\eps}(x,y)\chi_{A_{\eps}}(y)$, then it is a positive Harris recurrent Markov chain. Since $X_{n}^{\eps}$ is aperiodic by \textbf{(A4)}, it has then a unique invariant measure by the Aperiodic Ergodic Theorem \cite[Theorem~13.0.1]{meyn2012markov}. That this measure is absolutely continuous follows since the transition measure $\mbM_{\eps}^{x}$ is absolutely continuous by \textbf{(A2)}.
    
	\subsection{Proofs of Section \ref{SecR1}}	
		
	The proof of Proposition \ref{p03} relies on the following lemma, that is Lemma 6.6 of \cite{landim2021resolvent}, for which we present a proof for the sake of completeness. To easy notation, we redefine $\ms{P}_\epsilon(t)$ as the semigroup of the speeded-up process $\xi_{\epsilon}(\cdot)$.
	
	\begin{lemma}
		\label{lemma6.6}
		For all $T > 0$,
		\begin{align*}
			\sup\limits_{0 \leq t \leq T} \sup\limits_{x \in I} \left|F_{\epsilon}(x) - (\ms{P}_\epsilon(t)F_{\epsilon})(x) \right| \leq 2 \, T \, \lVert G \rVert_{\infty}.
		\end{align*}
	\end{lemma}
	\begin{proof}
		Fix $T > 0$ and $0 < t \leq T$. By the representation \eqref{formula_sol_reseq} of $F_{\epsilon}$ and the definition of semigroup (cf. \eqref{semigroup}) it holds, for $x \in I$,
		\begin{align*}
			(\ms{P}_\epsilon(t)F_{\epsilon})(x) &= \boldsymbol{E}_{x}^{\epsilon}\left[F_{\epsilon}(\xi_{\epsilon}(t))\right] = \boldsymbol{E}_{x}^{\epsilon}\left[\boldsymbol{E}_{\xi_{\epsilon}(t)}^{\epsilon}\left[\int_{0}^{\infty} e^{-\lambda s}G(\xi_{\epsilon}(s)) \ ds\right]\right]\\
			&= \boldsymbol{E}_{x}^{\epsilon}\left[\int_{0}^{\infty} e^{-\lambda s}G(\xi_{\epsilon}(t + s)) \ ds\right].
		\end{align*}
		By a change of variables, $(\ms{P}_\epsilon(t)F_{\epsilon})(x)$ can then be written as
		\begin{align*}
			&\boldsymbol{E}_{x}^{\epsilon}\left[\int_{t}^{\infty} e^{-\lambda s} G(\xi_{\epsilon}(s)) \ ds\right] + \boldsymbol{E}_{x}^{\epsilon}\left[\int_{t}^{\infty} \left[e^{-\lambda (s - t)} - e^{-\lambda s}\right] G(\xi_{\epsilon}(s)) \ ds\right]\\
			&\leq  F_{\epsilon}(x) + \lVert G \rVert_{\infty} \int_{0}^{t} e^{-\lambda s} \ ds + \lVert G \rVert_{\infty} \int_{t}^{\infty} \left[e^{-\lambda (s - t)} - e^{-\lambda s}\right] \ ds\\
			&\leq F_{\epsilon}(x) + t \, \lVert G \rVert_{\infty} + \lVert G \rVert_{\infty} \frac{1 - e^{-\lambda t}}{\lambda} \leq F_{\epsilon}(x) + 2 \, t \, \lVert G \rVert_{\infty}
		\end{align*}
		since $1 - e^{-\lambda t} \leq \lambda t$.
	\end{proof}
	
	\begin{proof}[Proof of Proposition \ref{p03}]
		We start by proving that, for all $i \in S, x \in \mcE_{i}^{\eps}$ and $t > 0$,
		\begin{equation}
			\label{in_sg}
			|(\tilde{\ms{P}}^{i}_\epsilon(t)F_{\epsilon})(x) - (\ms{P}_\epsilon(t)F_{\epsilon})(x)| \leq 2 \, \lVert F_{\epsilon} \rVert_{\infty} \, \bs{P}_{x}^{\epsilon}[H((\mcV_{i}^{\epsilon})^{c}) \leq t].
		\end{equation}
		By the definition of the semigroup (cf. \eqref{semigroup}), $(\ms{P}_\epsilon(t)F_{\epsilon})(x)$ is equal to 
		\begin{align*}
			&\bs{E}_{x}^{\epsilon}\left[F_{\epsilon}(\xi_{\epsilon}(t))\right] = \bs{E}_{x}^{\epsilon}\left[F_{\epsilon}(\xi_{\epsilon}(t)),H((\mcV_{i}^{\epsilon})^{c}) > t\right] + \bs{E}_{x}^{\epsilon}\left[F_{\epsilon}(\xi_{\epsilon}(t)),H((\mcV_{i}^{\epsilon})^{c}) \leq t\right].
		\end{align*}
		Since $\xi_{\epsilon}(t) = \tilde{\xi}_{\epsilon}^{i}(t)$ for all $t < H((\mcV_{i}^{\epsilon})^{c})$, it holds
		\begin{align*}
			\bs{E}_{x}^{\epsilon}\left[F_{\epsilon}(\xi_{\epsilon}(t)),H((\mcV_{i}^{\epsilon})^{c}) > t\right] &= \bs{E}_{x}^{\epsilon}\left[F_{\epsilon}(\tilde{\xi}_{\epsilon}^{i}(t)),H((\mcV_{i}^{\epsilon})^{c}) > t\right]\\
			&= \bs{E}_{x}^{\epsilon}\left[F_{\epsilon}(\tilde{\xi}_{\epsilon}^{i}(t))\right] - \bs{E}_{x}^{\epsilon}\left[F_{\epsilon}(\tilde{\xi}_{\epsilon}^{i}(t)),H((\mcV_{i}^{\epsilon})^{c}) \leq t\right]\\
			&= (\tilde{\ms{P}}^{i}_\epsilon(t)F_{\epsilon})(x) - \bs{E}_{x}^{\epsilon}\left[F_{\epsilon}(\tilde{\xi}_{\epsilon}^{i}(t)),H((\mcV_{i}^{\epsilon})^{c}) \leq t\right].
		\end{align*}
		We conclude that
		\begin{align*}
			|(\tilde{\ms{P}}^{i}_\epsilon(t)F_{\epsilon})(x) &- (\ms{P}^{i}_\epsilon(t)F_{\epsilon})(x)| \\
			&= \left|\bs{E}_{x}^{\epsilon}\left[F_{\epsilon}(\tilde{\xi}_{\epsilon}^{i}(t)),H((\mcV_{i}^{\epsilon})^{c}) \leq t\right] + \bs{E}_{x}^{\epsilon}\left[F_{\epsilon}(\xi_{\epsilon}(t)),H((\mcV_{i}^{\epsilon})^{c}) \leq t\right]\right|\\
			&\leq 2 \, \lVert F_{\epsilon} \rVert_{\infty} \, \boldsymbol{P}_{x}^{\epsilon}[H((\mcV_{i}^{\epsilon})^{c}) \leq t].
		\end{align*}
		
		By combining Lemma \ref{lemma6.6} with $T = t = \mb h_\epsilon$, \eqref{in_sg} and hypothesis \eqref{23} we conclude by a triangular inequality that
		\begin{equation}
			\label{lim_semi_F}
			\lim\limits_{\epsilon \to 0} \sup\limits_{x \in \mcE_{i}^{\eps}} \left|F_{\epsilon}(x) - (\tilde{\ms{P}}^{i}_\epsilon(\mb h_\epsilon)F_{\epsilon})(x)\right| = 0.
		\end{equation}
		Fix $i \in S, x \in \mcE_{i}^{\eps}$ and $0 < \varsigma < \varsigma_{0}$. By the definition of total variation distance
		\begin{equation}
			\label{in_dtv}
			\left|(\tilde{\ms{P}}^{i}_\epsilon(\mb h_\epsilon)F_{\epsilon})(x) - \int_{\mcV_{i}^{\epsilon}} F_{\eps}(y) \ d\tilde{\mu}_{\epsilon}^{i}(y)\right| \leq 2 \ \lVert F_{\epsilon} \rVert_{\infty} \ d^i_{\rm TV} (\delta_x \tilde{\ms P}^{i}_\epsilon(\mb h_\epsilon) \,,\,  \tilde{\mu}_{\epsilon}^{i}).
		\end{equation}
		Since $\tilde{\xi}_{\eps}^{i}(\cdot)$ is ergodic,
		\begin{equation*}
			d^i_{\rm TV} (\delta_x \tilde{\ms P}^{i}_\epsilon(t) \,,\,  \tilde{\mu}_{\epsilon}^{i}) = d^i_{\rm TV} (\delta_x \tilde{\ms P}^{i}_\epsilon(t) \,,\,  \tilde{\mu}_{\epsilon}^{i}\tilde{\ms P}^{i}_\epsilon(t))
		\end{equation*}
		is decreasing in $t > 0$, and it follows from \eqref{33} that the right-hand side of \eqref{in_dtv} is bounded from above by
		\begin{align}
			\label{b_dtv}
			2 \ \lVert F_{\epsilon} \rVert_{\infty} \sup_{x \in \mcE_{i}} d^i_{\rm TV} (\delta_x \tilde{\ms P}^{i}_\epsilon(t^{\epsilon,i}_{\rm mix} (\varsigma)) \,,\,  \tilde{\mu}_{\epsilon}^{i}) = 2 \, \lVert F_{\epsilon} \rVert_{\infty} \, \varsigma
		\end{align}
		by definition of mixing time. The result follows by combining \eqref{lim_semi_F}, \eqref{in_dtv} and \eqref{b_dtv} since $\lVert F_{\epsilon} \rVert_{\infty}$ is uniformly bounded in $\epsilon$ and $0 < \varsigma < \varsigma_{0}$ is arbitrary.
	\end{proof}
	
	We state and prove a lemma that will be useful to prove Proposition \ref{prop_not_jump}. For $A \in \mfB_{I}$, denote by
	\begin{equation}
		\label{def_hit_MC}
		\tau(A) = \min\{n \geq 0: X_{n}^{\epsilon} \in A\}
	\end{equation}
	the hitting time of $A$ by the embedded chain $X_{n}^{\epsilon}$. Recall that $H(A)$ is the hitting time of $A$ by $\xi_{\eps}(\cdot)$ (cf. \eqref{htime}).
	
	\begin{lemma}
		\label{lemma_discrete_cont}
		For all $A \in \mfB_{I},$ $n \geq 1$ and $c > 0$
		\begin{equation*}
			\sup\limits_{x \in I\setminus A} \bs{P}_{x}^{\epsilon}\left[|\beta_{\epsilon} H(A) - n| > c \, n|\tau(A) = n\right] \leq \frac{1}{c^{2} n}
		\end{equation*}
        in which the supremum is over $x \in I\setminus A$ such that $\bs{P}_{x}^{\epsilon}\left[\tau(A) = n\right] > 0$.
	\end{lemma}
	\begin{proof}
		Fix $A \in \mfB_{I}$ and, to easy notation, let $\tau \coloneqq \tau(A)$. Let $Z_{\tau}$ be a random variable with distribution $Gamma(\tau,1)$. Conditioned on $\tau = n$, the random variable $\beta_{\epsilon} H(A)$ has the same distribution as $Z_{\tau}$. This is the case since $\tau$ is the number of jumps to reach $A$ and each jump of $\eta_{\eps}(\cdot)$ takes an exponential time with rate $1$ to occur. Since the $\tau$ jumping times are independent, the time for $\eta_{\eps}(\cdot)$ to reach $A$, that is their sum, has a distribution $Gamma(n,1)$ when $\tau = n$ that is the distribution of the sum of $n$ exponentially distributed independent random variables with rate $1$. The multiplication by $\beta_{\epsilon}$ is due to the fact that $H(A)$ refers to the speeded-up process $\xi_{\eps}(\cdot)$ and $\tau$ refers to the embedded Markov chain. Since the mean and variance of a $Gamma(n,1)$ distribution equals $n$, by Chebyshev inequality, for $c > 0$, $n \geq 1$ and $x \in I\setminus A$ with $\bs{P}_{x}^{\epsilon}\left[\tau(A) = n\right] > 0$ fixed,
		\begin{align*}
			\boldsymbol{P}_{x}^{\epsilon}[|\beta_{\epsilon} H(A) - n| >  cn| \tau(A) = n] \leq \frac{Var(Z_{n})}{c^{2}n^{2}} = \frac{1}{c^{2} n}.
		\end{align*}
	\end{proof}
	
	\begin{proof}[Proof of Proposition \ref{prop_not_jump}]
		Recall the definition of $\tau(I_{i}^{c})$ in \eqref{def_hit_MC} and denote $\tau \coloneqq \tau(I_{i}^{c})$ to ease notation. For each $n \geq 1$ and $x \in \mcE_{i}^{\epsilon}$,
		\begin{align*}
			\boldsymbol{P}_{x}^{\epsilon}[\tau = n] &= \boldsymbol{P}_{x}^{\epsilon}[X_{m}^{\epsilon} \in I_{i} \ \forall m < n,X_{n}^{\epsilon} \in I_{i}^{c}] \leq \boldsymbol{P}_{x}^{\epsilon}[X_{n-1}^{\epsilon} \in I_{i},X_{n}^{\epsilon} \in I_{i}^{c}]\\
			&\leq \sup\limits_{y \in I_{i}} \boldsymbol{P}_{y}^{\epsilon}[X_{1}^{\epsilon} \in I_{i}^{c}] = q_{\epsilon}
		\end{align*}
		in which the last inequality is due to the strong Markov property. Therefore, 
		\begin{align}
			\label{in01}
			\boldsymbol{P}_{x}^{\epsilon}[\tau \leq 2 \, \beta_{\epsilon} \, \boldsymbol{h}_{\epsilon}] = \sum_{n=1}^{\lfloor 2 \beta_{\epsilon} \mb h_{\epsilon} \rfloor} \boldsymbol{P}_{x}^{\epsilon}[\tau = n] \leq 2 \, q_{\epsilon} \, \beta_{\epsilon} \, \boldsymbol{h}_{\epsilon}.
		\end{align}
		By Lemma \ref{lemma_discrete_cont} we conclude that, for $n \geq 1$ and $x \in \mcE_{i}^{\epsilon}$ with $\bs{P}_{x}^{\epsilon}\left[\tau(I_{i}^{c}) = n\right] > 0$ fixed,
		\begin{align}
			\label{in02}
			\boldsymbol{P}_{x}^{\epsilon}[|\beta_{\epsilon} H(I_{i}^{c}) - n| >  n/2| \tau(I_{i}^{c}) = n] \leq \frac{4}{n}.
		\end{align}
		Now, observe that, for $x \in \mcE_{i}^{\epsilon}$ fixed,
		\begin{align}
			\label{in00}
			\boldsymbol{P}_{x}^{\epsilon}[H(I_{i}^{c}) \leq \mb h_{\epsilon}] \leq \boldsymbol{P}_{x}^{\epsilon}[\tau \leq 2 \, \beta_{\epsilon} \, \mb h_{\epsilon}] + \boldsymbol{P}_{x}^{\epsilon}[H(I_{i}^{c}) \leq \mb h_{\epsilon},\tau > 2 \,\beta_{\epsilon} \, \mb h_{\epsilon}]
		\end{align}
		and
		\begin{align}
			\label{in03} \nonumber
			\boldsymbol{P}_{x}^{\epsilon}[H(I_{i}^{c}) \leq \mb h_{\epsilon},\tau &> 2 \, \beta_{\epsilon} \, \mb h_{\epsilon}] = \sum_{n = \lceil 2 \beta_{\epsilon} \mb h_{\epsilon} \rceil}^{\infty} \boldsymbol{P}_{x}^{\epsilon}[\beta_{\epsilon} H(I_{i}^{c}) \leq \beta_{\epsilon} \, \mb h_{\epsilon}|\tau = n] \boldsymbol{P}_{x}^{\epsilon}[\tau = n]\\ \nonumber
			&\leq \sum_{n = \lceil 2 \beta_{\epsilon} \mb h_{\epsilon} \rceil}^{\infty} \boldsymbol{P}_{x}^{\epsilon}[\beta_{\epsilon} H(I_{i}^{c}) \leq n/2|\tau = n] \boldsymbol{P}_{x}^{\epsilon}[\tau = n]\\ \nonumber
			&= \sum_{n = \lceil 2 \beta_{\epsilon} \mb h_{\epsilon} \rceil}^{\infty} \boldsymbol{P}_{x}^{\epsilon}[-\beta_{\epsilon} H(I_{i}^{c}) \geq -n/2|\tau = n] \boldsymbol{P}_{x}^{\epsilon}[\tau = n]\\ \nonumber
			&= \sum_{n = \lceil 2 \beta_{\epsilon} \mb h_{\epsilon} \rceil}^{\infty} \boldsymbol{P}_{x}^{\epsilon}[n - \beta_{\epsilon} H(I_{i}^{c}) > n/2|\tau = n] \boldsymbol{P}_{x}^{\epsilon}[\tau = n]\\
			&\leq \sum_{n = \lceil 2 \beta_{\epsilon} \mb h_{\epsilon} \rceil}^{\infty} \frac{4}{n} \boldsymbol{P}_{x}^{\epsilon}[\tau = n] \leq \frac{2}{\beta_{\epsilon} \, \mb h_{\epsilon}}
		\end{align}
		in which the penultimate inequality is due to \eqref{in02}. The result follows by combining \eqref{in01}, \eqref{in00} and \eqref{in03} with hypotheses \eqref{hyp_not_jump}.
	\end{proof}
	
	\begin{proof}[Proof of Lemma \ref{lemma_Mix_nt}]
		Fix $i \in S$ and for $t > 0$ let $N(t) = \max\{n: \tilde{\tau}_{n} \leq t\}$ be the random number of jumps of $\tilde{\xi}_{\epsilon}^{i}(\cdot)$ until time $t$. Since $\tilde{\tau}_{n+1} - \tilde{\tau}_{n}$ follows an exponential distribution with rate $\beta_{\epsilon}$, $N(\cdot)$ is a Poisson process with rate $\beta_{\epsilon}$. In particular, $N(t)$ follows a Poisson distribution with mean $\beta_{\epsilon}t$.
		
		For $x \in \mcE_{i}^{\eps}$, $n \geq 1$ and $t > 0$ fixed, by definition \eqref{dtv}, denoting $\tilde{\mu}^{i}_{\eps}\left[J\right]$ as the expectation of $J: I_{i} \mapsto \mbR$ under $\tilde{\mu}^{i}_{\eps}$,
        \begin{align}
            \label{dTV_tn} \nonumber
			d^i_{\rm TV} (\delta_x \tilde{\ms P}^{i}_\epsilon(t) \,,\,  &\tilde{\mu}_{\epsilon}^{i}) \leq d^i_{\rm TV} (\delta_x \tilde{\ms P}^{i}_\epsilon(t \wedge  \tilde{\tau}_{n}) \,,\,  \tilde{\mu}_{\epsilon}^{i})\\ \nonumber
            &= \frac{1}{2}\sup\limits_{J} \left|\bs{E}_{x}^{\epsilon}\left[J(\tilde{\xi}_{\epsilon}^{i}(t \wedge  \tilde{\tau}_{n})) - \tilde{\mu}^{i}_{\eps}\left[J\right]\right]\right|\\ \nonumber
			&\leq \frac{1}{2}\sup\limits_{J} \left|\bs{E}_{x}^{\epsilon}\left[J(\tilde{\xi}_{\epsilon}^{i}(\tilde{\tau}_{n})) - \tilde{\mu}^{i}_{\eps}\left[J\right],N(t) > n\right]\right| + \bs{P}_{x}^{\epsilon}\left[N(t) \leq n\right]\\ 
            &\leq d^i_{\rm TV} (\delta_x \tilde{\ms P}^{i}_\epsilon(\tilde{\tau}_{n}) \,,\,  \tilde{\mu}_{\epsilon}^{i}) + 2 \, \bs{P}_{x}^{\epsilon}\left[N(t) \leq n\right]
		\end{align}
		in which the first inequality follows from the fact that $d^i_{\rm TV} (\delta_x \tilde{\ms P}^{i}_\epsilon(t) \,,\,  \tilde{\mu}_{\epsilon}^{i})$ is decreasing in $t$ as $\tilde{\xi}_{\epsilon}^{i}(\cdot)$ is ergodic. For $n = \lfloor \beta_{\epsilon}t/2 \rfloor$, by the Chernoff's bounds for Poisson distribution (see \cite[Theorem~4.5]{mitzenmacher2017probability}),
		\begin{align}
            \label{chernoff}
			\bs{P}_{x}^{\epsilon}\left[N(t) \leq n\right] \leq \bs{P}_{x}^{\epsilon}\left[N(t) \leq \beta_{\epsilon}t/2 \right] \leq \left(\frac{2}{e}\right)^{\frac{\beta_{\epsilon}t}{2}}.
		\end{align}
		For $0 < \varsigma < 1$, taking $t = 2n_{mix}^{\epsilon,i}(\varsigma)/\beta_{\epsilon}$ and $n = n_{mix}^{\epsilon,i}(\varsigma)$ in \eqref{dTV_tn}, we conclude that 
		\begin{align}
			\label{in_mix_dTV}
			d^i_{\rm TV} &(\delta_x \tilde{\ms P}^{i}_\epsilon(2n_{mix}^{\epsilon,i}(\varsigma)/\beta_{\epsilon}) \,,\,  \tilde{\mu}_{\epsilon}^{i}) \leq \varsigma + 2 \left(\frac{2}{e}\right)^{n_{mix}^{\epsilon,i}(\varsigma)}
		\end{align}
        by \eqref{chernoff} and the definition of $n_{mix}^{\epsilon,i}(\varsigma)$ (cf. \eqref{def_nmix}). 
        
        We recall the following property of the mixing time: for each pair $0 < \varsigma,\varsigma^{\prime} < 1/2$ there exists a constant $C_{\varsigma,\varsigma^{\prime}}$, that does not depend on $\epsilon$, such that $t_{mix}^{\epsilon,i}(\varsigma) \leq C_{\varsigma,\varsigma^{\prime}} t_{mix}^{\epsilon,i}(\varsigma^{\prime})$ (see for example \cite[Section~4.5]{levin2017markov}). Therefore, if $\varsigma < 1/2 - 2\left(2/e\right)^{n_{mix}^{\epsilon,i}(\varsigma)}$, then there exists a constant $C_{\varsigma} > 0
		$ such that
		\begin{align*}
			t_{mix}^{\epsilon,i}\left(\varsigma\right) \leq C_{\varsigma}t_{mix}^{\epsilon,i}\left(\varsigma + 2\left(2/e\right)^{n_{mix}^{\epsilon,i}(\varsigma)}\right) \leq 2C_{\varsigma} \frac{n_{mix}^{\epsilon,i}(\varsigma)}{\beta_{\epsilon}}
		\end{align*}
		in which the last inequality is a consequence of \eqref{in_mix_dTV} and the definition of $t_{mix}^{\epsilon,i}\left(\varsigma\right)$ (cf. \eqref{def_tmix}). The result follows by taking $\varsigma_{0} = \sup\{\varsigma < 1/2: \lim\sup_{\epsilon \to 0} \left(2/e\right)^{n_{mix}^{\epsilon,i}(\varsigma)} < 1/4\}$. 
	\end{proof}
	
	\subsection{Potential theory and proof of Proposition \ref{cor_7.3}}
	
	In order to prove Proposition \ref{cor_7.3}, we introduce some concepts from potential theory for non-reversible Markov jump processes in the uncountable state space $I$. We refer to \cite{bovier2016metastability} for a presentation of potential theory in the context of Markov process with countable state spaces and of diffusions. Denote by $\langle \cdot,\cdot \rangle_{\mu_{\epsilon}}$ the scalar product in $L^{2}(\mu_{\epsilon})$ which satisfies for $F,G \in L^{2}(\mu_{\epsilon})$
	\begin{align*}
		\langle F,G \rangle_{\mu_{\epsilon}} = \int_{I} F(x) \, G(x) \, p_{\epsilon}(x) \ dx\, ,
	\end{align*}
	recalling that $p_{\epsilon}$ is the density of $\mu_{\epsilon}$. The Dirichlet form associated with the generator $\mcL_{\epsilon}$ of $\xi_{\epsilon}(\cdot)$ is the functional $D_{\epsilon}$ acting on functions $F: I \to \mbR$ as
	\begin{align}
		\label{dir_fom}
		D_{\epsilon}(F) \coloneqq \frac{\beta_{\epsilon}}{2} \int_{I} \int_{I} p_{\epsilon}(x) \rho_{\epsilon}(x,y) [F(x) - F(y)]^{2} \, dx \, dy = \langle F,(-\mcL_{\epsilon})F \rangle_{\mu_{\epsilon}}
	\end{align}
	in which the second equality holds for $F$ in the support of $\mcL_{\epsilon}$.
	
	Denote by $\mcL_{\epsilon}^{\dagger}$ the adjoint of the generator $\mcL_{\epsilon}$ in $L^{2}(\mu_{\epsilon})$. This operator is given by
	\begin{align*}
		(\mcL_{\epsilon}^{\dagger}F)(x) = \beta_{\epsilon} \int_{I} \rho_{\epsilon}^{\dagger}(x,y)[F(y) - F(x)] \ dy
	\end{align*}
	in which $\rho_{\epsilon}^{\dagger}: I^{2} \to \mbR_{+}$ satisfies the detailed balance equation
	\begin{align}
		\label{detail_balance}
		p_{\epsilon}(x) \rho_{\epsilon}(x,y) = p_{\epsilon}(y) \rho_{\epsilon}^{\dagger}(y,x)
	\end{align}
	for all $x,y \in I$. In particular, $\rho_{\epsilon}^{\dagger}$ is a transition density and $\mcL_{\epsilon}^{\dagger}$ is the generator of a Markov process that we denote by $\xi_{\epsilon}^{\dagger}(\cdot)$. Let $\boldsymbol{P}_{x}^{\epsilon,\dagger}$ be the probability measure induced by $\xi_{\epsilon}^{\dagger}(\cdot)$ on $D(\mbR_{+},I)$ starting from $x \in I$. Expectation with respect to this measure is denoted by $\boldsymbol{E}_{x}^{\epsilon,\dagger}$. Let $D_{\epsilon}^{\dagger}$ be the Dirichlet form associated with $\mcL_{\epsilon}^{\dagger}$.
	
	For $A \in \mfB_{I}$, let
	\begin{align*}
		H^{\dagger}(A) = \inf\{t \geq 0: \xi_{\epsilon}^{\dagger}(t) \in A\}
	\end{align*}
	be the hitting time of the set $A$ by the adjoint process $\xi_{\epsilon}^{\dagger}(\cdot)$. Fix two disjoint non-empty subsets $A,B \in \mfB_{I}$. The equilibrium potential between $A$ and $B$ with respect to the process $\xi_{\epsilon}^{\dagger}(\cdot)$ is the function $h_{A,B}^{\epsilon,\dagger}: I \to [0,1]$ given by
	\begin{align}
		\label{def_EP}
		h_{A,B}^{\epsilon,\dagger}(x) = \boldsymbol{P}_{x}^{\epsilon,\dagger}[H^{\dagger}(A) < H^{\dagger}(B)].
	\end{align}
	From now on, we assume that both $A$ and $B$ are the union of a finite number of disjoint non-degenerate intervals. In this case, $h_{A,B}^{\epsilon,\dagger}$ is the solution of the Dirichlet problem
	\begin{align}
		\label{dir_problem}
		\begin{cases}
			(\mcL_{\epsilon}^{\dagger}h_{A,B}^{\epsilon,\dagger})(x) = 0, & \forall x \in I\setminus(A \cup B)\\
			h_{A,B}^{\epsilon,\dagger}(x) = 1, & \forall x \in A\\
			h_{A,B}^{\epsilon,\dagger}(x) = 0, & \forall x \in B
		\end{cases}.
	\end{align}
	Analogously, denote by $h_{A,B}^{\epsilon}$ the equilibrium potential with respect to the process $\xi_{\epsilon}(\cdot)$. 
	
	For $A \in \mfB_{I}$, let
	\begin{align*}
		H^{\dagger,+}(A) = \inf\{t \geq \tau_{1}^{\dagger}: \xi_{\epsilon}^{\dagger}(t) \in A\} \text{ in which } \tau_{1}^{\dagger} = \inf\{t \geq 0: \xi_{\epsilon}^{\dagger}(t) \neq \xi_{\epsilon}^{\dagger}(0)\}
	\end{align*}
    be the hitting time of $A$ by $\xi_{\epsilon}^{\dagger}(\cdot)$ after at least one jump. The capacity between $A$ and $B$ for the process $\xi_{\epsilon}(\cdot)$ and the adjoint $\xi_{\epsilon}^{\dagger}(\cdot)$ are defined by, respectively,
	\begin{align}
		\label{cap} \nonumber
		\text{cap}_{\epsilon}(A,B) &= \beta_{\epsilon} \int_{A} p_{\epsilon}(x) \boldsymbol{P}_{x}^{\epsilon}[H(B) < H^{+}(A)] \, dx \text{ and }\\
		\text{cap}_{\epsilon}^{\dagger}(A,B) &= \beta_{\epsilon} \int_{A} p_{\epsilon}(x) \boldsymbol{P}_{x}^{\epsilon,\dagger}[H^{\dagger}(B) < H^{\dagger,+}(A)] \, dx.
	\end{align}
	As $A$ is the union of a finite number of disjoint intervals, the integrals above are well-defined.
    
    For $x \in A$,
	\begin{align}
        \label{cap0} \nonumber
		-(\mcL_{\epsilon}h_{A,B}^{\epsilon})(x) &= \beta_{\epsilon} \int_{I} \rho_{\epsilon}(x,y)[1 - \boldsymbol{P}_{y}^{\epsilon}[H(A) < H(B)]] \, dy\\ \nonumber
		&= \beta_{\epsilon} \int_{I} \rho_{\epsilon}(x,y)\boldsymbol{P}_{y}^{\epsilon}[H(B) < H(A)] \, dy\\ \nonumber
		&= \beta_{\epsilon} \int_{I} \rho_{\epsilon}(x,y)\boldsymbol{P}_{y}^{\epsilon}[H(B) < H(A)] \, \chi_{A^{c}}(y) \, dy\\
		&= \beta_{\epsilon} \, \boldsymbol{P}_{x}^{\epsilon}[H(B) < H^{+}(A)] \, dy.
	\end{align}
	Since $h_{A,B}^{\epsilon}$ is the solution of a Dirichlet problem analogous to \eqref{dir_problem}, and it is in the support of $\mcL_{\epsilon}$, it follows from \eqref{cap0} that
	\begin{align}
		\label{cap2}
		D_{\epsilon}(h_{A,B}^{\epsilon}) = \langle h_{A,B}^{\epsilon},-(\mcL_{\epsilon}h_{A,B}^{\epsilon}) \rangle_{\mu_{\epsilon}} = \text{cap}_{\epsilon}(A,B)
	\end{align}
	and an analogous identity holds for $\text{cap}_{\epsilon}^{\dagger}(A,B)$.
	
	We prove that $\text{cap}_{\epsilon}(A,B) = \text{cap}_{\epsilon}^{\dagger}(A,B)$. This result relies on an extension of \cite[Lemma~2.3]{gaudilliere2014dirichlet} for Markov jump processes in the uncountable state space $I$.
	
	\begin{lemma}
		\label{lemma_cap}
		Let $A,B \subseteq I$ be disjoint sets that are unions of a finite number of non-degenerate intervals. Then, $\text{cap}_{\epsilon}(A,B) = \text{cap}_{\epsilon}^{\dagger}(A,B)$.		
	\end{lemma}
	
	Before proving Lemma \ref{lemma_cap}, we state and prove a result about the distribution of $\xi_{\epsilon}(H(B))$ when the process starts from $x \in A$.
	
	\begin{lemma}
		\label{lemma_density}
		If $A,B \in \mfB_{I}$ are disjoint sets that are unions of a finite number of non-degenerate intervals, then, for $x \in A$,
		\begin{align*}
			\bs{P}_{x}^{\epsilon}\left[H(B) < H^{+}(A)\right] = \int_{B} \bs{P}_{x}^{\epsilon}\left[H(\{y\}) = H^{+}(A \cup B)\right] \, dy
		\end{align*}
		in which the probability inside the integral is actually a density function. An analogous result holds for the hitting times of the adjoint process.
	\end{lemma}
	\begin{proof}
		Fix $x \in A$. Since
		\begin{align*}
			\bs{P}_{x}^{\epsilon}\left[H(B) < H^{+}(A)\right] = \bs{P}_{x}^{\epsilon}\left[H^{+}(B) = H^{+}(A \cup B)\right],
		\end{align*}
		it is enough to show that the probability measure
		\begin{align*}
			\nu_{x}^{\epsilon}(C) \coloneqq \nu_{x,A,B}^{\epsilon}(C) = \bs{P}_{x}^{\epsilon}\left[H^{+}(C) = H^{+}(A \cup B)\right]
		\end{align*}
		for $C \in \mfB_{A \cup B}$ is absolutely continuous wrt Lebesgue measure. Define for $A \in \mfB_{I}$, $\tau^{+}(A) = \min\{n \geq 1: X_{n}^{\epsilon} \in A\}$ the hitting time of $A$ by the embedded chain after at least one jump. The result follows since
		\begin{align*}
			\nu_{x}^{\epsilon}(C) &= \bs{P}_{x}^{\epsilon}\left[\tau^{+}(C) = \tau^{+}(A \cup B)\right]\\
			&= \sum_{n = 1}^{\infty} \bs{P}_{x}^{\epsilon}\left[\tau^{+}(C) = \tau^{+}(A \cup B),\tau^{+}(A \cup B) = n\right]\\
			&= \sum_{n = 1}^{\infty} \bs{P}_{x}^{\epsilon}\left[X_{m}^{\epsilon} \in I\setminus(A \cup B) \, \forall 1 \leq m \leq n -1,X_{n}^{\epsilon} \in C\right]\\
			&= \sum_{n = 1}^{\infty} \int_{C} \, dy\int_{(I\setminus(A \cup B))^{n-1}} \!\!\!\!\!\!\!\!\!\!\!\!\!\!\!\!\!\!\!\!\!\!\!\!\!dz_{1} \cdots dz_{n-1} \, \rho_{\epsilon}(x,z_{1})\rho_{\epsilon}(z_{n-1},y) \prod_{m = 1}^{n-2} \rho_{\epsilon}(z_{m},z_{m + 1})  \\
			&= \int_{C} \bs{P}_{x}^{\epsilon}\left[H(\{y\}) = H^{+}(A \cup B)\right] \, dy
		\end{align*}
		by defining the density $\bs{P}_{x}^{\epsilon}\left[H(\{y\}) = H^{+}(A \cup B)\right]$ as
		\begin{align}
			\label{density_hit}
			\rho_{\epsilon}(x,y) +
			 \sum_{n = 2}^{\infty} \int_{(I\setminus(A \cup B))^{n-1}} \!\!\!\!\!\!\!\!\!\!\!\!\!\!\!\!\!\!\!\!\!\!\!\!\!dz_{1} \cdots dz_{n-1} \, \rho_{\epsilon}(x,z_{1}) \rho_{\epsilon}(z_{n-1},y) \prod_{m = 1}^{n-2} \rho_{\epsilon}(z_{m},z_{m + 1})\, ,
		\end{align}
		which is well-defined since $I\setminus(A \cup B)$ is a non-empty union of a finite number of intervals. An analogous result holds for the adjoint process since $\rho_{\epsilon}^{\dagger}(y,\cdot)$ is a probability density function for $y \in I$. 
	\end{proof}
	
	\begin{proof}[Proof of Lemma \ref{lemma_cap}]
		We first note that
		\begin{align*}
			\text{cap}_{\epsilon}(A,B) &= \beta_{\epsilon} \int_{A} \int_{B} p_{\epsilon}(x) \boldsymbol{P}_{x}^{\epsilon}[H(\{y\}) = H^{+}(A \cup B)] \, dy \, dx\\
			&= \beta_{\epsilon} \int_{A} \int_{B} p_{\epsilon}(y) \boldsymbol{P}_{y}^{\epsilon,\dagger}[H^{\dagger}(\{x\}) = H^{\dagger,+}(A \cup B)] \, dy \, dx = \text{cap}_{\epsilon}^{\dagger}(B,A)
		\end{align*}
		in which the first and last equality follow from Lemma \ref{lemma_density} and the definition of capacities \eqref{cap}, and the second equality is obtained by reversing the trajectory of $\xi_{\epsilon}(\cdot)$. This reversal is obtained by multiplying \eqref{density_hit} by $p_{\epsilon}(x)$ and applying \eqref{detail_balance} in sequence for each element in the product. It remains to prove that $\text{cap}_{\epsilon}(A,B) = \text{cap}_{\epsilon}(B,A)$. But it follows from \eqref{cap2} that
		\begin{align*}
			\text{cap}_{\epsilon}(A,B) = D_{\epsilon}(h_{A,B}^{\epsilon}) = D_{\epsilon}(1 - h_{A,B}^{\epsilon}) = D_{\epsilon}(h_{B,A}^{\epsilon}) = \text{cap}_{\epsilon}(B,A).
		\end{align*}
	\end{proof}
	
	For $i \neq j \in S$, define 
	\begin{align}
		\label{theta_adjoint}
		\theta_{\epsilon}^{\dagger}(i,j) = \frac{\beta_{\epsilon}}{\mu_{\epsilon}(\mcE_{i}^{\epsilon})} \int_{\mcE_{i}^{\epsilon}} p_{\epsilon}(x) \boldsymbol{P}_{x}^{\epsilon,\dagger}[H^{\dagger}(\mcE_{j}^{\epsilon}) < H^{\dagger,+}(\check{\mcE}_{j}^{\epsilon})] \ dx.
	\end{align}
	Observe that, for all $i \in S$,
	\begin{align}
		\label{sum_theta}
		\sum_{j \in S\setminus\{i\}} \theta_{\epsilon}^{\dagger}(i,j) = \frac{\beta_{\epsilon}}{\mu_{\epsilon}(\mcE_{i}^{\epsilon})} \int_{\mcE_{i}^{\epsilon}} p_{\epsilon}(x) \boldsymbol{P}_{x}^{\epsilon,\dagger}[H^{\dagger}(\check{\mcE}_{i}^{\epsilon}) < H^{\dagger,+}(\mcE_{i}^{\epsilon})] \ dx
	\end{align}
	and an analogous equality holds for $\sum_{j \in S\setminus\{i\}} \theta_{\epsilon}(i,j)$ recalling the definition of $\theta_{\epsilon}(i,j)$ in \eqref{asym_rate}. The following lemma is an extension of \cite[Lemma~5.1]{landim2021resolvent} to Markov jump processes in the uncountable state space $I$.
	
	\begin{lemma}
		\label{lemma5.1}
		For $i \neq j \in S$,
		\begin{align*}
			\mu_{\epsilon}(\mcE_{i}^{\epsilon}) \, \theta_{\epsilon}(i,j) = \mu_{\epsilon}(\mcE_{j}^{\epsilon}) \, \theta_{\epsilon}^{\dagger}(j,i)
		\end{align*}
		and
		\begin{align*}
			\sum_{j \in S\setminus\{i\}} \theta_{\epsilon}^{\dagger}(i,j) = \sum_{j \in S\setminus\{i\}} \theta_{\epsilon}(i,j) = \frac{1}{\mu_{\epsilon}(\mcE_{i}^{\epsilon})} \text{cap}_{\epsilon}(\mcE_{i}^{\epsilon},\check{\mcE}_{i}^{\epsilon}).
		\end{align*}
	\end{lemma}
	\begin{proof}
		The first assertion follows since
		\begin{align*}
			\mu_{\epsilon}(\mcE_{i}^{\epsilon}) \, \theta_{\epsilon}(i,j) &= \beta_{\epsilon} \int_{\mcE_{i}^{\epsilon}} p_{\epsilon}(x) \boldsymbol{P}_{x}^{\epsilon}[H(\mcE_{j}^{\epsilon}) < H^{+}(\check{\mcE}_{j}^{\epsilon})] \ dx\\
			&= \beta_{\epsilon} \int_{\mcE_{i}^{\epsilon}} \int_{\mcE_{j}^{\epsilon}} p_{\epsilon}(x) \boldsymbol{P}_{x}^{\epsilon}[H(\{y\}) = H^{+}(\mcE^{\epsilon})] \ dy \ dx\\
			&= \beta_{\epsilon} \int_{\mcE_{i}^{\epsilon}} \int_{\mcE_{j}^{\epsilon}} p_{\epsilon}(y) \boldsymbol{P}_{y}^{\epsilon,\dagger}[H^{\dagger}(\{x\}) = H^{\dagger,+}(\mcE^{\epsilon})] \ dy \ dx\\
			&= \mu_{\epsilon}(\mcE_{j}^{\epsilon}) \, \theta_{\epsilon}^{\dagger}(j,i)
		\end{align*}
		in which the second and last equality follow from Lemma \ref{lemma_density}, and the third equality follows by reversing the trajectory of $\xi_{\epsilon}(\cdot)$.
		
		By combining \eqref{cap} and \eqref{sum_theta} we conclude that
		\begin{align*}
			\mu_{\epsilon}(\mcE_{i}^{\epsilon}) \sum_{j \in S\setminus\{i\}} \theta_{\epsilon}^{\dagger}(i,j) = \text{cap}_{\epsilon}^{\dagger}(\mcE_{i}^{\epsilon},\check{\mcE}_{i}^{\epsilon})
		\end{align*}
		and an analogous inequality holds for the process $\xi_{\epsilon}(\cdot)$. The second assertion of the lemma is now direct from Lemma \ref{lemma_cap} since $\mcE_{i}^{\epsilon}$ and $\check{\mcE}_{i}^{\epsilon}$ are a union of a finite number of non-degenerate intervals.
	\end{proof}
	
	We are now in position to prove Proposition \ref{cor_7.3}. We follow closely the proof of Proposition 7.2 in \cite{landim2021resolvent}.
	
	\begin{proof}[Proof of Proposition \ref{cor_7.3}]
		Fix $i \in S$, and denote by $h_{i}^{\epsilon,\dagger} \coloneqq h_{\mcE_{i}^{\epsilon},\check{\mcE}_{i}^{\epsilon}}^{\epsilon,\dagger}$ the equilibrium potential of the adjoint process between $\mcE_{i}^{\epsilon}$ and $\check{\mcE}_{i}^{\epsilon}$. Multiplying both sides of the resolvent equation \eqref{resolvent_eq} by $h_{i}^{\epsilon,\dagger}$ and integrating with respect to the measure $\mu_{\epsilon}$ we have
		\begin{align}
			\label{int_res}
			\lambda \langle F_{\epsilon},h_{i}^{\epsilon,\dagger}\rangle_{\mu_{\epsilon}} - \langle \mcL_{\epsilon} F_{\epsilon},h_{i}^{\epsilon,\dagger}\rangle_{\mu_{\epsilon}} = \langle G,h_{i}^{\epsilon,\dagger}\rangle_{\mu_{\epsilon}}.
		\end{align}
		Since $G$ equals zero in $\Delta^{\epsilon}$ and $g(i)$ in $\mcE_{i}^{\epsilon}$, and $h_{i}^{\epsilon,\dagger}$ equal zero in $\check{\mcE}_{i}^{\epsilon}$ and one in $\mcE_{i}^{\epsilon}$, we conclude that the right-hand side of \eqref{int_res} can be written as
		\begin{align*}
			\langle G,h_{i}^{\epsilon,\dagger}\rangle_{\mu_{\epsilon}} = g(i) \mu_{\epsilon}(\mcE_{i}^{\epsilon}).
		\end{align*}
		Recalling the definition of $f_{\epsilon}(i)$ in \eqref{f_eps}, the first term on the left-hand side of \eqref{int_res} satisfies
		\begin{align*}
			\lambda \, \langle F_{\epsilon},h_{i}^{\epsilon,\dagger}\rangle_{\mu_{\epsilon}} &= \lambda \int_{\mcE_{i}^{\epsilon}} F_{\epsilon}(x)\,  p_{\epsilon}(x)\,  dx + \lambda \int_{\Delta^{\epsilon}} F_{\epsilon}(x)\,  h_{i}^{\epsilon,\dagger}(x)\,  p_{\epsilon}(x)\,  dx \\
			&\leq \lambda\,  \mu_{\epsilon}(\mcE_{i}^{\epsilon})\,  f_{\epsilon}(i) + \lambda \, C \, \mu_{\epsilon}(\Delta^{\epsilon}) 
		\end{align*}
		for some constant $C > 0$ since $\lVert F_{\epsilon} \rVert_{\infty}$ is uniformly bounded on $\epsilon$.
		
		Since $\mcL_{\epsilon}^{\dagger}h_{i}^{\epsilon,\dagger} = 0$ in $\Delta^{\epsilon}$ and $\mcL_{\epsilon}^{\dagger}$ is the adjoint of $\mcL_{\epsilon}$, the second term in the left-hand side of $\eqref{int_res}$ satisfies
		\begin{align*}
			 \langle \mcL_{\epsilon} F_{\epsilon},h_{i}^{\epsilon,\dagger}\rangle_{\mu_{\epsilon}} =  \langle F_{\epsilon},\mcL_{\epsilon}^{\dagger} h_{i}^{\epsilon,\dagger}\rangle_{\mu_{\epsilon}} = \sum_{j \in S} \int_{\mcE_{j}^{\epsilon}} F_{\epsilon}(x) \, (\mcL_{\epsilon}^{\dagger} h_{i}^{\epsilon,\dagger})(x) \, p_{\epsilon}(x) \, dx.
		\end{align*}
		In view of \eqref{def_EP} and \eqref{dir_problem}, since $h_{i}^{\epsilon,\dagger}(x) = 0$ for $x \in \mcE_{j}^{\epsilon}$ with $j \neq i$, in this case it holds
		\begin{align*}
			(\mcL_{\epsilon}^{\dagger} h_{i}^{\epsilon,\dagger})(x) &= \beta_{\epsilon} \int_{I} \rho_{\epsilon}^{\dagger}(x,y) \, h_{i}^{\epsilon,\dagger}(y) \, \chi_{I\setminus\check{\mcE}^{\epsilon}_{i}}(y) \, dy \\
			&=  \beta_{\epsilon} \int_{I} \rho_{\epsilon}^{\dagger}(x,y) \, \boldsymbol{P}_{y}^{\epsilon,\dagger}[H^{\dagger}(\mcE_{i}^{\epsilon}) < H^{\dagger}(\check{\mcE}_{i}^{\epsilon})] \, \chi_{I\setminus\check{\mcE}^{\epsilon}_{i}}(y)\, dy\\
			&=  \beta_{\epsilon} \, \boldsymbol{P}_{x}^{\epsilon,\dagger}[H^{\dagger}(\mcE_{i}^{\epsilon}) < H^{\dagger,+}(\check{\mcE}_{i}^{\epsilon})].
		\end{align*}
		Since $h_{i}^{\epsilon,\dagger}(y) - 1 = -\boldsymbol{P}_{y}^{\epsilon,\dagger}[H^{\dagger}(\check{\mcE}_{i}^{\epsilon}) < H^{\dagger}(\mcE_{i}^{\epsilon})]$, and $h_{i}^{\epsilon,\dagger}(x) = 1$ for $x \in \mcE_{i}^{\epsilon}$, in this case a deduction analogous to that above yields
		\begin{align*}
			(\mcL_{\epsilon}^{\dagger} h_{i}^{\epsilon,\dagger})(x) = - \beta_{\epsilon} \, \boldsymbol{P}_{x}^{\epsilon,\dagger}[H^{\dagger}(\check{\mcE}_{i}^{\epsilon}) < H^{\dagger,+}(\mcE_{i}^{\epsilon})].
		\end{align*}
		We conclude that
		\begin{align*}
			 \langle \mcL_{\epsilon} F_{\epsilon},h_{i}^{\epsilon,\dagger}\rangle_{\mu_{\epsilon}} =& \sum_{j \in S\setminus\{i\}} \beta_{\epsilon} \int_{\mcE_{j}^{\epsilon}} F_{\epsilon}(x) \, \boldsymbol{P}_{x}^{\epsilon,\dagger}[H^{\dagger}(\mcE_{i}^{\epsilon}) < H^{\dagger,+}(\check{\mcE}_{i}^{\epsilon})] \, p_{\epsilon}(x) \, dx \\
			 &- \beta_{\epsilon} \int_{\mcE_{i}^{\epsilon}} F_{\epsilon}(x) \, \boldsymbol{P}_{x}^{\epsilon,\dagger}[H^{\dagger}(\check{\mcE}_{i}^{\epsilon}) < H^{\dagger,+}(\mcE_{i}^{\epsilon})] \, p_{\epsilon}(x) \, dx.
		\end{align*}
		Adding and subtracting $f_{\epsilon}$ to the elements in the sum above, in view of \eqref{theta_adjoint} and \eqref{sum_theta}, the right-hand side can be rewritten as
		\begin{align}
			\label{inn0}
			\sum_{j \in S\setminus\{i\}} \mu_{\epsilon}(\mcE_{j}^{\epsilon}) \, f_{\epsilon}(j) \, \theta_{\epsilon}^{\dagger}(j,i) - \mu_{\epsilon}(\mcE_{i}^{\epsilon}) \, f_{\epsilon}(i) \sum_{j \in S\setminus\{i\}} \theta_{\epsilon}^{\dagger}(i,j) + R_{\epsilon}
		\end{align}
		in which $R_{\epsilon}$ is bounded from above by
		\begin{align}
			\label{inn1} \nonumber
			&\max\limits_{j \in S} \sup\limits_{x \in \mcE_{j}^{\epsilon}} \left|F_{\epsilon}(x) - f_{\epsilon}(j)\right| \left\{\sum_{j \in S\setminus\{i\}} \mu_{\epsilon}(\mcE_{j}^{\epsilon}) \, \theta_{\epsilon}^{\dagger}(j,i) + \mu_{\epsilon}(\mcE_{i}^{\epsilon}) \sum_{j \in S\setminus\{i\}} \theta_{\epsilon}^{\dagger}(i,j)\right\}\\
			&= 2 \, \mu_{\epsilon}(\mcE_{i}^{\epsilon})\max\limits_{j \in S} \sup\limits_{x \in \mcE_{j}^{\epsilon}} \left|F_{\epsilon}(x) - f_{\epsilon}(j)\right| \sum_{j \in S\setminus\{i\}} \theta_{\epsilon}(i,j).
		\end{align}
		The last equality in \eqref{inn1} is due to Lemma \ref{lemma5.1}. Moreover, due to Lemma \ref{lemma5.1}, the sum of the first two terms in \eqref{inn0} is equal to 
		\begin{align*}
			\mu_{\epsilon}(\mcE_{i}^{\epsilon}) \sum_{j \in S\setminus\{i\}} \theta_{\epsilon}(i,j)[f_{\epsilon}(j) - f_{\epsilon}(i)].
		\end{align*}
		We have proved so far that the identity \eqref{int_res} may be rewritten as
		\begin{align*}
			\lambda f_{\epsilon}(i) - \sum_{j \in S\setminus\{i\}} \theta_{\epsilon}(i,j)[f_{\epsilon}(j) - f_{\epsilon}(i)] = g(i) + R_{\epsilon}^{(2)}
		\end{align*}
		where the absolute value of $R_{\epsilon}^{(2)}$ is bounded from above by
		\begin{align*}
			C\frac{\mu_{\epsilon}(\Delta^{\epsilon})}{\mu_{\epsilon}(\mcE_{i}^{\epsilon})} + 2 \, \max\limits_{j \in S} \sup\limits_{x \in \mcE_{j}^{\epsilon}} \left|F_{\epsilon}(x) - f_{\epsilon}(j)\right| \sum_{j \in S\setminus\{i\}} \theta_{\epsilon}(i,j)
		\end{align*}
        for a constant $C > 0$. The result follows by $\mfR^{(1)}, \textbf{(A5)},$ $\textbf{(H0)}$ and the uniqueness of the solution of the reduced resolvent equation $(\lambda - \mcL)f = g$.
	\end{proof}

    \begin{remark}
        The proofs above rely heavily on the fact that both $\mcE^{\eps}$ and $\Delta^{\eps}$ are the union of a finite number of non-generate intervals. If these were subsets with Lebesgue measure zero, then alternative proofs would be necessary.
    \end{remark}
	
	\subsection{Proof of Proposition \ref{theorem_dTV}}

    In this section, $C$ is a constant, which may change from line to line, that does not depend on $\eps$, but may depend on quantities in the assumptions \textbf{(A1)}-\textbf{(A7)} (e.g. $\Lambda_{i}$, $\lambda_{i}$, $\kappa$, $\dots$).

    We start by estimating $d_{TV}(\mu_{\eps},\tilde{\mu}_{\eps})$ in which $\tilde{\mu}_{\epsilon} \coloneqq \sum_{i \in S} \mu_{\eps}(I_{i}) \, \tilde{\mu}_{\epsilon}^{i}$, recalling that $\tilde{\mu}_{\eps}^{i}$ is the ACIM of the restricted process in condition $\mathfrak{M}$. Since $d_{TV}(\mu_{\eps},\tilde{\mu}_{\eps})$ is equal to
    \begin{align}
        \label{en_dtv} 
        \sup_{A \in \mfB_{I}} \left|\sum_{i \in S} \mu_{\eps}(A \cap I_{i}) - \sum_{i \in S} \tilde{\mu}_{\eps}(A \cap I_{i})\right|
        \leq \sum_{i \in S} \sup_{B \in \mfB_{I_{i}}} \left|\mu_{\eps}(B) - \tilde{\mu}_{\eps}(B)\right|
    \end{align}
    we can uniformly bound $\left|\mu_{\eps}(B) - \tilde{\mu}_{\eps}(B)\right|$ for $B \in \mfB_{I_{i}}$ separately for each $i \in S$. Recall the definition of $\textbf{h}_{\eps}$ in \eqref{33}, and observe that $\left|\mu_{\eps}(B) - \tilde{\mu}_{\eps}(B)\right|$ is bounded by
    \begin{align}
        \label{finiteT}
        \left|\mu_{\eps}(B) - \mu_{\eps}(I_{i}) \boldsymbol{P}_{\mu_{\eps}^{i}}^{\eps}[\tilde{\xi}_{\eps}^{i}(\textbf{h}_{\eps}) \in B]\right| + \left|\mu_{\eps}(I_{i}) \boldsymbol{P}_{\mu_{\eps}^{i}}^{\eps}[\tilde{\xi}_{\eps}^{i}(\textbf{h}_{\eps}) \in B] - \tilde{\mu}_{\eps}(B)\right|
    \end{align}
    in which
    \begin{equation*}
        \mu_{\eps}^{i}(B) \coloneqq \frac{\mu_{\eps}(B)}{\mu_{\eps}(I_{i})}, \, \, \, B \in \mfB_{I_{i}}
    \end{equation*}
    is the measure $\mu_{\eps}$ conditioned on $I_{i}$. We bound the first term in \eqref{finiteT} by adapting the generator method and the second one as in the proof of Lemma \ref{lemma_Mix_nt} (cf. \eqref{in_mix_dTV}).
    
    Recalling from \eqref{generator_restricted} that $\tilde{\mcL}_{\eps}^{i}$ is the generator of $\tilde{\xi}_{\eps}^{i}(\cdot)$ that acts on functions $F: I_{i} \mapsto \mbR$, consider the Markov process $\tilde{\xi}_{\eps}(\cdot)$ with generator
    \begin{equation*}
        \tilde{\mcL}_{\eps}F(x) = \sum_{i \in S} \tilde{\mcL}_{\eps}^{i}F(x) \, \chi_{I_{i}}(x), \, \, \, x \in I
    \end{equation*}
    that acts on functions $F: I \mapsto \mbR$. This is the process that, starting from $x \in I_{i}$, behaves exactly as $\tilde{\xi}_{\eps}^{i}(\cdot)$. We note that $\tilde{\mu}_{\epsilon}$ is an invariant measure of this process.
    
    For $A \in \mfB_{I}$, consider the equation
    \begin{equation}
        \label{CS_eq}
        \tilde{\mcL}_{\eps}F_{A}^{\eps}(x) = \sum_{i \in S} \mu_{\eps}(I_{i}) \left[\boldsymbol{P}_{x}^{\eps}[\tilde{\xi}_{\eps}^{i}(\textbf{h}_{\eps}) \in A] - \chi_{A}(x)\right] \chi_{I_{i}}(x), \, \, \, x \in I.
    \end{equation}
    If $B \in \mfB_{I_{i}}$ for $i \in S$, it is well-know (see for example \cite[Proposition~1.1.5]{ethier2009markov}) that the solution of \eqref{CS_eq} is
    \begin{equation*}
        F_{B}^{\eps}(x) = \sum_{i \in S} \mu_{\eps}(I_{i}) \, \chi_{I_{i}}(x) \, \int_{0}^{\textbf{h}_{\eps}} \boldsymbol{P}_{x}^{\eps}[\tilde{\xi}_{\eps}^{i}(t) \in B] \, dt
    \end{equation*}
    for $x \in I$. In particular, $F_{B}^{\eps}(x) = 0$ if $x \notin I_{i}$. From now on, fix $i \in S$ and $B \in \mfB_{I_{i}}$.

    We proceed as in the generator method. Taking the expectation under $\mu_{\eps}^{i}$ and then absolute values on both sides of \eqref{CS_eq} for $B \in \mfB_{I_{i}}$ yield
    \begin{equation}
        \label{genM}
        |\mu_{\eps}^{i}[\tilde{\mcL}_{\eps}F_{B}^{\eps}]| = |\mu_{\eps}(I_{i}) \boldsymbol{P}_{\mu_{\eps}^{i}}^{\eps}[\tilde{\xi}_{\eps}^{i}(\textbf{h}_{\eps}) \in B] - \mu_{\eps}(B)|,
    \end{equation}
    so we can bound the first term in \eqref{finiteT} by bounding the left-hand side of \eqref{genM}. Consider a coupling between $\xi_{\epsilon}(\cdot)$ and $\tilde{\xi}_{\epsilon}(\cdot)$ such that $\xi_{\epsilon}(t,\omega) = \tilde{\xi}_{\epsilon}(t,\omega)$ for all $t < H(I_{i_{0}}^{c})(\omega)$ and $\omega \in \Omega$ in which $i_{0} \in S$ is such that $\xi_{\eps}(0) \in I_{i_{0}}$, that exists by definition of the restricted processes $\tilde{\xi}_{\epsilon}^{i}(\cdot)$. We observe that for $x \in I$
    \begin{align*}
        \tilde{\mcL}_{\eps}F_{B}^{\eps}(x) = \sum_{i \in S} \beta_{\eps} \, \mu_{\eps}(I_{i}) \, \chi_{I_{i}}(x) \, \int_{0}^{\textbf{h}_{\eps}} \boldsymbol{E}_{x}^{\eps}[\boldsymbol{P}_{\tilde{X}_{1}^{\eps}}^{\eps}[\tilde{\xi}^{i}_{\eps}(t) \in B]] - \boldsymbol{P}_{x}^{\eps}[\tilde{\xi}^{i}_{\eps}(t) \in B] \, dt.
    \end{align*}
    Taking expectations on $\mu^{i}_{\eps}$ we obtain        
    \begin{align}
        \label{Ss0} \nonumber
        \mu_{\eps}^{i}[\tilde{\mcL}_{\eps}F_{B}^{\eps}] &= \boldsymbol{E}^{\epsilon}_{\mu_{\eps}^{i}}[\tilde{\mcL}_{\epsilon}F_{B}^{\eps}(\tilde{\xi}_{\epsilon}(0))] = \boldsymbol{E}^{\epsilon}_{\mu_{\eps}}[\mu_{\eps}(I_{i})^{-1} \, \tilde{\mcL}_{\epsilon}F_{B}^{\eps}(\tilde{\xi}_{\epsilon}(0)),\tilde{\xi}_{\epsilon}(0) \in I_{i}]\\
        &= \boldsymbol{E}^{\epsilon}_{\mu_{\eps}}[\mu_{\eps}(I_{i})^{-1} \, \tilde{\mcL}_{\epsilon}F_{B}^{\eps}(\tilde{\xi}_{\epsilon}(0))]
    \end{align}
    in which the last equality holds since $\tilde{\mcL}_{\epsilon}F_{B}(x) = 0$ for $x \notin I_{i}$.

    Since $\lVert \tilde{\mcL}_{\epsilon}F_{B}^{\eps} \rVert_{\infty} \leq 2$ by \eqref{CS_eq}, we have that the absolute value of \eqref{Ss0} is bounded by
    \begin{align}
        \label{Ss1}
        \left|\boldsymbol{E}^{\epsilon}_{\mu_{\eps}}[\mu_{\eps}(I_{i})^{-1} \tilde{\mcL}_{\epsilon}F_{B}^{\eps}(\tilde{\xi}_{\epsilon}(0)),H(I_{i_{0}}^{c}) > \tau_{1}]\right| + \frac{2}{\mu_{\eps}(I_{i})} \ \boldsymbol{P}_{\mu_{\epsilon}}^{\eps}\left[H(I_{i_{0}}^{c}) = \tau_{1}\right]
    \end{align}
    in which the expectation is to be understood as over $\tilde{\xi}_{\epsilon}(0,\omega)$ for $\omega$ such that $H(I_{i_{o}}^{c})(\omega) > \tau_{1}(\omega)$ where $\tau_{1}$ is the time of the first jump of $\xi_{\eps}(\cdot)$. Recalling the definition of $q_{\eps}$ in \eqref{qeps}, we have
    \begin{align}
        \label{Ss1.1}
        \frac{2}{\mu_{\eps}(I_{i})} \ \boldsymbol{P}_{\mu_{\epsilon}}^{\eps}\left[H(I_{i_{0}}^{c}) = \tau_{1}\right] \leq C \, q_{\eps}
    \end{align}
    for a constant $C$ that depends on $\lim_{\eps \to 0} \mu_{\eps}(I_{i})$, which is greater than zero by \textbf{(A5)}.
    
    Define $\tilde{\mcL}_{\epsilon}F_{B}^{\eps}(\xi_{\epsilon}(0))$ as
    \begin{align}
        \label{Ss2}
         \beta_{\eps} \ \mu_{\eps}(I_{i}) \int_{0}^{\textbf{h}_{\eps}} \boldsymbol{E}_{\xi_{\epsilon}(0)}^{\eps}[\boldsymbol{P}_{X_{1}^{\eps}}^{\eps}[\tilde{\xi}_{\eps}(t) \in B]] - \boldsymbol{P}_{\xi_{\epsilon}(0)}^{\eps}[\tilde{\xi}_{\eps}(t) \in B] \, dt.
    \end{align}
    As $\mu_{\epsilon}$ is the invariant measure of $\xi_{\epsilon}(\cdot)$, 
    \begin{equation*}
        \boldsymbol{P}_{\mu_{\epsilon}}^{\epsilon}[\tilde{\xi}_{\epsilon}(t) \in B] = \boldsymbol{E}_{\mu_{\epsilon}}^{\epsilon}[\boldsymbol{P}_{X_{1}^{\eps}}^{\epsilon}[\tilde{\xi}_{\epsilon}(t) \in B]]
    \end{equation*}
    for all $t > 0$ and hence
    \begin{align*}
        \boldsymbol{E}^{\epsilon}_{\mu_{\eps}}[\mcL_{\epsilon}F_{B}^{\eps}(\xi_{\epsilon}(0))] = \beta_{\eps} \ \mu_{\eps}(I_{i}) \int_{0}^{\textbf{h}_{\eps}} \boldsymbol{E}_{\mu_{\epsilon}}^{\epsilon}[\boldsymbol{P}^{\epsilon}_{X_{1}^{\eps}}[\tilde{\xi}_{\epsilon}(t) \in B]] - \boldsymbol{P}_{\mu_{\epsilon}}^{\epsilon}[\tilde{\xi}_{\epsilon}(t) \in B] \, dt = 0.
    \end{align*}
    It follows from \eqref{Ss2} that
    \begin{align}
        \label{Ss2.1}
        |\tilde{\mcL}_{\epsilon}F_{B}^{\eps}(\xi_{\epsilon}(0))| \leq 2 \, \beta_{\eps} \, \mu_{\eps}(I_{i}) \, \textbf{h}_{\eps}
    \end{align}
    and, by the coupling between $\xi_{\eps}(\cdot)$ and $\tilde{\xi}_{\eps}(\cdot)$, we can conclude that
    \begin{align}
        \label{Ss3} \nonumber
        \Big|\boldsymbol{E}^{\epsilon}_{\mu_{\eps}}&[\mu_{\eps}(I_{i})^{-1} \tilde{\mcL}_{\epsilon}F_{B}^{\eps}(\tilde{\xi}_{\epsilon}(0)),H(I_{i_{0}}^{c}) > \tau_{1}]\Big| \\ \nonumber
        &= \left|\boldsymbol{E}^{\epsilon}_{\mu_{\eps}}[\mu_{\eps}(I_{i})^{-1} \tilde{\mcL}_{\epsilon}F_{B}^{\eps}(\xi_{\epsilon}(0)),H(I_{i_{0}}^{c}) > \tau_{1}]\right| \\ \nonumber
        &= \left|\boldsymbol{E}^{\epsilon}_{\mu_{\eps}}[\mu_{\eps}(I_{i})^{-1} \tilde{\mcL}_{\epsilon}F_{B}^{\eps}(\xi_{\epsilon}(0))] - \boldsymbol{E}^{\epsilon}_{\mu_{\eps}}[\mu_{\eps}(I_{i})^{-1} \tilde{\mcL}_{\epsilon}F_{B}^{\eps}(\xi_{\epsilon}(0)),H(I_{i_{0}}^{c}) = \tau_{1}]\right| \\
        &\leq 2 \, q_{\eps} \, \beta_{\eps} \, \textbf{h}_{\eps}
    \end{align}
    in which the first equality follows since, in the event $\{H(I_{i_{0}}^{c}) > \tau_{1}\}$, $\tilde{\mcL}_{\epsilon}F_{B}^{\eps}(\tilde{\xi}_{\epsilon}(0)) = \tilde{\mcL}_{\epsilon}F_{B}^{\eps}(\xi_{\epsilon}(0))$ as $X_{1}^{\eps} = \tilde{X}_{1}^{\eps}$, and the inequality is due to \eqref{Ss1.1} and \eqref{Ss2.1}. Recalling the definition $\textbf{h}_{\eps} = a_{\eps}/\beta_{\eps}$ in \eqref{he}, by combining \eqref{genM}-\eqref{Ss3} we conclude that
    \begin{equation}
        \label{Ss4}
        \left|\mu_{\eps}(I_{i}) \boldsymbol{P}_{\mu_{\eps}^{i}}^{\eps}[\tilde{\xi}_{\eps}^{i}(\textbf{h}_{\eps}) \in B] - \mu_{\eps}(B)\right| \leq C \, q_{\eps} \, a_{\eps}.
    \end{equation}
    
    The second term in \eqref{finiteT} is equal to
    \begin{align}
        \label{Ss5} \nonumber
        \mu_{\eps}(I_{i})\left|\boldsymbol{P}_{\mu_{\eps}^{i}}^{\eps}[\tilde{\xi}_{\eps}^{i}(\textbf{h}_{\eps}) \in B] - \tilde{\mu}_{\eps}^{i}(B)\right| &\leq \sup_{x \in I_{i}} d^i_{\rm TV} (\delta_x \tilde{\ms P}^{i}_\epsilon(\textbf{h}_{\eps}) \,,\,  \tilde{\mu}_{\epsilon}^{i})\\ \nonumber
        &\leq \sup_{x \in I_{i}} d^i_{\rm TV} (\delta_x \tilde{\ms P}^{i}_\epsilon(\tilde{\tau}_{\lfloor \beta_{\eps}\textbf{h}_{\eps}/2 \rfloor}) \,,\,  \tilde{\mu}_{\epsilon}^{i}) + 2 \, \left(\frac{2}{e}\right)^{\frac{\beta_{\eps}\textbf{h}_{\eps}}{2}}\\
        &\leq C \, \left[\lambda_{i} \vee \left(\frac{2}{e}\right)\right]^{\frac{a_{\eps}}{2}}
    \end{align}
    in which the second inequality follows from \eqref{dTV_tn} and \eqref{chernoff} by taking $t = \textbf{h}_{\eps}$ there, and the third is due to \eqref{bound_dTV_BV}, Theorem \ref{main_theorem} and the definition of $\textbf{h}_{\eps}$. Denote $\lambda^{\star} = \max_{i \in S} \lambda_{i} \vee 2/e$ and recall from \eqref{he} that $\{a_{\eps}\}$ is any sequence satisfying $1 \ll a_{\eps} \ll q_{\eps}^{-1} \wedge \beta_{\eps}$. Taking $a_{\eps} = 2 \, \log_{\lambda^{\star}} (q_{\eps} \vee \beta_{\eps}^{-1})$ in \eqref{Ss4} and \eqref{Ss5}, and substituting in \eqref{en_dtv}, we conclude that
    \begin{align*}
        d_{TV}(\mu_{\eps},\tilde{\mu}_{\eps}) \leq C \, q_{\eps} \vee \beta_{\eps}^{-1} \, \log(q_{\eps}^{-1} \wedge \beta_{\eps}).
    \end{align*}    
    The result follows by the triangular inequality
    \begin{align*}
        d_{TV}(\mu_{\epsilon},\mu) \leq d_{TV}(\mu_{\epsilon},\tilde{\mu}_{\epsilon}) +  d_{TV}(\tilde{\mu}_{\epsilon},\mu)
    \end{align*}
    and the following proposition.

    \begin{proposition}
        \label{lemma_q_eps}
        If $\xi_{\eps}(\cdot)$ is $\mathcal{L}$-metastable, then there exists a constant $C > 0$ such that $q_{\eps}^{-1} \leq C \, \beta_{\eps}$.
    \end{proposition}
    \begin{proof}
        Fix $i \in S$ and recall that $\tau(I_{i}^{c})$ is the hitting time of $I_{i}^{c}$ by the embedded Markov chain $X_{n}^{\eps}$. For $n \geq 0$ and $x \in \mathscr{E}_{i}^{\eps}$, by the strong Markov property,
        \begin{align*}
            \boldsymbol{P}_{x}^{\eps}[\tau(I_{i}^{c}) > n] = \boldsymbol{P}_{x}^{\eps}[X_{l}^{\eps} \in I_{i}: l = 1,\dots,n] \geq (1 - q_{\eps})^{n}
        \end{align*}
        and hence
        \begin{align*}
            \boldsymbol{E}_{x}^{\eps}[\tau(\check{\mathscr{E}}_{i}^{\eps})] \geq \boldsymbol{E}_{x}^{\eps}[\tau(I_{i}^{c})] = \sum_{n = 0}^{\infty} \boldsymbol{P}_{x}^{\eps}[\tau(I_{i}^{c}) > n] \geq q_{\eps}^{-1}
        \end{align*}
        in which the first inequality follows since $\check{\mathscr{E}}_{i}^{\eps} \subset I_{i}^{c}$. Denote $\tau \coloneqq \tau(\check{\mathscr{E}}_{i}^{\eps})$ to ease notation.
        
        As in the proof of Lemma \ref{lemma_discrete_cont}, let $Z_{\tau}$ be a random variable with distribution $Gamma(\tau,1)$ and observe that $H(\check{\mathscr{E}}_{i}^{\eps}) = Z_{\tau} \, \beta_{\eps}^{-1}$ since it is the hitting time of $\check{\mathscr{E}}_{i}^{\eps}$ by the speeded-up process. In particular,
        \begin{align*}
            \boldsymbol{E}_{x}^{\eps}[H(\check{\mathscr{E}}_{i}^{\eps})] = \beta_{\eps}^{-1} \boldsymbol{E}_{x}^{\eps}[\boldsymbol{E}_{x}^{\eps}[ Z_{\tau}|\tau]] = \beta_{\eps}^{-1} \, \boldsymbol{E}_{x}^{\eps}[\tau(\check{\mathscr{E}}_{i}^{\eps})]
        \end{align*}
        since $\boldsymbol{E}_{x}^{\eps}[ Z_{\tau}|\tau] = \tau$. Condition $\mathfrak{C}^{\mcL}$ implies that
        \begin{align*}
            0 < \lim\limits_{\eps \to 0} \inf_{x \in \mathscr{E}_{i}^{\eps}} \boldsymbol{E}_{x}^{\eps}[H(\check{\mathscr{E}}_{i}^{\eps})] =  \lim\limits_{\eps \to 0} \sup_{x \in \mathscr{E}_{i}^{\eps}} \boldsymbol{E}_{x}^{\eps}[H(\check{\mathscr{E}}_{i}^{\eps})] < \infty
        \end{align*}
        so there exists a constant $C > 0$ such that
        \begin{align*}
            \frac{q_{\eps}^{-1}}{\beta_{\eps}} \leq \frac{\boldsymbol{E}_{x}^{\eps}[\tau(\check{\mathscr{E}}_{i}^{\eps})]}{\beta_{\eps}} = \boldsymbol{E}_{x}^{\eps}[H(\check{\mathscr{E}}_{i}^{\eps})] < C.
        \end{align*}
    \end{proof}

    \subsection{Auxiliary results of Sections \ref{Sec_applications} and \ref{SecEx2}}
    \label{SecAux}

    Let $T$ be a map that satisfies \textbf{(A1)} and \textbf{(A6)}, and recall that each component of $I$ can be decomposed as $I_{i} = \bigcup_{j=1}^{\kappa_{i}} I_{i,j}$ with $\kappa_{i} < \infty$, where $I_{i,j}$ are intervals in which $T$ is continuous and one-to-one. Consider the Markov chain $X_{n}^{\eps}$ obtained perturbing $T$ by additive noise as in \eqref{additive_noise} in which the noise $\sigma_{\eps}^{x}$ has support $[-\eps_{1},\eps_{2}]$ with $\eps_{1},\eps_{2} > 0$ for all $x \in I$. For $A,B \in \mathfrak{B}_{I}$, we denote by
    \begin{equation*}
        A \oplus B = \{x + y: x \in A, y \in B\}
    \end{equation*}
    the Minkowski addition of $A$ and $B$. Observe that it must hold
    \begin{equation*}
        T(I) \oplus [-\eps_{1},\eps_{2}] \subset I
    \end{equation*}
    for the Markov process to be well-defined.

    We will show that if the derivative of $T$ is greater or equal to $2$ on $I_{ij}$ and the branches of $T$ in $I_{i}$ have the same image, then $X_{n}^{\eps}$ has a unique ACIM.

    \begin{proposition}
        \label{prop_ACIM_expand}
        Let $X_{n}^{\eps}$ be the Markov chain induced by a map $T$ and an additive noise $\sigma_{\eps}^{x}$ with support $[-\eps_{1},\eps_{2}]$ for $\eps_{1},\eps_{2} > 0$ and all $x \in I$. If $\inf_{x \in I_{ij}} |T^{\prime}(x)| \geq 2$ and $T(I_{ij}) = T(I_{ij^{\prime}})$ for all $i \in S$ and $j,j^{\prime} = 1,\dots,\kappa_{i}$, then $X_{n}^{\eps}$ has a unique ACIM.
    \end{proposition}

    The proof of Proposition \ref{prop_ACIM_expand} is a consequence of Proposition \ref{prob_cone_ACIM} and the following lemma.

    \begin{lemma}
        \label{lemma_expand}
        Let $T$ be a map satisfying the assumptions of Proposition \ref{prop_ACIM_expand}. If $J \subset T(I_{i}), i \in S,$ is an interval, then $\max_{j} \text{Leb}(T(J \cap I_{ij})) \geq \text{Leb}(J)$.
    \end{lemma}
    \begin{proof}
        We divide the proof into three cases, depending on how many of the intervals $I_{i,j}$ are intersected by the interval $J \subset T(I_{i})$. If it intersects three or more intervals, then there exists $I_{ij}$ with $I_{ij} \subset J$. In this case $T(J \cap I_{ij}) = T(I_{ij}) = T(I_{i})$, in which the second equality holds since the image of all sub-intervals $I_{ij}$ are equal, and the result is direct. If $J$ intersects only one interval, then there exists $I_{ij}$ with $J \subset I_{ij}$ and $\text{Leb}(T(J \cap I_{ij})) = \text{Leb}(T(J)) \geq 2 \, \text{Leb}(J)$ since the absolute value of the derivative of $T$ is greater or equal to $2$ on $I_{ij}$.
        
        It remains the case in which $J$ intersects two intervals, say, $I_{ij}$ and $I_{ij^{\prime}}$. Denote by $J_{1} = J \cap I_{ij}$ and $J_{2} = J \cap I_{ij^{\prime}}$ assuming that $\text{Leb}(J_{1}) \geq \text{Leb}(J_{2})$. Then
        \begin{align*}
            \text{Leb}(T(J \cap I_{ij}))= \text{Leb}(T(J_{1})) \geq 2 \, \text{Leb}(J_{1}) \geq \text{Leb}(J)
        \end{align*}
        in which the first inequality follows from the second case since $J_{1} \subset I_{ij}$.
    \end{proof}

    \begin{proof}[Proof of Proposition \ref{prop_ACIM_expand}]
        We will apply Proposition \ref{prob_cone_ACIM} with
        \begin{align*}
            A_{\eps} = T(I) \oplus [-\eps_{1},\eps_{2}].
        \end{align*}
        Since the support of $\rho_{\eps}(x,\cdot)$ is equal to $T(x) \oplus [-\eps_{1},\eps_{2}]$ for all $x \in I$, it holds $\rho_{\eps}(x,y) = \rho_{\eps}(x,y) \, \chi_{A_{\eps}}(y)$, so it is enough to prove \eqref{cond_unique_ACIM}. Fix $i \in S$ and $x \in I_{i}$. For $n \geq 1$, let $J_{n}$ be the greatest interval contained in $supp \, \rho_{\eps}^{n}(x,\cdot) \cap T(I_{i})$. We argue that, if $J_{n+1} \neq T(I_{i})$, then $\text{Leb}(J_{n+1}) \geq \text{Leb}(J_{n}) + \eps_{1} \wedge \eps_{2}$. Indeed,
        \begin{align}
            \label{ineq_Leb} \nonumber
            \text{Leb}(J_{n+1}) &\geq \max_{j} \, \text{Leb}((T(J_{n} \cap I_{ij}) \oplus [-\eps_{1},\eps_{2}]) \cap T(I_{i}))\\ \nonumber
            &\geq \max_{j} \, \text{Leb}(T(J_{n} \cap I_{ij})) + \eps_{1} \wedge \eps_{2}\\
            &\geq \text{Leb}(J_{n}) + \eps_{1} \wedge \eps_{2}
        \end{align}
        in which the first inequality holds since $(T(J_{n} \cap I_{ij}) \oplus [-\eps_{1},\eps_{2}]) \cap T(I_{i})$ is an interval contained in $supp \, \rho_{\eps}^{n+1}(x,\cdot) \cap T(I_{i})$, and the third inequality is due to Lemma \ref{lemma_expand}. Now, the second inequality holds since, if $J_{n+1} \neq T(I_{i})$, then when the interval $T(J_{n} \cap I_{ij})$ is summed with $[-\eps_{1},\eps_{2}]$, it increases by $\eps_{1}$ on the left and $\eps_{2}$ on the right, but it still does not contain $T(I_{i})$. This implies that the Lebesgue measure of $(T(J_{n} \cap I_{ij}) \oplus [-\eps_{1},\eps_{2}]) \cap T(I_{i})$ is greater than that of $T(J_{n} \cap I_{ij})$ by at least $\eps_{1} \wedge \eps_{2}$.
        
        Taking $n_{\eps}^{i} = \lceil \text{Leb}(T(I_{i}))/\eps_{1} \wedge \eps_{2} \rceil$, it follows by iterating \eqref{ineq_Leb} that $J_{n_{\eps}^{i}} = T(I_{i})$ so
        \begin{align}
            \label{bounded_below}
            \inf_{x \in I_{i}} \inf_{y \in A_{\eps}^{i}} \rho_{\eps}^{n_{\eps}^{i} + 1}(x,y) > 0
        \end{align}
        in which $A_{\eps}^{i} = T(I_{i}) \oplus[-\eps_{1},\eps_{2}]$. Since $A_{\eps}^{i} \cap I_{i}^{c} \neq \emptyset$, it follows from \textbf{(A1.4)} that \eqref{cond_unique_ACIM} holds with
        \begin{equation*}
            n_{\eps} = \kappa + \sum_{i = 1}^{\kappa} n_{\eps}^{i}
        \end{equation*}
        and the proof is complete.
    \end{proof}

    \begin{remark}
        The result in Proposition \ref{prop_ACIM_expand} remains true when $\sigma_{x}^{\eps}$ has support $[-\eps_{1},\eps_{2}]$ for some $x \in I$, and support $[-\eps_{2},\eps_{1}]$ for other values of $x \in I$, which are the cases treated in Sections \ref{Sec_applications} and \ref{SecEx2}.
    \end{remark}

    \begin{remark}
        \label{rem_restricted}
        The proof of Proposition \ref{prop_ACIM_expand} can be easily extended for restricted maps as those in Figures \ref{f3} and \ref{f5} since it analyses the support of $\rho_{\eps}(x,\cdot)$ restricted to $I_{i}$, hence applies to restricted maps. For instance, if \eqref{bounded_below} holds, then an analogous inequality holds for a restricted process with the same $n_{\eps}^{i}$, but taking the infimum over $y \in (T(I_{i}) \oplus[-\eps_{1},\eps_{2}]) \cap I_{i}$. 
    \end{remark}

    \begin{remark}
        \label{rem_mid_branch}
        The proof of Proposition \ref{prop_ACIM_expand} could also be extended to cover the case of the middle component of the map in Figure \ref{f4} in which the image of the branches are not equal. The equality of images $T(I_{ij})$ is used in Lemma \ref{lemma_expand} to treat the case in which $J$ intersects three intervals, to conclude that $T(J) = T(I_{i})$. But in the particular case in Figure \ref{f4}, if $J$ intersects the three intervals of the component $I_{2}$, then clearly $T(J) = T(I_{2})$, so Lemma \ref{lemma_expand} remains true, and so do Proposition \ref{prop_ACIM_expand} and the two remarks above. In particular, the respective Markov chain has a unique ACIM.
    \end{remark}

    \section{Acknowledgements}
    
    The authors thank the Mathematical Sciences Institute and France-Australia Mathematical Sciences and Interactions ANU-CNRS International Research Lab at The Australian National University where this work was carried out.
    
    \bibliographystyle{plain}
	\bibliography{Ref}
\end{document}